\documentclass[12pt,reqno]{amsart}
\usepackage{graphicx}
\usepackage{hyperref}
\usepackage{amssymb}
\usepackage{amsthm}
\usepackage{tikz}
\usetikzlibrary{intersections}
\usetikzlibrary{positioning}
\usetikzlibrary{graphs}

\usepackage{verbatim}
\usepackage{empheq} 
\usepackage{bbm}
\usepackage{array}
\usepackage{hyperref}
\hypersetup{
    colorlinks=true,
    linkcolor=blue,
    filecolor=magenta,
    urlcolor=cyan,
    citecolor=blue,
    pdftitle={Sharelatex Example},
}
\newcommand{\ol}{\overline}

\newcommand{\RN}[1]{%
  \textup{\uppercase\expandafter{\romannumeral#1}}%
}
\newtheorem{thm}{Theorem}[section]
\newtheorem{thmx}{Theorem}[section]

\newtheorem{defn}{Definition}[section]
\newtheorem{lem}{Lemma}[section]
\newtheorem{cor}{Corollary}[section]
\newtheorem{prop}{Proposition}[section]
\newtheorem{ex}{Example}[section]
\newtheorem{rem}{Remark}[section]

\numberwithin{equation}{section}

\def\C{\mathbb{C}}
\def\E{\mathbb{E}}

\def\N{\mathbb{N}}
\def\P{\mathbb{P}}

\def\R{\mathbb{R}}

\def\DD{\mathcal{D}}
\def\EE{\mathcal{E}}
\def\HH{\mathcal{H}}
\def\II{\mathcal{I}}

\def\LL{\mathcal{L}}

\def\NN{\mathcal{N}}

\def\UU{\mathcal{U}}
\def\PP{\mathcal{P}}

\def\andd{\,\,\text{and}\,\,}

\def\range{{\rm ran}}

\def\dist{{\rm dist}}

\def\lan{\langle}
\def\ran{\rangle}

\def\pa{\partial}
\def\stst{\subset\subset}
\def\md{\mathrm{d}}
\def\sm{\setminus}

\def\al{\alpha}
\def\be{\beta}
\def\ep{\epsilon}
\def\ka{\kappa}
\def\tha{\theta}
\def\la{\lambda}
\def\La{\Lambda}
\def\si{\sigma}
\def\Si{\Sigma}

\def\Om{\Omega}
\def\de{\delta}
\def\De{\Delta}
\def\ga{\gamma}
\def\Ga{\Gamma}
\def\vp{\varphi}
\begin{document}
\title[Quasi-stationary distributions]{Quasi-stationary distributions of multi-dimensional diffusion processes}

\author{Alexandru Hening}
\address{Department of Mathematics, Tufts University, Bromfield-Pearson Hall, 503 Boston Avenue, Medford, MA 02155, United States}
\email{alexandru.hening@tufts.edu}
\thanks{A.H. was supported by NSF through the grant DMS 1853463. W.Q. was partially supported by a postdoctoral fellowship from the University of Alberta. Z.S. was partially supported by a start-up grant from the University of Alberta, NSERC RGPIN-2018-04371 and NSERC DGECR-2018-00353.  Y.Y. was partially supported by NSERC RGPIN-2020-04451, PIMS CRG grant, a faculty development grant from the University of Alberta, and a Scholarship from Jilin University.}

\author{Weiwei Qi}
\address{Department of Mathematical and Statistical Sciences, University of Alberta, Edmonton, AB T6G 2G1, Canada}
\email{wqi2@ualberta.ca, qiweiwei13@mails.ucas.ac.cn}

\author{Zhongwei Shen}
\address{Department of Mathematical and Statistical Sciences, University of Alberta, Edmonton, AB T6G 2G1, Canada}
\email{zhongwei@ualberta.ca}

\author{Yingfei Yi}
\address{Department of Mathematical and Statistical Sciences, University of Alberta, Edmonton, AB T6G 2G1, Canada, and School of Mathematics, Jilin University, Changchun 130012, PRC}
\email{yingfei@ualberta.ca}

\begin{abstract}

The present paper is devoted to the investigation of the long term behavior of a class of singular multi-dimensional diffusion processes that get absorbed in finite time with probability one. Our focus is on the analysis of quasi-stationary distributions (QSDs), which describe the long term behavior of the system conditioned on not being absorbed. Under natural Lyapunov conditions, we construct a QSD and prove the sharp exponential convergence to this QSD for compactly supported initial distributions. Under stronger Lyapunov conditions ensuring that the diffusion process comes down from infinity, we show the uniqueness of a QSD and the exponential convergence to the QSD for all initial distributions. Our results can be seen as the multi-dimensional generalization of Cattiaux et al (Ann. Prob. 2009) as well as the complement to Hening and Nguyen (Ann. Appl. Prob. 2018) which looks at the long term behavior of multi-dimensional diffusions that can only become extinct asymptotically.

The centerpiece of our approach concerns a uniformly elliptic operator that we relate to the generator, or the Fokker-Planck operator, associated to the diffusion process. This operator only has singular coefficients in its zeroth-order terms and can be handled more easily than the generator. For this operator, we establish the discreteness of its spectrum, its principal spectral theory, the stochastic representation of the semigroup generated by it, and the global regularity for the associated parabolic equation. We show how our results can be applied to most ecological models, among which cooperative, competitive, and predator-prey Lotka-Volterra systems.

\end{abstract}

\subjclass[2010]{Primary 60J60, 60J70, 34F05; secondary 92D25, 60H10}



\keywords{diffusion process, quasi-stationary distribution, existence, uniqueness, exponential convergence, spectral theory, semigroup, stochastic representation, stochastic Lotka-Volterra system}

\maketitle

\tableofcontents


\section{\bf Introduction}

Absorbed diffusion processes are often used in population biology to model the evolution of interacting species. Although the eventual extinction of all species is inevitable due to finite population effects (finite resources, finite population sizes, mortality, etc.) species can typically persist for a period of time that is long compared to human timescales \cite{CM10}. It is important to understand the behavior of the ecosystem before the eventual extinction. This motivates the study of the dynamics of multi-dimensional diffusion processes conditioned on not going extinct.

To be more specific, consider the stochastic Lotka--Volterra competition system:
\begin{equation}\label{sde-LV}
\md Z^i_t=Z^i_t\left(r_i-\sum_{j=1}^d c_{ij}Z^j_t\right)\md t+\sqrt{\ga_i Z^i_t}\md W^i_t,\quad i\in\{1,\dots, d\},
\end{equation}
where $Z_{t}=(Z^i_t)\in\ol{\UU}:=[0,\infty)^{d}$ are the abundances of the species at time $t$, $\{r_i\}_{i}$ are per-capita growth rates, $\{c_{ii}\}_{i}$ are the intra-specific competition rates, $\{c_{ij}\}_{i\neq j}$ are inter-specific competition rates, $\{\ga_i\}_{i}$ are demographic parameters describing ecological timescales (see e.g. \cite{CCLMMS09, CM10}), and $\{W^i\}_{i}$ are independent standard one-dimensional Wiener processes on some probability space. It is well-known (see e.g. \cite{CM10,CV18-general}) that $Z_{t}$ reaches the boundary, also called the extinction set, $\Ga:=\left\{z=(z_i)\in \ol{\UU}:  z_i=0\text{ for some } i\in\{1,\dots, d\}\right\}$, of $\ol{\UU}$ in finite time almost surely. This corresponds to the extinction of at least one species of the considered community. Nonetheless, typical trajectories or sample paths of $Z_{t}$ will stay in $\UU:=(0,\infty)^{d}$ for a long period before hitting $\Ga$. This can be interpreted as the temporary coexistence of species, before their ultimate extinction. To understand this type of behavior, notions such as quasi-steady states and metastable states have been put forward. These concepts are often formalized in terms of the \emph{quasi-stationary distributions} (QSDs), which are stationary distributions of $Z_{t}$ conditioned on no species going extinct. In this context, it is of fundamental mathematical importance to analyze the existence, uniqueness, and domains of (exponential) attraction of QSDs.

The purpose of the present paper is to investigate the existence and uniqueness of QSDs and the exponential convergence to QSDs for a class of irreversible diffusion processes given by models of the form
\begin{equation}\label{sde-orig}
\md Z_{t}^i=b_i(Z_{t})\md t+\sqrt{a_i(Z_{t}^i)}\md W^i_t,\quad i\in \{1,\dots,d\},
\end{equation}
where $Z_{t}:=(Z_{t}^i)\in \ol{\UU}$, $b_i:\ol{\UU}\to \R$ and $a_i:[0,\infty)\to [0,\infty)$. 
We make the following assumptions.
\begin{itemize}
\item [{\bf (H1)}] $a_i\in C^2([0,\infty))$, $a_i(0)=0$, $a'_i(0)>0$, $a_{i}>0$ on $(0,\infty)$,  $\limsup_{s\to \infty}\left[\frac{|a'_{i}(s)|^2}{a_i(s)}+a''_i(s)\right]<\infty$ and $\int_1^{\infty} \frac{\md s}{\sqrt{a_i(s)}}=\infty$ for all $i\in \{1,\dots,d\}$.

\item [{\bf (H2)}] $b_i\in C^1(\ol{\UU})$ and  $\left.b_i\right|_{z_i=0}=0$ for all $i\in \{1,\dots,d\}$, where $z_{i}=0$ means the set $\left\{z=(z_{i})\in\ol{\UU}:z_{i}=0\right\}$.

\item [{\bf (H3)}] There exists a positive function $V\in C^2(\ol{\UU})$ satisfying the following conditions.
\begin{enumerate}
\item[\rm(1)]  $\lim_{|z|\to \infty} V(z)=\infty$ and  $\lim_{|z|\to \infty} (b\cdot\nabla_{z} V)(z)= -\infty$.

\item[\rm(2)] There exists a non-negative and continuous function $ \Tilde{V}:[0,\infty)\to [0,\infty)$ satisfying
$$
\int_{1}^{\infty}\frac{e^{-\be \tilde{V}}}{a_i}\md s<\infty,\quad\forall \be>0\andd i\in \{1,\dots, d\}
$$
such that $V(z)\geq \sum_{i=1}^d\tilde{V}(z_i)$ for all $z=(z_i)\in \ol{\UU}$.

\item[\rm(3)]  The following limit holds
$$
\lim_{|z|\to\infty} \frac{1}{b\cdot\nabla_{z} V}\sum_{i=1}^d\left(| \pa_{z_i} b_i|+\frac{|a'_ib_i|}{a_i}+|a'_i\pa_{z_i}V|+|a_i \pa_{z_iz_i}^{2}V|\right)=0.
$$

\item[\rm(4)] There exist constants $C>0$ and $R>0$ such that
$$
\sum_{i=1}^d \left( a_i|\pa_{z_i} V|^2 +\frac{b_i^2}{a_i}\right)\leq -Cb\cdot \nabla_{z}V\quad\text{in}\quad \UU\sm B_{R}^+,
$$
where $B_R^+:=\{z=(z_i)\in\UU: z_{i}\in(0,R),\,\,\forall i\in\{1,\dots, d\}\}$ for $R>0$.
\end{enumerate}
\end{itemize}

Assumption {\bf (H1)} says that each $a_{i}(s)$ behaves like $a_{i}'(0)s$ near $s\approx 0$, and allows each $a_{i}(s)$ to behave like $s^{\ga}$ for some  $\ga\in(-\infty,2]$ near $s\approx\infty$. Assumption {\bf(H2)} is satisfied if $b_{i}(z)=z_{i}f_{i}(z)$ for $f_{i}\in C^{1}(\ol{\UU})$. {\bf(H1)} and {\bf(H2)} ensure that \eqref{sde-orig} generates a diffusion process $Z_t$ on $\ol{\UU}$ having $\Ga$ as an absorbing set. {\bf(H3)}(1) and the condition $\lim_{|z|\to\infty }\frac{\sum_{i=1}^{d}|a_i \pa_{z_iz_i}^{2}V|}{b\cdot\nabla_{z}V}=0$ contained in {\bf(H3)}(3) imply the dissipativity of $Z_t$, and hence, that it does not explode in finite time almost surely. Other assumptions in {\bf (H3)} are technical ones, but they are made according to examples discussed in Section \ref{sec-app}. We note that for a reversible system, the potential function is a natural choice for $V$. For irreversible systems, polynomials are usually good choices for $V$, especially when the coefficients are polynomials or rational functions -- this is often the case in applications.

We show in Proposition \ref{prop-absorbing} that $Z_t$ reaches $\Ga$ in finite time almost surely under {\bf (H1)}-{\bf (H3)}, and hence, that $Z_{t}$ does not admit a stationary distribution that has positive concentration in $\UU$. It is then natural to look at $Z_{t}$ before reaching $\Ga$ in order to understand the dynamics of $Z_{t}$. This drives us to examine quasi-stationary distributions of $Z_t$ or \eqref{sde-orig} conditioned on coexistence, i.e., $[t<T_{\Ga}]$, where $T_{\Ga}:=\inf\{t>0: Z_{t}\in \Ga\}$ is the first time when $Z_{t}$ hits $\Ga$. Denote by $\P^{\mu}$ the law of $Z_{t}$ with initial distribution $\mu$, and by $\E^{\mu}$ the expectation with respect to $\P^{\mu}$.

\begin{defn}[Quasi-stationary distribution]\label{defn-qsd}
A Borel probability measure $\mu$ on $\UU$ is called a {\em quasi-stationary distribution} (QSD) of $Z_{t}$ or \eqref{sde-orig} if for each $f\in C_b(\UU)$, one has
$$
\E^{\mu}\left[f(Z_{t})\big|t<T_{\Ga}\right]=\int_{\UU} f\md\mu,\quad\forall t\geq0.
$$
\end{defn}

The QSDs of $Z_{t}$ are simply stationary distributions of $Z_{t}$ conditioned on $[t<T_{\Ga}]$. This is why QSDs can be seen as governing the dynamics of $Z_{t}$ before extinction. It is known from the general theory of QSDs (see e.g. \cite{MV12,CMS13}) that if $\mu$ is a QSD of $Z_t$, then there exists a unique $\la>0$ such that if $Z_{0}\sim \mu$ the time $T_{\Ga}$ is exponentially distributed with rate $\la$, i.e., $\P^{\mu}\left[T_{\Ga}>t\right]=e^{-\la t}$ for all $t\geq0$. The number $\la$ is often called the \emph{extinction rate} associated to $\mu$.


Our first result concerning the existence of QSDs and the conditioned dynamics of $Z_{t}$ is stated in the following theorem. Denote by $\PP(\UU)$ the set of Borel probability measures on $\UU$.

\begin{thmx}\label{thm-qsd-existence-dynamics}
Assume {\bf(H1)}-{\bf(H3)}.
Then, $Z_{t}$ admits a QSD $\mu_{1}$, and there exists $r_{1}>0$ such that the following hold.
\begin{itemize}
\item For any $0<\ep\ll1$ and $\mu\in\PP(\UU)$ with compact support in $\UU$ we have
$$
\lim_{t\to\infty}e^{(r_{1}-\ep)t}\left|\E^{\mu}\left[f(Z_{t})\big|t<T_{\Ga}\right]-\int_{\UU}f\md\mu_{1}\right|=0,\quad\forall f\in C_{b}(\UU).
$$

\item There exists $f\in C_{b}(\UU)$ such that for a.e. $x\in \UU$, there is a discrete set $\II_x\subset(0,\infty)$ with distances between adjacent points admitting an $x$-independent positive lower bound, such that for each $0<\de\ll 1$ we have
$$
\lim_{\substack{t\to\infty\\t\in (0,\infty)\sm\II_{x,\de}}} e^{(r_1+\ep)t}\left|\E^{x} [f(X_t)\big |t<T_{\Ga}]-\int_{\UU}f\md\mu_1\right|=\infty,\quad\forall 0<\ep\ll 1,
$$
where $\II_{x,\de}$ is the $\de$-neighbourhood of $\II_x$ in $(0,\infty)$.
\end{itemize}
\end{thmx}
\begin{rem}
The first conclusion in Theorem \ref{thm-qsd-existence-dynamics} actually holds for a much larger class of initial distributions (see Remark \ref{rem-on-sharp-convergence} for more details).
\end{rem}

We point out that the sharp exponential convergence rate $r_{1}$ is given by the spectral gap, between the principal eigenvalue and the rest of the spectrum, of the Fokker-Planck operator associated to $Z_{t}$ in an appropriate weighted function space. The QSD is essentially given by the positive eigenfunction associated to the principal eigenvalue, and the associated extinction rate is just the absolute value of the principal eigenvalue. Such characterizations of the QSD and the exponential convergence rate have been obtained in \cite{CCLMMS09,CM10} in the reversible case. Our result is the first of this type for the general setting when $Z_t$ is irreversible.  Theorem \ref{thm-qsd-existence-dynamics} applies to a large class of population models including stochastic Lotka-Volterra models, models with Holling type functional responses, and Beddington-DeAngelis models. We refer the reader to Section \ref{sec-app} for more details.

The set $\II_{x}$ in the second conclusion more or less corresponds to the zeros of the function $t\mapsto\E^{x} [f(X_t)\big |t<T_{\Ga}]-\int_{\UU}f\md\nu_1$. For irreversible systems one generally has complex eigenvalues, which give rise to oscillations. As a result, the zeros of the above function exist and form a discrete set as described in the statement of Theorem \ref{thm-qsd-existence-dynamics}.


Although the QSD $\mu_{1}$ obtained in Theorem \ref{thm-qsd-existence-dynamics} attracts all compactly supported initial distributions, there is no assertion that it is the unique QSD of the process $Z_t$. To study the uniqueness, we make the following additional assumption.
\begin{itemize}
\item [{\bf (H4)}] There exist positive constants $C$, $\ga$ and $R_*$ such that
\begin{gather*}
\lim_{|z|\to \infty}V^{-\ga-2}\sum_{i=1}^{d}a_i|\pa_{z_i} V|^2 = 0\quad \andd \quad \frac{1}{2}\sum_{i=1}^{d}a_i\pa^2_{z_i z_i}V+b\cdot\nabla_{z} V\leq -CV^{\ga+1}\quad \text{in}\quad \UU\sm B_{R_*}^+.
\end{gather*}
\end{itemize}

\begin{thmx}\label{thm-uniqueness}
Assume {\bf (H1)}-{\bf(H4)}. Let $\mu_{1}$ and $r_{1}$ be as in Theorem \ref{thm-qsd-existence-dynamics}.
Then, $\mu_{1}$ is the unique QSD of $Z_{t}$, and for any $0<\ep\ll1$ and $\mu\in\PP(\UU)$, there holds
    $$
    \lim_{t\to\infty}e^{(r_{1}-\ep)t}\left|\E^{\mu}\left[f(Z_{t})\big|t<T_{\Ga}\right]-\int_{\UU}f\md\mu_{1}\right|=0,\quad\forall f\in C_{b}(\UU).
    $$
\end{thmx}

Assumption {\bf (H4)} concerns the strong dissipativity of $Z_{t}$ near infinity, and implies in particular that $Z_{t}$ \textit{comes down from infinity} (see Remark \ref{lem-5-29-1-1}), that is, for each $\la>0$, there exists $R=R(\la)>0$ such that $\sup_{z\in \UU\sm B^+_R}\E^{z}\left[e^{\la T_R}\right]<\infty$, where $T_R:=\inf\left\{t\geq0: Z_t\not\in\UU\sm B_R^+\right\}$. This property plays a crucial role in the proof of Theorem \ref{thm-uniqueness}. It says that with high probability the process $Z_{t}$ quickly enters a bounded region. This happens even if the initial distribution of $Z_t$ has a heavy tail near $\infty$. As a result, it makes no difference to the QSD $\mu_{1}$ whether the initial distribution of $Z_{t}$ is compactly supported or not. Theorem \ref{thm-uniqueness} applies to a large class of biological models including in particular the stochastic competition system \eqref{sde-LV} and the stochastic weak cooperation system (i.e., the system \eqref{sde-LV} with $\{-c_{ij}\}_{i\neq j}$ being positive and small in comparison to $\{c_{ii}\}_{i}$). See Section \ref{sec-app} for more details.

\textbf{Comparison to existing literature.} Due to their popularity in describing non-stationary states that are often observed in applications, QSDs have been attracting significant attention. We refer the reader to \cite{Pollet,MV12,CMS13} and references therein for an overview of the theory, developments and applications of QSD. We next present the current state of the art for diffusion processes. The investigation of QSDs for one-dimensional diffusion processes has attracted a lot of attention. We refer the reader to \cite{Mandl61,CMSM95,MSM04,SE07,KS12,ZH16,CV16,CV17,CV18} and references therein for the analysis of the regular case.

For singular diffusion processes including in particular \eqref{sde-LV} and \eqref{sde-orig} in the one-dimensional setting, the work \cite{CCLMMS09} lays the foundation and is generalized in \cite{Lit12, Miu14, CV18, HK19}. In contrast, there have not been many studies of QSDs for multi-dimensional diffusion processes. Regular diffusion processes restricted to a bounded domain and killed on the boundary have been studied in \cite{Pinsky85,GQZ88,CV18-general}. The stochastic competition system \eqref{sde-LV} has been studied in \cite{CM10} in the reversible case, and in \cite{CV19} in the irreversible case. In both of the above papers, the exponential convergence to the unique QSD is established. In \cite{CM10}, the authors also deal with the model in the weak cooperation case. The model treated in \cite{CV19} has a more general deterministic vector field. In \cite{CV18-general}, the authors study general multi-dimensional diffusion processes,  establish the existence and convergence to the QSD. However, they do not look at the uniqueness problem.

The approaches used in \cite{CM10}, \cite{CV19} and \cite{CV18-general} for treating singular diffusion processes in higher dimensions are quite different. The work done in \cite{CM10} relies on the spectral theory of the generator in the weighted space $L^{2}(\UU,d\mu)$ with $\mu$ being the infinite Gibbs measure, and the density of the Markov semigroup with respect to $\mu$. These tools are developed earlier in \cite{CCLMMS09} for one-dimensional singular diffusion processes. We note that the assumptions from \cite{CM10} make the diffusion process reversible, and much easier to analyze. However, most multi-dimensional diffusion processes would be irreversible. In \cite{CV19}, the authors study general Markov processes and apply their abstract results in particular to \eqref{sde-LV}. The main purpose of \cite{CV19} is to find practical sufficient conditions in order to use the necessary and sufficient condition of Doeblin-type established in their earlier work \cite{CV16} for the exponential convergence to the unique QSD. Their sufficient conditions are composed of a Lyapunov condition involving a pair of functions, a local Doeblin-type condition and a ``coming down from infinity" condition. In \cite{CV18-general}, the authors almost only impose a commonly used dissipative Lyapunov condition except the Lyapunov constant is assumed to be greater than some exponential rate related to the first exit time.

In comparison to \cite{CM10,CV19,CV18-general}, the main novelty of the present paper lies in the approach that allows us to treat $Z_{t}$ or \eqref{sde-orig} under elementary Lyapunov conditions. The results we obtain are as strong as those established in \cite{CCLMMS09,CM10} for reversible diffusion processes, even though we work in the much more general setting of irreversible processes. The centerpiece of our approach is the spectral theory of an elliptic operator derived from the Fokker-Planck operator associated to $Z_{t}$ through a two-stage change of variables. We establish the discreteness of the  spectrum  of this operator as well as its principal spectral theory. These results allow us to use the semigroup generated by this operator in order to establish its stochastic representation and to study the fine dynamical properties of $Z_{t}$. A direct consequence of our approach is the characterization of the sharp exponential convergence rate in Theorem \ref{thm-qsd-existence-dynamics} -- this was previously unknown for singular diffusion processes in higher dimensions. Moreover, since spectral theory and the stochastic representation of semigroups are important tools, our results go beyond the study of QSD and are of independent interest.

\textbf{Demographic and environmental stochasticity.}
Consider an isolated ecosystem of interacting species. Due to finite population effects and demographic stochasticity, extinction of all species is certain to occur in finite time for all populations. However, the time to extinction can be large and the species densities can fluctuate before extinction occurs.

One way of capturing this behaviour is ignoring the effects of demographic stochasticity (i.e. finite population effects) and focusing on models with environmental stochasticity where extinction can only be asymptotic as $t\to\infty$. This approach led to the development of the field of modern coexistence theory (MCT), started by Lotka \cite{L26} and Volterra \cite{V28}, and later developed by Chesson \cite{C82, C20} and other authors \cite{T77, G84, SBA11,BHS08}. Recently, there have been powerful results that have led to a general theory of coexistence and extinction \cite{HN18, B18, HNC20}.

A second way of analyzing the long term dynamics of the species is by including demographic stochasticity and studying the QSDs of the system - this is the approach we took in this paper. Our work can be seen as complementary to the work done for systems with environmental stochasticity.

\textbf{Overview of proofs.} The proofs of Theorem \ref{thm-qsd-existence-dynamics} and Theorem \ref{thm-uniqueness} use techniques from PDE, spectral theory, semigroup theory and probability theory, and are rather involved. For the reader's convenience, we outline the strategy of the proofs with the help of Figure \ref{figure-outline}.


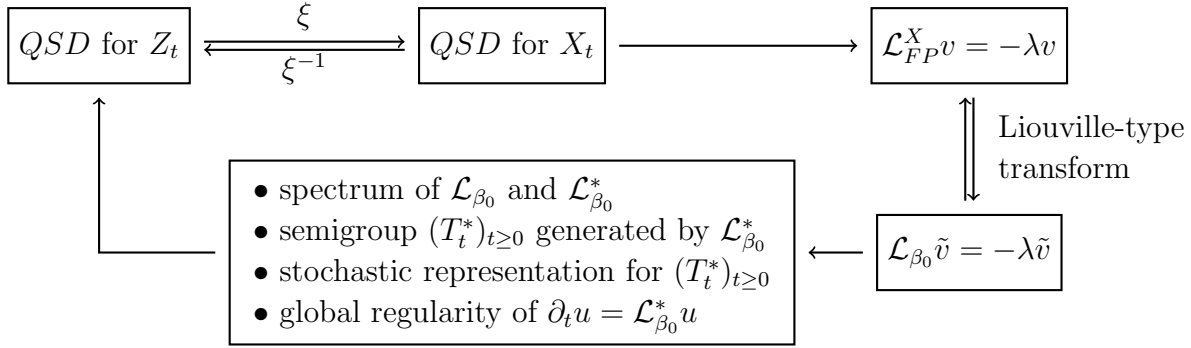
\begin{figure}\label{figure-outline}
	\begin{center}
		\begin{tikzpicture}[
			thick,
			nonterminal/.style = {
				rectangle,			
				minimum size = 10mm,			
				thick,
				draw = black,
			}
		]
			\matrix [row sep=10mm, column sep=5mm] {
				\node (qsdz) [nonterminal] {$QSD$ for $Z_t$}; &
				\node (qsdx) [nonterminal] {$QSD$ for $X_t$}; && 
				\node (lfp) [nonterminal] {$\mathcal{L}^X_{FP}v = -\lambda v$}; \\ 
				&
				\node (lba) [nonterminal] {
						\begin{tabular}{l}
							$\bullet$ spectrum of $\mathcal{L}_{\beta_0}$ and $\mathcal{L}^*_{\beta_0}$ \\
							$\bullet$ semigroup $(T^*_t)_{t\geq 0}$ generated by $\mathcal{L}^*_{\beta_0}$ \\
							$\bullet$ stochastic representation for $(T^*_t)_{t\geq 0}$ \\
							$\bullet$ global regularity of $\partial_t u  = \mathcal{L}^*_{\beta_0} u$
						\end{tabular}
					}; && 
				\node (lbeta) [nonterminal] {$\mathcal{L}_{\beta_0}\tilde{v} = -\lambda \tilde{v}$};	\\
			};
			
			\path (qsdz) edge[->, shorten <=5pt, shorten >=5pt, transform canvas={yshift=0.3ex}] node[above]{$\xi$} node[below]{$\xi^{-1	}$} (qsdx);
			\path (qsdx) edge[->, shorten <=5pt, shorten >=5pt, transform canvas={yshift=-0.3ex}] (qsdz);
			\path (qsdx) edge[->, shorten <=5pt, shorten >=5pt] (lfp);
			\path (lfp) edge[->, shorten <=5pt, shorten >=5pt, transform canvas={xshift=0.3ex}] node[right]{\begin{tabular}{l} Liouville-type \\ transform\end{tabular}} (lbeta);
			\path (lbeta) edge[->, shorten <=5pt, shorten >=5pt, transform canvas={xshift=-0.3ex}] (lfp);
			\path (lbeta) edge[->, shorten <=5pt, shorten >=5pt ] (lba);
			\path (lba) edge[->, shorten <=5pt, shorten >=5pt, to path = {-| (\tikztotarget)}] (qsdz);
		\end{tikzpicture}
	\end{center}
	    \caption{Overview of proofs.}
\end{figure}
\begin{itemize}
\item (Equivalent formalism) Theoretically, the study of QSDs of $Z_t$ can be accomplished by investigating the (principal) spectral theory of $\LL^Z_{\bf FP}$, the Fokker-Planck operator associated to $Z_{t}$. However, the degeneracy of $\LL^Z_{\bf FP}$ on $\Ga$ would cause significant drawbacks. To circumvent this, we introduce a homeomorphism $\xi:\ol{\UU}\to\ol{\UU}$ and define a new process  $X_t=\xi(Z_t)$ whose Fokker-Planck operator $\LL^X_{\bf FP}$ has  $\frac{1}{2}\De$ as its second-order term.

Although $\LL^X_{\bf FP}$ has the best possible second-order term, the coefficients of its first-order terms unfortunately have blow-up singularities on $\Ga$. Introducing a Liouville-type transform, we convert $\LL^X_{\bf FP}$ into a uniformly elliptic operator
$$
\LL_{\be_0}:=e^{\frac{Q}{2}+\beta_{0}U}\LL_{\bf FP}^{X}e^{-\frac{Q}{2}-\beta_{0}U}
$$
whose blow-up singularities on $\Ga$ only appear in the coefficients of the zeroth-order terms. Here $U=V\circ\xi^{-1}$ and $Q$, given in \eqref{eqn-function-Q}, has singularities near $\Ga$ (see Remark \ref{rem-about-Q-near-Ga}).

The number $\beta_{0}$ is chosen so that $\LL_{\be_0}$ satisfies certain a priori estimates. The details are presented in Subsection \ref{subsec-equivalent-formalism}. The number $\beta_{0}$ is fixed in Lemma \ref{lem-3-24-2} (3).

\item (Spectral analysis) Our spectral analysis focuses on the operator $\LL_{\be_0}$ in $L^{2}(\UU;\C)$ as well as its adjoint $\LL_{\beta_{0}}^{*}$. According to the behavior of the coefficients of $\LL_{\be_0}$ near $\Ga$ and infinity, we design a weight function and define a weighted first-order Sobolev space $\HH^{1}(\UU;\C)$ that is compactly embedded into $L^2(\UU;\C)$. Applying the a priori estimates of $\LL_{\be_0}$, we are able to solve the elliptic problem for $\LL_{\be_0}-M$ for some $M\gg1$ in $\HH^{1}(\UU;\C)$. The discreteness of the spectrum and principal spectral theory of $\LL_{\be_0}$ and $\LL_{\beta_{0}}^{*}$ then follow.

The details are given in Subsection \ref{spectrum} and Subsection \ref{subsec-adjoint-operator}.

\item (Semigroup and stochastic representation) The operator $\LL^*_{\be_0}$ generates an analytic and eventually compact semigroup $(T^*_t)_{t\geq 0}$ on $L^{2}(\UU;\C)$ that can be ``block-diagonalized" according to spectral projections. We establish the representation of $(T^*_t)_{t\geq 0}$ in terms of $X_{t}$ before reaching $\Ga$, and therefore, connect the dynamics of $(T^*_t)_{t\geq 0}$ with that of $X_{t}$ conditioned on $[t<S_{\Ga}]$, where $S_{\Ga}$ is the first time that $X_{t}$ hits $\Ga$. More precisely, we show that for each $f\in C_b(\UU;\C)$ satisfying $\tilde{f}:=fe^{-\frac{Q}{2}-\be_0 U}\in L^2(\UU;\C)$, there holds
\begin{equation*}
T^{*}_t\tilde{f}=e^{-\frac{Q}{2}-\be_0 U}\E^{\bullet}\left[f(X_t)\mathbbm{1}_{\{t<S_{\Ga}\}}\right],\quad\forall t\geq0.
\end{equation*}

The semigroups are given in Subsection \ref{spectrum} and Subsection \ref{subsec-adjoint-operator}. The stochastic representation of $(T^*_t)_{t\geq 0}$ is established in Subsection \ref{subsec-stochastic-representation}.

\item (Global regularity and conclusions) The spectral theory and stochastic representation allow us to prove the results as in Theorem \ref{thm-qsd-existence-dynamics} and Theorem \ref{thm-uniqueness} for the process $X_{t}$. While proving the existence of QSDs is pretty straightforward, we run into significant technical difficulties trying to establish the convergence even for compactly supported initial distributions. This is due to: (i) the limitations of the stochastic representation because of the unboundedness of the Liouville-type transform and its inverse (i.e., $e^{\frac{Q(x)}{2}+\beta_{0}U(x)}$ grows to $\infty$ as $|x|\to\infty$ and $e^{-\frac{Q}{2}-\beta_{0}U}$ blows up at $\Ga$); (ii) the requirement of $L^{\infty}$ properties of $(T^*_t)_{t\geq 0}$. These issues are overcome by establishing the global regularity of solutions of $\partial_{t}u=\LL_{\beta_{0}}^{*}u$ leading in particular to the global regularity of $(T^*_t)_{t\geq 0}$.

The details are given in Section \ref{sec-exi-uniq-conv}.
\end{itemize}

\medskip

The rest of the paper is organized as follows. In Section \ref{sec-pre}, we provide some preliminaries including the proof of $Z_t$ being absorbed by $\Ga$ in finite time almost surely, the derivation of the operator $\LL_{\be_0}$, and results related to the approximation of $S_{\Ga}$. In Section \ref{sec-spectrum-semigroup}, we study the spectral theory of $\LL_{\be_0}$ and its adjoint operator $\LL^{*}_{\beta_{0}}$, and establish the associated semigroups $(T_t)_{t\geq 0}$ and $(T^*_t)_{t\geq 0}$. Section \ref{sec-stoch-repre} is devoted to the stochastic representation of $(T^*_t)_{t\geq 0}$. In Section \ref{sec-exi-uniq-conv}, we investigate the existence and uniqueness of QSDs and the exponential convergence to QSDs of $X_t$ conditioned on the coexistence. Theorem \ref{thm-qsd-existence-dynamics} and Theorem \ref{thm-uniqueness} are proven in this section. In the last section, Section \ref{sec-app}, we discuss applications of Theorem \ref{thm-qsd-existence-dynamics} and Theorem \ref{thm-uniqueness} to a wider variety of ecological models including stochastic Lotka-Volterra systems, and  models with Holling  type or Beddington-DeAngelis functional responses. Appendix \ref{sec-app-proof-technical-lem} is included to provide the proof of some technical lemmas.


\section{\bf Preliminaries}\label{sec-pre}


In Subsection \ref{subsec-absorb}, we show that $Z_{t}$ hits $\Ga$ in finite time almost surely. In Subsection \ref{subsec-equivalent-formalism}, we present equivalent formulations for studying the existence of QSDs, and derive the operator we shall focus on in later sections. In Subsection \ref{subsec-first-exit-times}, we fix a family of first exit times and present an approximation result.

\subsection{Hitting the absorbing boundary}\label{subsec-absorb}

We prove that $Z_{t}$ reaches $\Ga$ in finite time almost surely. Denote by $\LL^{Z}$ the diffusion operator associated to $Z_{t}$, namely,
$$
\LL^{Z}=\frac{1}{2}\sum_{i=1}^{d}a_i\pa^2_{z_iz_i}+b\cdot\nabla_z.
$$

\begin{prop}\label{prop-absorbing}
Assume {\bf (H1)}-{\bf(H3)}. Then, $\P^z[T_{\Ga}<\infty]=1$ for each $z\in\UU$.
\end{prop}
\begin{proof}
Note that {\bf (H1)}-{\bf (H2)} ensure the pathwise uniqueness as well as the strong Markov property of solutions of \eqref{sde-orig} until the explosion time. Recall that for $R>0$,
$$
B^+_{R}=\left\{z=(z_{i})\in \UU: z_i\in(0,R),\,\,\forall i\in\{1,\dots, d\}\right\}.
$$
The result is proven in four steps.

\medskip

\paragraph{\bf Step 1} We claim the existence of $R>0$ such that $\P^z[T_{R}<\infty]=1$ for each $z\in \UU$, where $T_{R}:=\inf\left\{t\in [0,T_{\Ga}]:Z_t\in B^+_R\right\}$.

By the assumptions {\bf (H3)}(1)(3), there is $R>0$ such that
$\LL^{Z}V\leq -1$ in $\UU\sm B_{R}^+$. For each $z\in \UU$, It\^o's formula  gives
\begin{equation*}
\begin{split}
\E^z\left[V(Z_{t\wedge T_{R}})\right]=V(z)+\E^z\left[\int_0^{t\wedge T_{R}}\LL^{Z}V(Z_s)\md s\right]\leq V(z)-\E^z\left[t\wedge T_{R}\right],\quad \forall t\geq 0.
\end{split}
\end{equation*}
Passing to the limit $t\to \infty$ yields $\E^z\left[T_{R}\right]\leq V(z)<\infty$.
The claim follows.

\medskip

\paragraph{\bf Step 2} We prove $\P^z[\tau_{2R}<\infty]=1$ for each $z\in B^+_{2R}$, where $\tau_{2R}:=\inf\left\{t\geq 0:Z_t\notin B^+_{2R}\right\}$.


For each $i\in \{1,\dots, d\}$, we set $\ol{b}_i:=\sup_{B^+_{2R}}b_i$, denote by $Y_{t}^{i,y_{i}}$ the solution of the following one-dimensional SDE
\begin{equation*}\label{eqn-6-14-1}
    \md Y^i_t=\ol{b}_i\md t+\sqrt{a_i(Y_i)}\md W^i_t
\end{equation*}
with initial condition $Y_{0}^{i,y_{i}}=y_{i}\in[0,\infty)$, and let $\tau_{i}^{y_{i}}$ be the first time that $Y_{t}^{i,y_{i}}$ hits $0$, namely, $\tau_{i}^{y_{i}}=\inf\left\{t\geq0:Y_{t}^{i,y_{i}}=0\right\}$.
The assumptions on $a_{i}$ and \cite[Theorem VI-3.2]{IW81} guarantee that $\P\left[\tau_{i}^{y_{i}}<\infty\right]=1$ for all $y_{i}\in[0,\infty)$ and $i\in\{1,\dots,d\}$.

Let $z=(z_{i})\in B^+_{2R}$. By the comparison theorem for one-dimensional SDEs (see e.g. \cite[Theorem VI-1.1]{IW81}) and the fact that $\P\left[\tau_{i}^{z_{i}}<\infty\right]=1$ for each $i\in\{1,\dots,d\}$, we find up to a set of  probability zero,
\begin{equation*}\label{eqn-6-14-2}
[\tau_{2R}=\infty]\subset\left[Z^i_t\leq Y^{i,z_{i}}_t,\,\,\forall t\in[0,\tau_{i}^{z_{i}}],\,\,i\in\{1,\dots,d\}\right]\subset [\tau_{2R}<\infty].
\end{equation*}
From this we conclude that $\P^{z}[\tau_{2R}=\infty]=0$.

\medskip

\paragraph{\bf Step 3} We show that $\inf_{z\in \ol{B^+_R}}\P^z[Z_{\tau_{2R}}\in \Ga]>0$.

Fix $i\in\{1,\dots, d\}$. Calculating the probability that the process $Y^{i,R}_{t}$ first exits the interval $(0,\frac{3R}{2})$ through $\frac{3R}{2}$
(see \cite[Theorem VI-3.1]{IW81}), we find $\P\left[Y^{i,R}_{t}\in \left[0,3R/2\right),\,\,\forall t\in [0,\tau_i^R]\right]>0$. Since
$$
\P\left[Y^{i,y_i}_t\leq Y^{i,R}_t,\,\,\forall t\in [0, \tau^{y_i}_i]\right]=1,\quad\forall y_{i}\in[0,R]
$$
due to the comparison theorem (see e.g. \cite[Theorem VI-1.1]{IW81}), we deduce
\begin{equation*}\label{eqn-2020-06-17}
    \inf_{y_i\in [0,R]}\P\left[Y^{i,y_i}_{t}\in \left[0,3R/2\right),\,\,\forall t\in[0,\tau^{y_i}_i]\right]\geq \P\left[Y^{i,R}_{t}\in \left[0,3R/2\right),\,\,\forall t\in[0,\tau^{R}_i]\right]>0.
\end{equation*}
This together with the comparison theorem yields for each $z=(z_{i})\in \ol{B^+_R}$,
\begin{equation*}
\begin{split}
    \P^z[Z_{\tau_{2R}}\in\Ga]&\geq\P\left[Y^{i,z_i}_{t}\in \left[0,3R/2\right),\,\,\forall t\in[0,\tau^{z_i}_i],\,\,i\in\{1,\dots,d\}\right]\\
    &=\prod_{i=1}^d\P\left[Y^{i,z_i}_{t}\in \left[0,3R/2\right),\,\,\forall t\in[0,\tau^{z_i}_i]\right]\\
    &\geq\prod_{i=1}^d \inf_{y_i\in [0,R]}\P\left[Y^{i,y_i}_{t}\in \left[0,3R/2\right),\,\,\forall t\in[0,\tau^{y_i}_i]\right]>0,
\end{split}
\end{equation*}
where we used the independence of $Y_{t}^{i,z_{i}}$, $i\in\{1,\dots,d\}$ in the equality. The claim follows.

 \medskip

\paragraph{\bf Step 4} We finish the proof of the proposition. By {\bf Step 3}, $p:=\inf_{z\in \pa B^+_{R}\sm \Ga}\P^z[Z_{\tau_{2R}}\in \Ga]>0$. Set
$$
T^{(1)}_{R}:=\inf\left\{t\in [0,T_{\Ga}]:Z_t\in B^+_R\right\}\quad\text{and}\quad S^{(1)}_{2R}:=\inf\left\{t\geq T^{(1)}_R:Z_t\notin B^+_{2R}\right\},
$$
and recursively define for each $n\geq 1$,
$$
T^{(n+1)}_{R}:=\inf\left\{t\in [S^{(n)}_{2R},T_{\Ga}]:Z_t\in B^+_R\right\}\quad\andd\quad S^{(n+1)}_{2R}:=\inf\left\{t\geq T^{(n+1)}_R:Z_t\notin B^+_{2R}\right\}.
$$
Fix $z\in\UU$. Since {\bf Step 1}, {\bf Step 2} and the strong Markov property ensure $\P^{z}\left[T^{(n)}_{R}<\infty\right]=1$ and $\P^{z}\left[S^{(n)}_{2R}<\infty\right]=1$ for all $n\in\N$, we find
$\P^{z}\left[Z_{S^{(n)}_{2R}}\in \pa B^+_{2R}\sm \Ga\right]\leq (1-p)^n$ for all $n\in\N$. As a result
$$
\P^{z}\left[T_{\Ga}=\infty\right]=\P^{z}\left[S^{(n)}_{2R}<\infty,\forall n\in\N\right]\leq \lim_{n\to\infty}(1-p)^n=0.
$$
This completes the proof.
\end{proof}

\begin{rem}
The assumptions {\bf(H3)}(2)(4) are not needed in the proof of Proposition \ref{prop-absorbing}.
\end{rem}

\subsection{Equivalent formulation}\label{subsec-equivalent-formalism}

Denote by $\LL_{\bf FP}^{Z}$ the Fokker-Planck operator associated to $Z_{t}$ or \eqref{sde-orig}, namely,
\begin{equation}\label{fpe-orig}
\LL_{\bf FP}^{Z}u:=\frac{1}{2}\sum_{i=1}^{d}\pa^2_{z_iz_i}(a_iu)-\nabla_{z}\cdot(bu)\quad\text{in}\quad\UU.
\end{equation}

\begin{prop}\label{prop-5-7-1}
Assume {\bf (H1)}-{\bf (H2)}. Let $\mu$ be a QSD of $Z_{t}$. Then, $\mu$ admits a positive density $u\in C^2(\UU)$ that satisfies  $-\LL_{\bf FP}^{Z}u=\la_1 u$ in $\UU$, where $\la_{1}$ is the extinction rate associated to $\mu$.
\end{prop}
\begin{proof}
It follows from \cite[Proposition 4]{MV12} and the elliptic regularity theory (see e.g. \cite{BKR01}).
\end{proof}

Proposition \ref{prop-5-7-1} suggests studying the principal spectral theory of the operator $-\LL^{Z}$ in order to find a QSD for $Z_{t}$. Direct analysis of the operator $-\LL^{Z}$ is however difficult due to the degeneracy of the diffusion matrix $\text{diag}\{a_{1},\dots,a_{d}\}$ on the boundary $\Ga$ of $\UU$. To resolve this issue, we define a new process that is equivalent to $Z_{t}$ and whose Fokker-Planck operator or diffusion operator is uniformly non-degenerate in $\UU$. We proceed as follow.

For each $i\in \{1,\dots, d\}$, we define $\xi_i:[0,\infty)\to [0,\infty)$ by setting
$$
\xi_i(z_i):=\int_0^{z_i}\frac{1}{\sqrt{a_i(s)}}\md s,\quad z_i\in [0,\infty).
$$
By {\bf(H1)}, each $\xi_{i}$ is increasing and onto, and thus, $\xi_{i}^{-1}$ is well-defined. 
Set
$$
\xi:=(\xi_i):\ol{\UU}\to \ol{\UU}\quad\text{and}\quad \xi^{-1}:=(\xi_i^{-1}):\ol{\UU}\to \ol{\UU}.
$$
Clearly, $\xi:\ol{\UU}\to\ol{\UU}$ is a homeomorphism with inverse $\xi^{-1}$, and satisfies $\xi(\Ga)=\Ga$ and $\xi(\UU)=\UU$.

Define a new process $X_{t}=(X^{i}_{t})$ by setting
$$
X^i_t:=\xi_i(Z^i_t),\quad i\in \{1,\dots, d\},\quad\text{or simply},\quad X_{t}=\xi(Z_{t}),\quad t\geq 0.
$$
It is clear that $\Ga$ is also an absorbing set for the process $X_{t}$, and $X_{t}$ reaches $\Ga$ in finite time almost surely. Moreover, QSDs of $Z_{t}$ and $X_{t}$ are in an one-to-one correspondence as shown in the next result whose proof is straightforward.

\begin{prop}\label{prop-equi-qsd}
Let $\mu$ be a Borel probability measure on $\UU$. Then, $\mu$ is a QSD of $Z_{t}$ if and only if $\xi_*\mu$ is a QSD of $X_{t}$, where $\xi_*$ is the pushforward operator induced by $\xi$. Moreover, $\mu$ and $\xi_*\mu$ have the same extinction rates.
\end{prop}

It\^o's formula gives
\begin{equation}\label{sde-trans}
\begin{split}
\md X_{t}^i=\left[p_i(X_{t})-q_i(X_{t}^i)\right]\md t+\md W_t^i,\quad i\in \{1,\dots, d\}\quad\text{in}\quad\UU,
\end{split}
\end{equation}
where $p_i:\UU\to \R$ and $q_i:(0,\infty)\to \R$ are given by
\begin{equation*}\label{definition-p-and-q}
    p_i(x):=\frac{b_i(\xi^{-1}(x))}{\sqrt{a_i(\xi^{-1}_i(x_i))}}\quad \andd\quad q_i(x_i):=\frac{a'_i(\xi^{-1}_i(x_i))}{4\sqrt{a_i(\xi^{-1}_i(x_i))}},\quad x=(x_{i})\in \UU.
\end{equation*}
Denote by $\LL^{X}_{\bf FP}$ the Fokker-Planck operator associated to \eqref{sde-trans}, namely,
\begin{equation*}\label{FP-operator-for-X-process}
\LL^{X}_{\bf FP}v=\frac{1}{2}\De v-\nabla\cdot\left((p-q)v \right)\quad \text{in}\quad \UU,
\end{equation*}
where $p=(p_{i})$ and $q=(q_{i})$. Then, Proposition \ref{prop-5-7-1} has a counterpart for QSDs of $X_{t}$.

\begin{prop}\label{prop-2020-05-19}
Assume {\bf (H1)}-{\bf (H2)}. Let $\nu$ be a QSD of $X_{t}$ with extinction rate $\la_{1}$. Then, $\nu$ admits a positive density $v\in C^2(\UU)$ that satisfies $-\LL^{X}_{\bf FP}v=\la_1 v$ in $\UU$.
\end{prop}

\begin{rem}\label{rem-X-t-sde}
Note that the process $X_{t}$ and the process generated by solutions of \eqref{sde-trans} are not really the same, as \eqref{sde-trans} is only defined in $\UU$. However, the two processes agree as long as $X_{t}$ stays in $\UU$. More precisely, if we denote by $S_{\Ga}$ the first time that $X_{t}$ reaches $\Ga$, that is, $S_{\Ga}=\inf\left\{t\geq0:X_{t}\in\Ga\right\}$, then $X_{t}$ satisfies \eqref{sde-trans} on the event $[t<S_{\Ga}]$.
\end{rem}

As indicated by Proposition \ref{prop-2020-05-19},  QSDs of $X_{t}$ are closely related to positive eigenfunctions of $-\LL_{\bf FP}^{X}$, and therefore, it is natural to investigate the associated eigenvalue problem, namely,
\begin{equation}\label{fpe-trans}
-\LL^{X}_{\bf FP}v=\la v\quad \text{in}\quad \UU.
\end{equation}
Note that the operator $\LL_{\bf FP}^{X}$ is uniformly elliptic in $\UU$, but the functions $q_i$, $i\in\{1,\dots,d\}$ appearing in its first-order terms satisfy $q_{i}(x_{i})\to\infty$ as $x_{i}\to0^{+}$ for each $i\in\{1,\dots,d\}$. Such blow-up singularities make the investigation of the above eigenvalue problem very hard. In the following, we transform \eqref{fpe-trans} into the eigenvalue problem of another elliptic operator that has blow-up singularities only in the zeroth-order term and thus is easier to deal with.

Set
\begin{equation}\label{definition-U}
    U:=V\circ\xi^{-1}\quad\text{in}\quad\UU,
\end{equation}
where $V$ is given in {\bf (H3)}, and
\begin{equation}\label{eqn-function-Q}
    Q(x):=\sum_{i=1}^d\int_{1}^{x_{i}}2q_{i}(s)ds=\frac{1}{2}\sum_{i=1}^d \left[ \ln a_i(\xi_i^{-1}(x_i))-\ln a_i(\xi_i^{-1}(1))\right],\quad x\in \UU.
\end{equation}
For each $\beta>0$, we use the Liouville-type transform to define the differential operator
$$
\LL_{\beta}:=e^{\frac{Q}{2}+\beta U}\LL_{\bf FP}^{X}e^{-\frac{Q}{2}-\beta U}.
$$
It is straightforward to check that
\begin{equation}\label{main-operator}
\LL_{\be}=\frac{1}{2}\De-(p+\be\nabla U)\cdot\nabla-e_{\be}\quad\text{in}\quad \UU,
\end{equation}
where
\begin{equation}\label{eqn-5-8-1}
    e_{\be}=\frac{1}{2}\left(\be\De U-\be^2|\nabla U|^2\right)-\be p\cdot\nabla U+ \frac{1}{2}\sum_{i=1}^d (q_i^2-q'_i)-p\cdot q +\nabla \cdot p.
\end{equation}
Note that the coefficient of the first-order term $-(p+\be\nabla U)$ behaves nicely near $\Ga$, and the term $\frac{1}{2}\sum_{i=1}^d (q_i^2-q'_i)-p\cdot q$ blows up at the boundary $\Ga$, but it appears in the zeroth-order term of $\LL_{\beta}$.

The following proposition establishes the ``equivalence" between the eigenvalue problem \eqref{fpe-trans} and  the eigenvalue problem associated to the operator $\LL_{\beta}$.

\begin{prop}\label{prop-2020-05-10}
Suppose $v\in W^{2,1}_{loc}(\UU)$ and $\la\in\R$. Set $\tilde{v}:=ve^{\frac{Q}{2}+\be U}$. Then, $(v,\la)$ solves \eqref{fpe-trans} if and only if $-\LL_{\beta}\tilde{v}=\la\tilde{v}$ in $\UU$.
\end{prop}

According to Proposition \ref{prop-2020-05-10}, the investigation of QSDs of $X_{t}$ is reduced to the exploration of the principal spectral theory of $-\LL_{\be}$ (with a fixed $\beta$), something which we will do by choosing an appropriate function space.



\subsection{Approximation by first exit times}\label{subsec-first-exit-times}

Let $\{\UU_n\}_{n\in\N}$ be a sequence of arbitrarily fixed bounded, connected and open sets in $\UU$ with $C^2$ boundaries that satisfy
$$
\UU_{n}\stst\UU_{n+1}\stst\UU, \quad\forall n\in\N\quad\text{and}\quad \UU=\bigcup_{n\in\N}\UU_n.
$$
For each $n\in\N$, denote by $\tau_{n}$ the first  time that $X_{t}$ exits $\UU_{n}$, namely,
$$
\tau_n=\inf\left\{t\geq 0:X_t\not\in \UU_n\right\}.
$$

Recall that $S_{\Ga}$ is the first time that $X_{t}$ hits $\Ga$. The following result turns out to be useful.
\begin{lem}\label{lem-exit-times-approximation}
Assume {\bf(H1)}-{\bf(H3)}. For each $x\in\UU$, one has $\P^{x}\left[\lim_{n\to \infty} \tau_n=S_{\Ga}\right]=1$ and
\begin{equation*}\label{eqn-6-27-1}
\lim_{n\to \infty}\E^x\left[f(X_t)\mathbbm{1}_{\{t<\tau_n\}}\right]=\E^x\left[f(X_t)\mathbbm{1}_{\{t<S_{\Ga}\}}\right],\quad\forall f\in C_{b}(\UU).
\end{equation*}
\end{lem}
\begin{proof}
Fix $x\in\UU$. Obviously, $\tau_n<\tau_{n+1}$ for each $n\in\N$. Set $\tau:=\lim_{n\to\infty}\tau_n$. The first conclusion follows if we show $\P^x[\tau=S_{\Ga}]=1$.

Clearly, $\tau_n<S_{\Ga}$ for each $n\in\N$, leading to $\tau\leq S_{\Ga}$.
Since $X_t=\xi(Z_t)$ for $t\geq 0$, we find from Proposition \ref{prop-absorbing} that $\P^x[S_{\Ga}<\infty]=\P^{\xi^{-1}(x)}[T_{\Ga}<\infty]=1$. Therefore, $\P^x[\tau<\infty]=1$.

Noting that arguments in the proof of Proposition \ref{prop-absorbing} ensure that $Z_t$ and $X_t$ do not explode in finite time, we derive  $|X_{\tau}|=\lim_{n\to\infty} |X_{\tau_n}|<\infty$. Moreover, since $X_{\tau_n}\in \pa\UU_n$ and $\UU=\cup_{n\in\N}\UU_n$, it follows that $X_{\tau}\in \Ga$. As $S_{\Ga} $ is the first hitting time of  the boundary $\Ga$ and $\tau\leq S_{\Ga}$, one has $\tau=S_{\Ga}$.

Since $\tau_n$ increases to $S_{\Ga}$ $\P$-a.s., we find  $\lim_{n\to\infty}\mathbbm{1}_{\{t<\tau_n\}}=\mathbbm{1}_{\{t<S_{\Ga}\}}$ for each $t\geq 0$. The second conclusion then follows from the dominated convergence theorem. This completes the proof.
\end{proof}

\section{\bf Spectral theory and semigroup}\label{sec-spectrum-semigroup}

In this section we study for some appropriately fixed $\beta$ the spectral theory of $-\LL_{\be}$ in an appropriate function space. We also analyze the semigroup generated by $\LL_{\beta}$. In Subsection \ref{subsec-alpha} we define a weighted Hilbert space. In Subsection \ref{subsec-some-estimates} we derive some important estimates and meanwhile fix a special $\beta$, denoted by $\beta_{0}$. In Subsection \ref{spectrum} we study the (principal) spectral theory of $-\LL_{\beta_{0}}$ and the semigroup generated by $\LL_{\beta_{0}}$. In Subsection \ref{subsec-adjoint-operator} the spectral theory of $-\LL_{\beta_{0}}^{*}$,  where $\LL_{\beta_{0}}^{*}$ is the adjoint operator of $\LL_{\beta_{0}}$, and the semigroup generated by $\LL_{\beta_{0}}^{*}$ are investigated.

\subsection{A weighted Hilbert space}\label{subsec-alpha}

For $\de\in(0,1)$, let
$$
\Ga_{\de}:=\left\{x=(x_{i})\in\UU: x_i\leq \de \text{ for some } i\in\{1,\dots, d\}\right\}.
$$

It is easy to see from {\bf (H3)}(1)  that there exists $R_0\gg 1$ such that $-(b\cdot\nabla_{z}V)\circ\xi^{-1}>0$ in $\UU\sm B^+_{R_0}$, where we recall $B^+_R=\left\{x=(x_{i})\in\UU: x_{i}\in(0,R),\,\,\forall i\in\{1,\dots, d\}\right\}$ for $R>0$. Fix some $\de_{0}\in(0,1)$. Let $\al:\UU\to\R$ be defined by
\begin{equation}\label{eqn-3-28-2}
\al(x):=\begin{cases}
\displaystyle\sum_{i=1}^d \max\left\{\frac{1}{x_i^2},1\right\},&  x\in \Ga_{\de_0}\bigcap B^+_{R_0},\\
\displaystyle\sum_{i=1}^d \max\left\{\frac{1}{x_i^2},1\right\}-(b\cdot\nabla_{z} V)(\xi^{-1}(x)),&  x\in \Ga_{\de_0}\bigcap (\UU\sm B^+_{R_0}),\\
-(b\cdot\nabla_{z}V)(\xi^{-1}(x)),&  x\in (\UU\sm\Ga_{\de_0})\bigcap (\UU\sm B^+_{R_0}),\\
1,& \text{otherwise}.
\end{cases}
\end{equation}
Obviously, $\inf_{\UU}\al>0$. This $\al$ is defined according to the behavior of the coefficients of $-\LL_{\beta}$ near $\Ga$ and $\infty$. Its significance is partially reflected in Lemma \ref{lem-3-24-2} below. See Remark \ref{rem-5-8-1} after Lemma \ref{lem-3-24-2} for more comments.

Denote by $\HH^{1}(\UU;\C)$ the space of all weakly differentiable complex-valued functions $\phi:\UU\to \C$ satisfying
$$
\|\phi\|_{\HH^{1}}:=\left(\int_{\UU}\al|\phi|^2\md x \right)^{\frac{1}{2}}+\left( \int_{\UU} |\nabla \phi|^2\md x\right)^{\frac{1}{2}}<\infty.
$$
It is not hard to verify that $\HH^{1}(\UU;\C)$ is a Hilbert space  with the inner product:
$$
\lan\phi,\psi\ran_{\HH^{1}}:=\int_{\UU}\al\phi\ol{\psi}\md x+\int_{\UU}\nabla \phi\cdot\nabla\ol{\psi}\md x,\quad \forall \phi\andd\psi\in\HH^{1}(\UU;\C),
$$
where $\ol{\psi}$ denotes the complex conjugate of $\psi$.

\begin{lem}\label{lem-3-28-1}
Assume {\bf(H3)}. Then, $\HH^{1}(\UU;\C)$ is compactly embedded into $L^2(\UU;\C)$.
\end{lem}
\begin{proof}
Let $\{\phi_n\}_{n\in\N}\subset\HH^{1}(\UU;\C)$ satisfy $\sup_{n\in\N}\|\phi_n\|_{\HH^{1}}\leq 1$. Fix $R>0$. Since $H^1(B^+_R;\C)$ is compactly embedded into $L^2(B^+_R;\C)$, there is a subsequence, still denoted by $\{\phi_{n}\}_{n\in\N}$, and a measurable function $\phi_R\in L^2(B_R^+;\C)$, such that $\phi_{n}(x)\to \phi_R(x)$  for a.e. $x\in B_{R}^{+}$ and $\lim_{n\to \infty}\int_{B^+_R}|\phi_{n}-\phi_R|^2\md x=0$.

Let $\{R_m\}_{m}\subset(0,\infty)$ satisfy $R_m\to\infty$ as $m\to\infty$. Then, the above results hold for each $R_m$ in place of $R$. We apply the standard diagonal argument to find a subsequence, still denoted by $\{\phi_{n}\}_{n\in\N}$, and a measurable function $\phi:\UU\to\C$ such that $\phi_n\to \phi$ a.e. in $\UU$ as $n\to\infty$, and
\begin{equation}\label{eqn-3-30-4}
\lim_{n\to \infty}\int_{B^+_R}|\phi_{n}-\phi|^2\md x=0,\quad\forall R>0.
\end{equation}

Applying Fatou's lemma, we find $\int_{\UU}\al \phi^2\md x\leq \liminf_{n\to\infty}\int_{\UU}\al \phi_{n}^2\md x\leq 1$. It follows from \eqref{eqn-3-30-4} that
\begin{equation*}
\limsup_{n\to\infty} \int_{\UU} |\phi_{n}-\phi|^2\md x\leq  \limsup_{n\to\infty}\int_{\UU\sm B^+_R}|\phi_{n}-\phi|^2\md x,\quad\forall R>0.
\end{equation*}
Note that
\begin{equation*}
\begin{split}
\int_{\UU\sm B^+_R}|\phi_{n}-\phi|^2\md x
\leq \frac{2}{\inf_{\UU\sm B^+_R}\al}\int_{\UU\sm B^+_R}\al \left(\phi_{n}^2+\phi^2\right)\md x\leq \frac{2}{\inf_{\UU\sm B^+_R}\al},
\end{split}
\end{equation*}
which together with the fact $\al (x)\to\infty$ as $|x|\to\infty$ yields $\limsup_{n\to\infty} \int_{\UU\sm B^+_R} |\phi_{n}-\phi|^2\md x=0$, and hence, $\lim_{n\to\infty} \int_{\UU} |\phi_{n}-\phi|^2\md x=0$. This completes the proof.
\end{proof}


\subsection{Some estimates}\label{subsec-some-estimates}


We recall from \eqref{eqn-5-8-1} the definition of $e_{\be}$ and define for $N\geq 1$,
\begin{equation}\label{e-beta-N}
\begin{split}
e_{\be,N}:&=e_{\be}-\frac{N-1}{N}(\nabla\cdot p+\be \De U)\\
&=\left(\frac{1}{N}-\frac{1}{2}\right)\be\De U-\frac{\be^2}{2}|\nabla U|^2-\be p\cdot\nabla U+\frac{1}{2}\sum_{i=1}^d(q_i^2-q_i')-p\cdot q+\frac{\nabla\cdot p}{N}.
\end{split}
\end{equation}

Obviously, $e_{\be,1}=e_{\be}$ for all $\be>0$. The main reason for introducing $e_{\beta,N}$ is that they arise naturally in deriving a priori estimates for both sesquilinear forms and partial differential equations related to $\LL_{\be}$ or its adjoint (see Lemma \ref{lem-3-25-1} and Lemma \ref{lem-4-21-2}).


\begin{lem}\label{lem-3-24-2}
Assume {\bf (H1)-(H3)}. Then, the following hold.
\begin{enumerate}
\item[\rm(1)] There exists $C>0$ such that
\begin{equation*}
|\nabla U|^2+|p|^2\leq C \al \quad \text{in}\quad \UU,
\end{equation*}
where $\al$ is defined in \eqref{eqn-3-28-2}.

\item[\rm(2)] For each $\be>0$, there exists $C(\be)>0$ such that
\begin{equation*}
|e_{\be,N}|\leq C(\be)\al \quad \text{in}\quad \UU,\quad\forall N\geq 1.
\end{equation*}

\item[\rm(3)] There exist positive constants $\be_0$, $M$ and $C_*$ such that,
$$
e_{\be_0,N}+M\geq C_*\al \quad \text{in}\quad \UU,\quad\forall N\geq 1.
$$
\end{enumerate}
\end{lem}

Since the proof of this lemma is long and relatively independent, we postpone it to Appendix \ref{appendix-1} for the sake of readability.

\begin{rem}\label{rem-5-8-1}
Note that $e_{\beta,1}=e_{\beta}$ is the zeroth-order term of the operator $\LL_{\beta}$ (see \eqref{main-operator}) that has blow-up singularities at $\Ga$ as mentioned earlier. Lemma \ref{lem-3-24-2} says in particular that $e_{\beta}$ is well-controlled by the weight function $\al$, laying the foundation for our analysis.

In what follows, the positive constants $\be_0$, $M$ and $C_*$ are fixed such that the conclusion in Lemma \ref{lem-3-24-2} (3) holds.
\end{rem}


\subsection{Spectrum and semigroup}\label{spectrum}

We investigate the spectral theory of  $-\LL_{\be_{0}}$ and the semigroup generated by $\LL_{\be_{0}}$. Corresponding results are stated in Theorem \ref{thm-3-28-2} and Theorem \ref{thm-4-26-2}.

Denote by $\EE_{\beta_{0}}:\HH^{1}(\UU;\C)\times\HH^{1}(\UU;\C)\to \C$ the sesquilinear form associated to $-\LL_{\be_{0}}$, namely,
$$
\EE_{\be_{0}}(\phi,\psi)=\frac{1}{2}\int_{\UU}\nabla\phi\cdot\nabla\ol{\psi}\md x+\int_{\UU}(p+\be_{0}\nabla U)\cdot\nabla\phi\ol{\psi}\md x+\int_{\UU}e_{\be_{0}}\phi\ol{\psi}\md x,\quad \forall \phi,\psi\in\HH^{1}(\UU;\C).
$$

The following lemma addresses the boundedness and ``coercivity" of $\EE_{\beta_{0}}$.

\begin{lem}\label{lem-3-25-1}
Assume {\bf(H1)}-{\bf(H3)}.
\begin{enumerate}
\item[\rm(1)] There exists $C>0$ such that $\left|\EE_{\be_{0}}(\phi,\psi)\right|\leq C \|\phi\|_{\HH^{1}} \|\psi\|_{\HH^{1}}$ for all $\phi,\psi\in\HH^{1}(\UU;\C)$.

\item[\rm(2)] For each $\phi=\phi_1+i\phi_2\in\HH^{1}(\UU;\C)$, we have
\begin{equation*}
\EE_{\be_{0}}(\phi,\phi)
=\frac{1}{2}\int_{\UU}|\nabla\phi|^2\md x+\int_{\UU} e_{\be_{0},2}|\phi|^2\md x+i\int_{\UU}(p+\be_{0}\nabla U)\cdot(\phi_1 \nabla \phi_2-\phi_2\nabla\phi_1)\md x,
\end{equation*}
where $e_{\be_{0},2}$ is defined in \eqref{e-beta-N}. In particular,
\begin{equation*}
\Re\EE_{\be_0}(\phi,\phi)+M\|\phi\|^2_{L^{2}}\geq \min\left\{\frac{1}{2},C_*\right\}\|\phi\|_{\HH^{1}}^2,\quad\forall \phi\in\HH^{1}(\UU;\C).
\end{equation*}
\end{enumerate}
\end{lem}
\begin{proof}
(1) Let $ \phi,\psi\in\HH^{1}(\UU;\C)$. Applying H\"older's inequality, we derive
\begin{equation*}
\begin{split}
|\EE_{\be_{0}}(\phi,\psi)|
&\leq \frac{1}{2}\left(\int_{\UU}|\nabla \phi|^2\md x \right)^{\frac{1}{2}} \left( \int_{\UU}|\nabla \psi|^2\md x \right)^{\frac{1}{2}}\\
&\quad +  \left(\int_{\UU}|\nabla \phi|^2\md x \right)^{\frac{1}{2}} \left(\int_{\UU}|p+\be_{0}\nabla U|^2|\psi|^2\md x \right)^{\frac{1}{2}}\\
&\quad+ \left(\int_{\UU}|e_{\be_{0}}||\phi|^2\md x \right)^{\frac{1}{2}} \left(\int_{\UU}|e_{\be_{0}}||\psi|^2 \md x\right)^{\frac{1}{2}}.
\end{split}
\end{equation*}

Since Lemma \ref{lem-3-24-2} (1) ensures the existence of $C>0$ such that
$$
\int_{\UU}|p+\be_{0}\nabla U|^2|\psi|^2\md x\leq C(1+\be_{0}^2)\int_{\UU}\al |\psi|^2\md x,
$$
the conclusion follows readily from Lemma \ref{lem-3-24-2} (2) and the definition of the norm $\|\cdot\|_{\HH^{1}}$.

(2) Let $\{\eta_n\}_{n\in\N}$ be a sequence of smooth functions on $\UU$ taking values in $[0,1]$ and satisfying
\begin{equation*}
\eta_n(x)=
\begin{cases}
1,&  x\in\left(\UU\sm \Ga_{\frac{2}{n}} \right)\bigcap B^+_{\frac{n}{2}}, \\
0,&  x\in \Ga_{\frac{1}{n}}\bigcup\left(\UU\sm B^+_{n} \right),
\end{cases}
\quad
\text{and}
\quad
|\nabla \eta_n(x)|\leq
\begin{cases}
2n,& x\in\Ga_{\frac{2}{n}}\sm \Ga_{\frac{1}{n}},\\
4,& x\in\left(\UU\sm \Ga_{\frac{2}{n}} \right)\bigcap \left( B^+_{n}\sm B^+_{\frac{n}{2}}\right).
\end{cases}
\end{equation*}

Fix $\phi\in \HH^1(\UU;\R)$. We calculate
\begin{equation}\label{eqn-3-21-2}
\begin{split}
\EE_{\be_{0}}(\phi,\eta_n^2\phi)&=\frac{1}{2}\int_{\UU}\nabla\phi\cdot\nabla(\eta^2_n\ol{\phi})\md x+\int_{\UU} (p+\be_{0}\nabla U)\cdot\nabla\phi(\eta_n^2\ol{\phi})\md x+\int_{\UU}e_{\be_{0}}\eta_n^2|\phi|^2\md x\\
&=\frac{1}{2}\int_{\UU}\eta^2_n |\nabla \phi|^2\md x+\int_{\UU}\eta_n\ol{\phi}\nabla\phi\cdot\nabla\eta_n\md x\\
&\quad+\int_{\UU} (p+\be_{0}\nabla U)\cdot\nabla\phi(\eta_n^2\ol{\phi})\md x+\int_{\UU}e_{\be_{0}}\eta_n^2|\phi|^2\md x\\
&=:I_1(n)+I_2(n)+I_3(n)+I_4(n).
\end{split}
\end{equation}

We find $\lim_{n\to\infty}I_{1}(n)=\frac{1}{2}\int_{\UU} |\nabla \phi|^2\md x$ from $\int_{\UU}|\nabla\phi|^2\md x<\infty$ and the dominated convergence theorem.

Applying H\"older's inequality, we find
$|I_{2}(n)|\leq \left(\int_{\UU}\eta^2_n|\nabla \phi|^2\md x \right)^{\frac{1}{2}} \left( \int_{\UU}|\nabla \eta_n|^2|\phi|^2\md x\right)^{\frac{1}{2}}$. From the construction of $\eta_n$, we see that
\begin{equation*}
|\nabla \eta_n|^2|\phi|^2\leq
\begin{cases}
4n^2|\phi|^2&\quad\text{in}\quad\Ga_{\frac{2}{n}}\sm \Ga_{\frac{1}{n}},\\
16|\phi|^2&\quad\text{in}\quad \left(\UU\sm \Ga_{\frac{2}{n}} \right)\bigcap \left( B^+_{n}\sm B^+_{\frac{n}{2}}\right),\\
0&\quad \text{otherwise}.
\end{cases}
\end{equation*}
Since $n^2\leq \sum_{i=1}^d \max\left\{\frac{1}{x_i^2},1\right\}$ in $\Ga_{\frac{2}{n}}\sm \Ga_{\frac{1}{n}}$ for $n\geq 1$, the definition of $\al$ yields the existence of $C_1>0$ such that $|\nabla \eta_n|^2|\phi|^2\leq C_1\al|\phi|^2$ in $\UU$ for all $n\gg 1$. Since $\lim_{n\to \infty}|\nabla \eta_n|=0$, we apply the dominated convergence theorem to conclude
\begin{equation}\label{eqn-3-21-3}
\lim_{n\to\infty}\int_{\UU}|\nabla \eta_n|^2|\phi|^2\md x=0,
\end{equation}
which leads to $\lim_{n\to\infty} I_{2}(n)=0$.

Denote $\phi=\phi_1+i\phi_2$. We calculate for each $j\in\{1,\dots, d\}$,
\begin{equation*}
    (\pa_{j}\phi)\ol{\phi}=(\pa_{j}\phi_1+i\pa_{j}\phi_2)(\phi_1-i\phi_2)=\frac{1}{2}\pa_{j} |\phi|^2+i(\phi_1\pa_{j}\phi_2-\phi_2\pa_{j}\phi_1).
\end{equation*}
Integration by parts yields
\begin{equation*}
\begin{split}
I_{3}(n)&=\frac{1}{2}\int_{\UU} (p+\be_{0}\nabla U)\cdot\eta_n^2\nabla|\phi|^2\md x+i\int_{\UU}(p+\be_{0}\nabla U)\cdot\eta_n^2(\phi_1\nabla\phi_2-\phi_2\nabla\phi_1)\md x\\
&=\frac{1}{2}\int_{\UU} (p+\be_{0}\nabla U)\cdot\nabla(\eta^2_n|\phi|^2)\md x-\int_{\UU} (p+\be_{0}\nabla U)\cdot\nabla\eta_n (\eta_n|\phi|^2)\md x\\
&\qquad+i\int_{\UU}(p+\be_{0}\nabla U)\cdot\eta_n^2(\phi_1\nabla\phi_2-\phi_2\nabla\phi_1)\md x\\
&=-\frac{1}{2}\int_{\UU}(\nabla\cdot p+\be_{0}\De U) \eta_n^2|\phi|^2\md x-\int_{\UU} (p+\be_{0}\nabla U)\cdot\nabla \eta_n (\eta_n|\phi|^2)\md x\\
&\qquad+i\int_{\UU}(p+\be_{0}\nabla U)\cdot\eta_n^2(\phi_1\nabla\phi_2-\phi_2\nabla\phi_1)\md x.
\end{split}
\end{equation*}
It follows that
\begin{equation}\label{eqn-3-25-1}
\begin{split}
I_3(n)+I_4(n)
&=-\frac{1}{2}\int_{\UU}(\nabla \cdot p+\be_{0}\De U) \eta_n^2|\phi|^2\md x-\int_{\UU} (p+\be_{0} \nabla U)\cdot\nabla \eta_n (\eta_n|\phi|^2)\md x\\ &\quad+\int_{\UU}e_{\be_{0}}\eta_n^2|\phi|^2\md x+i\int_{\UU}(p+\be_0\nabla U)\cdot\eta_n^2(\phi_1\nabla\phi_2-\phi_2\nabla\phi_1)\md x\\
&=-\int_{\UU} (p+\be_{0}\nabla U)\cdot\nabla\eta_n (\eta_n|\phi|^2)\md x +\int_{\UU}e_{\be_{0},2} \eta_n^2|\phi|^2\md x\\
&\quad+i\int_{\UU}(p+\be_{0}\nabla U)\cdot\eta_n^2(\phi_1\nabla\phi_2-\phi_2\nabla\phi_1)\md x.
\end{split}
\end{equation}

We apply H\"older's inequality to find
\begin{equation*}
\begin{split}
\left|\int_{\UU} (p+\be_{0}\nabla U)\cdot\nabla\eta_n (\eta_n|\phi|^2)\md x\right|
&\leq\left(\int_{\UU}|p+\be_{0}\nabla U|^2 \eta_n^2|\phi|^2\md x \right)^{\frac{1}{2}} \left(\int_{\UU}|\nabla \eta_n|^2|\phi|^2\md x \right)^{\frac{1}{2}} \\
&\leq \left(\int_{\UU}|p+\be_{0}\nabla U|^2 |\phi|^2\md x \right)^{\frac{1}{2}} \left(\int_{\UU}|\nabla \eta_n|^2|\phi|^2\md x \right)^{\frac{1}{2}}.
\end{split}
\end{equation*}
Note that Lemma \ref{lem-3-24-2} (1) gives
$\int_{\UU}|p+\be_{0}\nabla U|^2 |\phi|^2\md x\leq C_2\int_{\UU}\al |\phi|^2\md x$ for some $C_2>0$, which together with \eqref{eqn-3-21-3}, yields
\begin{equation}\label{eqn-3-25-2}
\lim_{n\to\infty} \int_{\UU} (p+\be_{0}\nabla U)\cdot\nabla \eta_n (\eta_n|\phi|^2)\md x=0.
\end{equation}
It follows from Lemma \ref{lem-3-24-2} (2) that $|e_{\be_{0},2}|\eta_n^2|\phi|^2\leq C_3 \al |\phi|^2$ for some $C_3>0$. Together with the fact $\phi\in\HH^{1}(\UU;\R)$ and the dominated convergence theorem this yields
\begin{equation}\label{eqn-2020-05-11}
\lim_{n\to\infty} \int_{\UU}e_{\be_{0},2}\eta_n^2|\phi|^2\md x=\int_{\UU}e_{\be_{0},2}|\phi|^2\md x.
\end{equation}

Since Young's inequality gives
\begin{equation*}
    \begin{split}
    \left|(p+\be_{0}\nabla U)\cdot\eta_n^2(\phi_1\nabla\phi_2-\phi_2\nabla\phi_1)\right|&\leq \frac{1}{2}\eta_n^2|\nabla \phi|^2+\frac{1}{2}|p+\be_{0}\nabla U|^2\eta_n^2|\phi|^2\\
    &\leq \frac{1}{2}|\nabla \phi|^2+\frac{1}{2}|p+\be_{0}\nabla U|^2|\phi|^2,
    \end{split}
\end{equation*}
we deduce from the dominated convergence theorem that
$$
\lim_{n\to\infty}\int_{\UU}(p+\be_{0}\nabla U)\cdot\eta_n^2(\phi_1\nabla\phi_2-\phi_2\nabla\phi_1)\md x=\int_{\UU}(p+\be_{0}\nabla U)\cdot(\phi_1\nabla\phi_2-\phi_2\nabla\phi_1)\md x.
$$

Letting $n\to\infty$ in \eqref{eqn-3-25-1}, we conclude from \eqref{eqn-3-25-2}, \eqref{eqn-2020-05-11} and the above equality that
\begin{equation*}
\begin{split}
\lim_{n\to\infty}\left[I_{3}(n)+I_{4}(n)\right]=\int_{\UU}e_{\be_{0},2}\phi^2\md x+i\int_{\UU}(p+\be_{0}\nabla U)\cdot(\phi_1\nabla\phi_2-\phi_2\nabla\phi_1)\md x.
\end{split}
\end{equation*}

Passing to the limit $n\to\infty$ in \eqref{eqn-3-21-2}, we derive the expected identity from the limits of $I_{1}(n)$, $I_{2}(n)$, $I_{3}(n)$ and $I_{4}(n)$ as $n\to\infty$.

The last part of claim (2) is an immediate consequence of Lemma \ref{lem-3-24-2} (3).
\end{proof}

For $f\in L^2(\UU;\C)$, we consider the following problem:
\begin{equation}\label{eqn-elliptic}
(-\LL_{\be_0}+M)u=f\quad\text{in}\quad\UU,
\end{equation}
and look for solutions in $\HH^{1}(\UU;\C)$.

\begin{defn}
\emph{A function $u\in\HH^{1}(\UU;\C)$ is called a {\em weak solution} of \eqref{eqn-elliptic} if
$$
\EE_{\be_0}(u,\phi)+M\lan u,\phi\ran_{L^{2}}=\lan f,\phi\ran_{L^{2}},\quad \forall \phi\in\HH^{1}(\UU;\C),
$$
where $\lan\cdot,\cdot\ran_{L^{2}}$ is the usual inner product on $L^{2}(\UU;\C)$.
}
\end{defn}


\begin{lem}\label{thm-3-25-1}
Assume {\bf (H1)-(H3)}. Then, for any $f\in L^{2}(\UU;\C)$, \eqref{eqn-elliptic} admits a unique weak solution $u_{f}$ in $\HH^{1}(\UU;\C)$. Moreover, the following hold.
\begin{enumerate}
\item[\rm(1)] There is a constant $C>0$ such that $\|u_f\|_{\HH^{1}}\leq C\|f\|_{L^2}$ for all $f\in L^2(\UU;\C)$.

\item[\rm(2)] We have $u_{f}\in H_{loc}^{2}(\UU;\C)$ and
$$
(-\LL_{\be_0}+M)u_{f}=f\quad \text{a.e. in}\quad\UU,
$$
and
\begin{equation*}
    \EE_{\be_0}(u_{f},\phi)=\lan-\LL_{\be_0}u_{f},\phi\ran_{L^{2}},\quad\forall \phi\in\HH^{1}(\UU;\C).
\end{equation*}

\item[\rm(3)] If $f\in L^{2}(\UU;\C)$ satisfies $f\geq 0$ a.e. in $\UU$, then $u_f\geq 0$ a.e. in $\UU$. If in addition $f>0$ on a set of positive Lebesgue measure, then $u_f>0$ in $\UU$.
\end{enumerate}
\end{lem}
\begin{proof}
Fix $f\in L^2(\UU;\C)$. H\"older's inequality gives
\begin{equation}\label{eqn-4-18-3}
\left|\lan f,\phi\ran_{L^{2}}\right|\leq \left(\int_{\UU}\frac{1}{\al}|f|^2\md x\right)^{\frac{1}{2}} \left( \int_{\UU}\al |\phi|^2\md x\right)^{\frac{1}{2}}\leq \frac{1}{\left(\inf_{\UU} \al\right)^{\frac{1}{2}}}\|f\|_{L^2}\|\phi\|_{\HH^{1}},\quad\forall \phi\in\HH^{1}(\UU;\C).
\end{equation}
Hence, $\phi\mapsto\lan f,\phi\ran_{L^{2}}:\HH^{1}(\UU;\C)\to \C$ is a continuous linear functional.

By Lemma \ref{lem-3-25-1}  and the fact $\|\phi\|_{L^2}\leq \left(\inf_{\UU}\al\right)^{-\frac{1}{2}}\|\phi\|_{\HH^1}$ for $\phi\in \HH^1(\UU;\C)$ one has
\begin{equation*}\label{eqn-4-18-2}
 |\EE_{\be_0}(\phi,\psi)|+M|\lan\phi,\psi\ran_{L^{2}}|\leq C_1\|\phi\|_{\HH^{1}} \|\psi\|_{\HH^{1}},\quad\forall \phi\andd\psi\in\HH^{1}(\UU;\C)
\end{equation*}
for some $C_1>0$, and
\begin{equation}\label{eqn-4-18-2-1}
\Re\EE_{\be_0}(\phi,\phi)+M\|\phi\|^2_{L^{2}}\geq \min\left\{\frac{1}{2},C_*\right\}\|\phi\|_{\HH^{1}}^2,\quad\forall \phi\in\HH^{1}(\UU;\C).
\end{equation}
We apply the Lax-Milgram theorem (see e.g. \cite{GT01}) to find a unique $u_f\in\HH^{1}(\UU;\C)$ such that
\begin{equation}\label{eqn-5-20-1}
\EE_{\be_0}(u_f,\phi)+M\lan u_f,\phi\ran_{L^{2}}=\lan f,\phi\ran_{L^{2}},\quad\forall \phi\in\HH^{1}(\UU;\C).
\end{equation}
This shows that $u_f$ is the unique weak solution of \eqref{eqn-elliptic}.

\medskip

(1) Setting $\phi=u$  in \eqref{eqn-5-20-1}, we derive from \eqref{eqn-4-18-3} and \eqref{eqn-4-18-2-1} that
$$
\min\left\{\frac{1}{2},C_*\right\}\|u_f\|^2_{\HH^{1}}\leq\Re\EE_{\be_0}(u_f,u_f)+M\|u_f\|^2_{L^{2}}\leq  \frac{1}{(\inf_{\UU} \al)^{\frac{1}{2}}}\|f\|_{L^2}\|u_f\|_{\HH^{1}}.
$$

\medskip
(2) The classical regularity theory of elliptic equations ensures $u_f\in H^2_{loc}(\UU;\C)$. Hence, $u_f$ is a strong solution.

Fix $\phi\in\HH^{1}(\UU;\C)$. Let $\{\eta_n\}_{n\in\N}$ be the sequence of non-negative functions constructed in the proof of Lemma \ref{lem-3-25-1}. Then, $0\leq \eta_n\leq 1$ for each $n\in\N$ and $\lim_{n\to\infty}\eta_n=1$ in $\UU$. Integration by parts gives
\begin{equation}\label{eqn-5-15-2}
\begin{split}
\lan-\LL_{\be_0}u_{f},\eta_n^2\phi\ran_{L^{2}}&=\frac{1}{2}\int_{\UU}\nabla u_{f}\cdot\nabla (\eta_n^2\ol{\phi})\md x+\int_{\UU}(p+\be_0\nabla U)\cdot\nabla u_{f} (\eta_n^2\ol{\phi})\md x+\int_{\UU}e_{\be_0}\eta_n^2 u_{f}\ol{\phi}\md x\\
&=\frac{1}{2}\int_{\UU}\eta_n^2\nabla u_{f}\cdot\nabla\ol{\phi}\md x+\int_{\UU}\eta_n\ol{\phi}\nabla u_{f}\cdot\nabla \eta_n\md x\\
&\quad+\int_{\UU}(p+\be_{0}\nabla U)\cdot\nabla u_{f} (\eta_n^2\ol{\phi})\md x+\int_{\UU}e_{\be_0}\eta_n^2 u_{f}\ol{\phi}\md x.
\end{split}
\end{equation}

As $\LL_{\be_0}u_{f}\in L^2(\UU;\C)$ and $\phi\in\HH^{1}(\UU;\C)\subset L^2(\UU;\C)$, we derive from the dominated convergence theorem that
$\lim_{n\to\infty}\lan-\LL_{\be_0}u_{f},\eta_n^2\phi\ran_{L^{2}}=\lan-\LL_{\be_0}u_{f},\phi\ran_{L^{2}}$. Since $\nabla u_{f},\nabla\ol{\phi}\in L^2(\UU;\C)$, the dominated convergence theorem ensures that
$\lim_{n\to\infty}\frac{1}{2}\int_{\UU}\eta_n^2\nabla u_{f}\cdot\nabla\ol{\phi}\md x=\frac{1}{2}\int_{\UU}\nabla u_{f}\cdot\nabla\ol{\phi}\md x$. Arguing as in the proof of Lemma \ref{lem-3-25-1}, we find $\lim_{n\to\infty}\int_{\UU}\eta_n\ol{\phi}\nabla u_{f}\cdot\nabla \eta_n\md x=0$. Thanks to Lemma \ref{lem-3-24-2}, we apply H\"older's inequality to find positive constants $C_2$ and $C_3$ such that
\begin{equation*}
\begin{split}
   \left|\int_{\UU}(p+\be_{0}\nabla U)\cdot\nabla u_{f} \ol\phi\md x\right|&\leq \left(\int_{\UU}|\nabla u_{f}|^2\md x\right)^{\frac{1}{2}} \left(\int_{\UU}|p+\be_{0}\nabla U|^2|\phi|^2\md x \right)^{\frac{1}{2}}\\
   &\leq C_2 \left( \int_{\UU}\al |u_{f}|^2\md x\right)^{\frac{1}{2}}\left( \int_{\UU}\al |\phi|^2\md x\right)^{\frac{1}{2}},
   \end{split}
\end{equation*}
and
\begin{equation*}
   \left|\int_{\UU}e_{\be_0}u_{f}\ol\phi\md x \right|\leq C_3\left( \int_{\UU}\al |u_{f}|^2\md x\right)^{\frac{1}{2}}\left( \int_{\UU}\al |\phi|^2\md x\right)^{\frac{1}{2}}.
\end{equation*}
Therefore, the dominated convergence theorem yields
$$
\lim_{n\to\infty}\int_{\UU}(p+\be_{0}\nabla U)\cdot\nabla u_{f} (\eta_n^2\ol\phi)\md x+\int_{\UU}e_{\be_0}\eta_n^2 u_{f}\ol\phi\md x=\int_{\UU}(p+\be_{0}\nabla U)\cdot\nabla u_{f}\ol\phi\md x+\int_{\UU}e_{\be_0}u_{f}\ol\phi\md x.
$$
Letting $n\to\infty$ in \eqref{eqn-5-15-2}, the result follows.

\medskip
(3) Suppose $f\geq 0$ a.e. in $\UU$. In this case, $u_{f}$ must be real-valued. It is easy to verify that $u_f^-\in\HH^{1}(\UU;\C)$.  Thanks to (2), we obtain $\EE_{\be_0}(u_f,u_f^-)+M\langle u_f, u_f^{-}\rangle_{L^2}=\lan f,u_f^-\ran_{L^{2}}\geq 0$. It follows from Lemma \ref{lem-3-25-1} (2) and the fact $\EE_{\be_0}(u_f,u_f^-)=-\EE_{\be_0}(u_f^-,u_f^-)$ and $\langle u_f, u_f^{-}\rangle_{L^2}=-\langle u_f^{-}, u_f^{-}\rangle_{L^2}$ that
\begin{equation*}
    \begin{split}
        \min\left\{\frac{1}{2},C_*\right\}\|u_f^-\|_{\HH^{1}}^2\leq \EE_{\be_0}(u_f^-,u_f^-)+M\|u_f^-\|_{L^{2}}^2\leq 0.
    \end{split}
\end{equation*}
This implies $u_f^-=0$, and hence that $u_f\geq 0$. The last part follows from the weak Harnack's inequality of weak solutions of elliptic equations (see e.g. \cite[Theorem 8.18]{GT01}).
\end{proof}

By Lemma \ref{thm-3-25-1} and Lemma \ref{lem-3-28-1}, the operator
$$
(-\LL_{\beta_{0}}+M)^{-1}:L^{2}(\UU;\C)\to L^{2}(\UU;\C),\quad f\mapsto u_{f}
$$
is linear, positive and compact. In light of Lemma \ref{thm-3-25-1}, we define the domain of $\LL_{\be_0}$ as follows:
\begin{equation*}
\begin{split}
\DD:&=(-\LL_{\beta_{0}}+M)^{-1}L^{2}(\UU;\C)=\left\{\phi\in \HH^{1}(\UU;\C):\LL_{\be_0}\phi\in L^2(\UU;\C)\right\}.
\end{split}
\end{equation*}



The next result collects basic spectral properties of $-\LL_{\be_0}$.

\begin{thm}\label{thm-3-28-2}
Assume {\bf (H1)}-{\bf (H3)}. Then, the following hold.
\begin{enumerate}
\item[\rm(1)] The operator $-\LL_{\be_0}$ has a discrete spectrum and is contained in $\{\la\in\C:\Re\la>-M\}$.

\item[\rm(2)] The number $\la_1:=\inf\left\{\Re\la:\la\in\si(-\LL_{\be_0})\right\}$ is a simple eigenvalue of $-\LL_{\be_0}$, and is dominating, in the sense that $\inf\left\{\Re\la:\la\in\si(-\LL_{\be_0})\sm\{\la_{1}\}\right\}>\la_{1}$.

\item[\rm(3)] The eigenspace of $\la_{1}$ is spanned over $\C$ by $\tilde{v}_1$ for some  $\tilde{v}_{1}\in\DD$ a.e. positive in $\UU$.
\end{enumerate}
\end{thm}
\begin{proof}
Due to compactness, $(-\LL_{\be_0}+M)^{-1}$ has a discrete spectrum. Since Lemma \ref{thm-3-25-1} (3) ensures the a.e. positivity of $(-\LL_{\be_0}+M)^{-1}\phi$ in $\UU$ for each $\phi\in L^{2}_{+}(\UU)\sm\{0\}$, the operator $(-\LL_{\be_0}+M)^{-1}$ is in particular positive and nonsupporting (in the language of I. Sawashima \cite{Sawashima64}). As a result, we are able to apply the results in \cite{Sawashima64} (also see \cite[Theeorem 2.3]{Mar70}) to conclude
\begin{itemize}
\item the spectral radius $r_{1}$ of $(-\LL_{\be_0}+M)^{-1}$ is a positive and simple eigenvalue of $(-\LL_{\be_0}+M)^{-1}$;

\item the eigenspace of $r_{1}$ is spanned over $\C$ by $\tilde{v}_1$ for some $\tilde{v}_1>0$ a.e. in $\UU$;

\item $r_{1}$ is dominating in the sense that  $\sup\left\{|\la|:\la\in\si((-\LL_{\be_0}+M)^{-1})\sm\{r_{1}\}\right\}<r_{1}$.
\end{itemize}
The theorem then follows from the spectral mapping theorem.
\end{proof}

\begin{rem}
We point out that the positive cone $L^{2}_{+}(\UU)$ has empty interior so that the celebrated Kre\v{\i}n-Rutman theorem \cite{KR50} for compact and strongly positive operators, often used to treat elliptic operators on bounded domains, does not apply here. Restricting $-\LL_{\beta_{0}}$ to a smaller space does not help as $\UU$ is unbounded.

The number $\la_{1}$ is often called \emph{the principal eigenvalue} of $-\LL_{\beta_{0}}$. So far, it is not clear whether $\la_{1}$ is positive. The positivity of $\la_{1}$ is shown later by means of the absorbing properties of the process $X_{t}$.
\end{rem}


The following result concerns the semigroup generated by $\LL_{\beta_{0}}$.

\begin{thm}\label{thm-4-26-2}
Assume {\bf (H1)}-{\bf(H3)}. Then, $(\LL_{\be_0}, \DD)$ generates a $C_0$-semigroup $(T_t)_{t\geq 0}$ on $L^2(\UU;\C)$. Moreover, $(T_t)_{t\geq 0}$ is positive (i.e., $T_{t}L^{2}_{+}(\UU)\subset L^{2}_{+}(\UU)$ for all $t\geq0$), extends to an analytic semigroup and is immediately compact.
\end{thm}
\begin{proof}
Note that it is equivalent to study the operator $\LL_{\be_0}-M$ with domain $\DD$. The proof is broken into two steps.

\medskip

\paragraph{\bf Step 1} We claim:
\begin{enumerate}
\item [(a)] $(\LL_{\be_0}-M, \DD)$ is closed and $\DD$ is dense in $L^2(\UU;\C)$;

\item [(b)] The resolvent set of $\LL_{\be_0}-M$ contains $(0,\infty)$, and for each $\la>0$ one has
\begin{equation*}
\|(\la+M-\LL_{\be_0})^{-1}\|_{L^2\to L^2}\leq \frac{1}{\la}.
\end{equation*}
\end{enumerate}
Thus, $(\LL_{\be_0}-M,\DD)$ generates a $C_0$-semigroup of contractions $\{T_{t}\}_{t\geq0}$ in $L^2(\UU;\C)$ by the Hille-Yosida theorem (see e.g. \cite{Pa83,EN00}). By Lemma \ref{thm-3-25-1} (3), this semigroup must be positive.

\medskip

(a) Obviously, $C_0^{\infty}(\UU;\C)\subset\DD$. The density of $\DD$ in $L^2(\UU;\C)$ follows readily. To see the closedness of $(\LL_{\be_0}-M,\DD)$, we take sequences of functions $\{\phi_n\}_{n\in\N}\subset\DD$ and $\{f_n\}_{n\in\N}\subset L^2(\UU;\C)$ that converge respectively to $\phi$ and $f$ in $L^2(\UU;\C)$ and satisfy $(\LL_{\be_0}-M)\phi_n=f_n$ for all $n\in\N$. By Lemma \ref{thm-3-25-1}, we find $\tilde{\phi}:=(\LL_{\be_0}-M)^{-1}f\in\DD$ and $\|\phi_n-\tilde{\phi}\|_{\HH^{1}}\leq C\|f_n-f\|_{L^2}$ for all $n\in\N$, where $C>0$ is independent of $n$. Passing to the limit as  $n\to\infty$, we conclude $\phi=\tilde{\phi}\in\DD$, and hence, $(\LL_{\be_0}-M)\phi=f$. This proves (a).

\medskip

(b) Fix $\la>0$. By Theorem \ref{thm-3-28-2}, $\la$ is in the resolvent set of $\LL_{\beta_{0}}-M$. To establish the upper bound of  $\|(\la+M-\LL_{\be_0})^{-1}\|_{L^2\to L^2}$, we let $f\in L^2(\UU;\C)$ and $u\in \DD$ be such that $(\la+M-\LL_{\be_0})u=f$. It follows from Lemma \ref{thm-3-25-1} (2) that
$\EE_{\be_0}(u,\phi)+(\la+M)\lan u,\phi\ran_{L^{2}}=\lan f,\phi\ran_{L^{2}}$ for all $\phi\in \HH^1(\UU;\C)$. As a result, Lemma \ref{lem-3-25-1} (2) ensures
$$
\min\left\{\frac{1}{2},C_*\right\}\|u\|_{\HH^{1}}^2+\la\|u\|_{L^{2}}^2\leq \Re\EE_{\be_0}(u,\phi)+\la\|u\|^2_{L^{2}}\leq \|f\|_{L^2}\|u\|_{L^2},
$$
yielding the expected upper bound.

\medskip

\paragraph{\bf Step 2}
To show that $(T_t)_{t\geq 0}$ extends to an analytic semigroup, we set
$$
S:=\left\{\lan-(\LL_{\be_0}-M)u,u\ran_{L^{2}}:u\in\DD, \,\,\|u\|_{L^2}=1\right\}.
$$
We claim there exists $\tha\in (0,\frac{\pi}{2})$ such that $S\subset \{\la\in \C: |\arg\la|\leq \tha\}$. Then, for fixed $\tha_*\in (\tha,\frac{\pi}{2})$,  $\Si_{\tha_*}:=\{\la\in \C:|\arg\la|>\tha_*\}\subset \C\sm \ol{S}$ and there is $C_1>0$ such that $d(\la, \ol{S})\geq C_1|\la|$ for all $\la \in \Si_{\tha_*}$.

Since $(-\infty,0)$ is contained in the resolvent of $\LL_{\be_0}-M$, an application of  \cite[Theorem 1.3.9]{Pa83} yields that $\Si_{\tha_*}$ is contained in the resolvent of $\LL_{\be_0}-M$ and
$$
\|(\la+M-\LL_{\be_0})^{-1}\|_{L^2\to L^2}\leq \frac{1}{d(\la, \ol{S})}\leq \frac{1}{C_1|\la|},\quad\forall\la\in \Si_{\tha_*}.
$$
As a result,  \cite[Theorem 2.5.2]{Pa83} enables us to extend $(T_t)_{t\geq 0}$ to an analytic semigroup. Moreover, as $(\LL_{\be_0}-M)^{-1}$ is compact by Lemma \ref{lem-3-28-1}, the immediate compactness of $(T_t)_{t\geq 0}$ follows from \cite[Theorem II.4.29]{EN00}.

It suffices to prove the claim. To do so, we fix $u\in \DD$. Note that Lemma \ref{thm-3-25-1} (2) gives
\begin{equation*}
\lan-(\LL_{\be_0}-M)u,u\ran_{L^{2}}=\EE_{\be_0}(u,u)+M\|u\|^2_{L^{2}}.
\end{equation*}
We see from Lemma \ref{lem-3-25-1} (2) that
\begin{equation*}
\Re\lan-(\LL_{\be_0}-M)u,u\ran_{L^{2}}=\Re\EE_{\be_0}(u,u)+M\|u\|^2_{L^{2}}\geq\min\left\{\frac{1}{2},C_*\right\}\|u\|_{\HH^1}.
\end{equation*}
Applying Young's inequality, we derive from  Lemma \ref{lem-3-24-2} that
\begin{equation*}
    \begin{split}
        \left|\Im\lan-(\LL_{\be_0}-M)u,u\ran_{L^{2}}\right|&=|\Im \EE_{\be_0}(u,u)|\\
        &=\left|\int_{\UU}(p+\be_0\nabla U)\cdot(u_1\nabla u_2-u_2\nabla u_1)\right|\\
        &\leq \frac{1}{2}\int_{\UU}|\nabla u|^2\md x+\frac{1}{2}\int_{\UU}|p+\be_0\nabla U|^2|u|^2\md x\\
        &\leq \frac{1}{2}\int_{\UU}|\nabla u|^2\md x+\frac{C_2}{2}\int_{\UU}\al |u|^2\md x,
    \end{split}
\end{equation*}
where $u:=u_1+iu_2$ and $C_2>0$ is independent of $u\in\DD$. Therefore,
$$
0\leq \frac{\left|\Im\lan-(\LL_{\be_0}-M)u,u\ran_{L^{2}}\right|}{\Re\lan-(\LL_{\be_0}-M)u,u\ran_{L^{2}}}\leq \frac{\frac{1}{2}+\frac{C_2}{2}}{\min\left\{\frac{1}{2},C_*\right\}}.
$$
The claim follows. This completes the proof.
\end{proof}



\subsection{Adjoint operator and semigroup}\label{subsec-adjoint-operator}

Let $(\LL^{*}_{\beta_{0}},\DD^*)$ be the adjoint operator of $(\LL_{\beta_{0}},\DD)$ in $L^2(\UU;\C)$. Then, $\DD^*$ is given by
$$
\DD^{*}:=\left\{w\in L^2(\UU;\C): \exists f\in L^2(\UU;\C)\,\,\text{s.t.}\,\,\langle w,\LL_{\be_0}\phi\rangle_{L^2}=\langle f,\phi\rangle_{L^2},\,\,\forall \phi\in \DD\right\}.
$$
For each $w\in \DD^*$, $\LL_{\be_0}^{*}w$ is  the unique element in $L^2(\UU;\C)$ such that $\langle w,\LL_{\be_0}\phi\rangle_{L^2}=\langle\LL_{\be_0}^*{w},\phi\rangle_{L^2}$ for all $\phi\in \DD$. Integration by parts yields
\begin{equation*}\label{eqn-5-19-1}
\LL^{*}_{\be_0}w=\frac{1}{2}\De w+\nabla\cdot\left((p+\be_0\nabla U) w\right)-e_{\be_0} w, \quad  w\in C_0^{\infty}(\UU;\C).
\end{equation*}

The following lemma summarizes some properties of the operator $-\LL_{\beta_{0}}^{*}$.

\begin{lem}\label{lem-characterization-domain}
Assume {\bf (H1)-(H3)}. Then, the following hold.
\begin{enumerate}
\item[\rm(1)] $\si(-\LL^{*}_{\be_0})=\si(-\LL_{\be_0})\subset\{\la\in\C:\Re\la>-M\}$.

\item[\rm(2)] $\DD^{*}=\left\{w\in \HH^{1}(\UU;\C):\LL^*_{\be_0}w\in L^2(\UU;\C) \right\}$.

\item[\rm(3)] For each $ \phi\in \HH^1(\UU;\C)$ and $w\in \DD^*$ one has
$\langle\phi,-\LL^*_{\be_0}w\rangle_{L^2}=\EE_{\be_0}(\phi, w)$.

\item[\rm(4)] $\la_1$ is a simple and dominating eigenvalue of $-\LL^*_{\be_0}$ with the associated eigenspace spanned over $\C$ by $\tilde{v}_{1}^{*}$ for some $\tilde{v}_{1}^{*}\in\DD^{*}$ a.e. positive in $\UU$.
\end{enumerate}
\end{lem}
\begin{proof}
(1) Note that $\si(-\LL^{*}_{\be_0})=\ol{\si(-\LL_{\be_0})}$. Since the spectrum of $-\LL_{\be_0}$ consists of eigenvalues due to Lemma \ref{thm-3-28-2} (1), and the coefficients of $-\LL_{\be_0}$ are real-valued, we have $\La\in\si(-\LL_{\be_0})$ if and only if $\ol{\La}\in\si(-\LL_{\be_0})$. Hence, $\ol{\si(-\LL_{\be_0})}=\si(-\LL_{\be_0})$, which leads to $\si(-\LL^{*}_{\be_0})=\si(-\LL_{\be_0})$.

(2) Let $\DD^{*}_{1}:=\left\{w\in \HH^{1}(\UU;\C):\LL^*_{\be_0}w\in L^2(\UU;\C) \right\}$. It is easy to verify that $\DD^{*}_{1}\subset\DD^{*}$. To prove the converse, we take $u^{*}\in\DD^{*}$. Then, there exists a unique $f\in L^2(\UU;\C)$ such that $-\LL^{*}_{\beta_{0}}u^{*}=f$. Fix $\la>0$. Setting $f^*:=f+(M+\la)u^*\in L^2(\UU;\C)$, we find $(-\LL^*_{\be_0}+M+\la)u^{*}=f^{*}$. It follows from (1) that $-\la$ belongs to the resolvent set of $-\LL_{\beta_{0}}^{*}+M$, and therefore,  $u^{*}=(-\LL^*_{\be_0}+M+\la)^{-1}f^*$.

Arguing as in the proof of Lemma \ref{thm-3-25-1}, we deduce $(-\LL_{\beta_{0}}^{*}+M+\la)^{-1}L^{2}(\UU;\C)\subset\DD^{*}_{1}$,
and hence, $u^{*}\in \DD^{*}_{1}$. It follows that $\DD^{*}=\DD_{1}^{*}$.

(3) Note that $\langle\phi,-\LL^*_{\be_0}w\rangle_{L^2}=\langle-\LL_{\be_0}\phi,w\rangle_{L^2}$ for all $\phi\in C_0^{\infty}(\UU;\C)$ and $w\in \DD^*$. It follows from Lemma \ref{thm-3-25-1} (2) that $\langle\phi,-\LL^*_{\be_0}w\rangle_{L^2}=\EE_{\be_0}(\phi, w)$ for all $\phi\in C_0^{\infty}(\UU;\C)$ and $w\in \DD^*$. Since $C_0^{\infty}(\UU;\C)$ is dense in $\HH^1(\UU;\C)$, the conclusion follows from approximation arguments as in the proof of Lemma \ref{thm-3-25-1} (2).

(4) This follows from (1) and arguments as in the proof of Theorem \ref{thm-3-28-2}.
\end{proof}

Denote by $(T^{*}_t)_{t\geq 0}$ the dual semigroup of $(T_t)_{t\geq 0}$. It is well-known (see e.g. \cite[Corollary 1.10.6]{Pa83}) that $(T^{*}_t)_{t\geq 0}$ is a $C_{0}$-semigroup with infinitesimal generator $(\LL^{*}_{\be_0},\DD^{*})$.

\begin{thm}\label{thm-5-22-1}
Assume {\bf(H1)}-{\bf(H3)}. Then, $(T^{*}_t)_{t\geq 0}$ is an analytic semigroup. Moreover, it is positive, i.e., $T^{*}_{t}L_{+}^{2}(\UU)\subset L_{+}^{2}(\UU)$ for all $t\geq0$, and immediately compact.
\end{thm}
\begin{proof}
Note that the resolvent set of $\LL_{\be_0}^*-M$ coincides with that of $\LL_{\be_0}-M$. Thanks to \cite[Theorem 2.5.2]{Pa83}, the conclusion is a straightforward consequence of the analyticity of $(T_t)_{t\geq 0}$ and the fact
$\|(\la+M-\LL_{\be_0}^*)^{-1}\|_{L^2\to L^2}=\|(\ol{\la}+M-\LL_{\be_0})^{-1}\|_{L^2\to L^2}$ for each $\la\in\C$ with $\Re\la>0$. The positivity and immediate compactness follow from arguments as in the proof of Theorem \ref{thm-4-26-2}.
\end{proof}

\section{\bf Stochastic representation of semigroups}\label{sec-stoch-repre}

In this section, we study the stochastic representation of the semigroup $(T^{*}_{t})_{t\geq0}$. Subsection \ref{subsec-stochastic-represent-bd-domain} and Subsection \ref{subsec-estimate-semigroup-bd} are respectively devoted to the stochastic representation and estimates of semigroups generated by $\LL_{\beta_{0}}^{*}$ restricted to bounded domains with zero Dirichlet boundary condition. In Subsection \ref{subsec-stochastic-representation}, we establish the stochastic representation for $(T^{*}_{t})_{t\geq0}$.

\subsection{Stochastic representation in bounded domains}\label{subsec-stochastic-represent-bd-domain}

Let $\Om\stst \UU$ be a bounded and connected subdomain with $C^2$ boundary. Denote by $\LL^{X}$ the diffusion operator associated to $X_{t}$ or \eqref{sde-trans}, namely,
$$
\LL^{X}=\frac{1}{2}\De+(p-q)\cdot\nabla.
$$

For each $N>1$, let $\LL^{X}_{N}|_{\Om}$ be $\LL^{X}$ considered as an operator in $L^{N}(\Om;\C)$ with domain $W^{2,N}(\Om;\C)\cap W^{1,N}_0(\Om;\C)$. It is well-known (see e.g. \cite{GT01,Pa83,EN00}) that the spectrum of  $-\LL^{X}_{N}|_{\Om}$ is discrete and contained in $\{\la\in\C:\Re\la >0\}$ and  $\LL^{X}_{N}|_{\Om}$ generates an analytic semigroup $(S^{(\Om,N)}_{t})_{t\geq0}$ of contractions on $L^{N}(\Om;\C)$ that satisfies $S_{t}^{(\Om,N)}L^{N}_{+}(\Om)\subset L^{N}_{+}(\Om)$ for all $t\geq0$. Moreover, the following stochastic representation holds: for each $f\in C(\ol{\Om};\C)$,
\begin{equation}\label{stochastic-representation}
S^{(\Om,N)}_{t}f(x)=\E^x\left[f(X_t)\mathbbm{1}_{\{t<\tau_{\Om}\}}\right],\quad \forall (x,t)\in \ol{\Om}\times[0,\infty),
\end{equation}
where $\tau_{\Om}:=\inf\{t\geq 0: X_{t}\not\in\Om\}$ is the first time that $X_{t}$ exits $\Om$.

For $N>1$, let $\LL_{\beta_{0}}^{*,N}|_{\Om}$ be $\LL^{*}_{\beta_{0}}$ considered as an operator in $L^{N}(\Om;\C)$ with domain $W^{2,N}(\Om;\C)\cap W^{1,N}_{0}(\Om;\C)$.

\begin{prop}\label{lem-5-3-3}
The following statements hold.
\begin{enumerate}
\item[\rm(1)] The spectrum of $-\LL_{\beta_{0}}^{*,N}|_{\Om}$ is discrete and is contained in $\{\la\in\C:\Re\la>0\}$.

\item[\rm(2)] $\LL_{\beta_{0}}^{*,N}|_{\Om}$ generates an analytic semigroup of contractions $(T_t^{(*,\Om,N)})_{t\geq 0}$ on $L^N(\Om;\C)$ that is positive, namely, $T_t^{(*,\Om,N)}L^{N}_{+}(\Om)\subset L^{N}_{+}(\Om)$ for all $t\geq0$.

\item[\rm(3)] For each $f\in L^{N}(\Om;\C)$ one has
$$
T_{t}^{(*,\Om,N)}\tilde{f}=e^{-\frac{Q}{2}-\be_0 U}S_{t}^{(\Om,N)}f,\quad\forall t\geq0,
$$
where $\tilde{f}:=e^{-\frac{Q}{2}-\be_0 U} f$.

\item[\rm(4)] For each $f\in C(\ol{\Om};\C)$ one has
$$
T_t^{(*,\Om,N)} \tilde{f}(x)=e^{-\frac{Q(x)}{2}-\be_0 U(x)}\E^x\left[f(X_t)\mathbbm{1}_{\{t<\tau_{\Om}\}}\right],\quad \forall (x,t)\in \ol{\Om}\times[0,\infty),
$$
where $\tilde{f}:=e^{-\frac{Q}{2}-\be_0 U} f$.

\item[\rm(5)] For any $N_1, N_2>1$, $T_t^{(*,\Om,N_1)}$ and $T_t^{(*,\Om,N_2)}$ coincide on $L^{N_1}(\Om;\C)\cap L^{N_2}(\Om;\C)$ for all $t\geq 0$.
\end{enumerate}
\end{prop}
\begin{proof}
Since straightforward calculations give $$
\LL_{\beta_{0}}^{*,N}|_{\Om}\tilde{f}=e^{-\frac{Q}{2}-\be_0 U}\LL^{X}_{N}|_{\Om}f,\quad \forall f\in W^{2,N}(\Om;\C)\cap W^{1,N}_{0}(\Om;\C),
$$
where $\tilde{f}:=e^{-\frac{Q}{2}-\be_0 U} f$, the conclusions (1)-(4) follow immediately from the corresponding properties of $\LL^{X}_N|_{\Om}$ and $(S^{(\Om,N)}_t)_{t\geq 0}$.

In particular, for any $N_1, N_2>1$ we have  $T_t^{(*,\Om,N_1)}\tilde{f}=T_t^{(*,\Om,N_2)}\tilde{f}$ for all $\tilde{f}\in C(\ol{\Om};\C)$. Statement (5) then follows from the density of $C(\Om;\C)$ in $L^N(\Om;\C)$ for any $N>1$.
\end{proof}


\subsection{Estimates of semigroups in bounded domains}\label{subsec-estimate-semigroup-bd}

We prove two useful lemmas concerning some estimates of the semigroup $(T^{(*,\Om,N)}_{t})_{t\geq0}$.

\begin{lem}\label{lem-4-21-2}
Let $N\geq 2$ and $\tilde{f}\in L^N(\Om)$. Then, $\tilde{w}:=T^{(*,\Om,N)}_{\bullet}\tilde{f}$ satisfies the following inequalities:
\begin{equation*}
\begin{split}
&\frac{1}{N}\int_{\Om} |\tilde{w}|^N(\cdot, t)\md x+\frac{N-1}{2}\int_{t_1}^t\int_{\Om} |\tilde{w}|^{N-2}|\nabla \tilde{w}|^2\md x\md s+C_*\int_{t_1}^t\int_{\Om}\al |\tilde{w}|^N\md x\md s\\
&\qquad \leq \frac{1+e^{NM(t-t_1)}}{N}\int_{\Om} |\tilde{w}(\cdot,t_1)|^N\md x,\quad\forall t>t_1\geq 0,
\end{split}
\end{equation*}
and
\begin{equation*}
\begin{split}
&\frac{1}{N}\int_{\Om}|\tilde{w}|^N(\cdot, t)\md x+\frac{N-1}{2}\int_{t_2}^{t}\int_{\Om}|\tilde{w}|^{N-2} |\nabla \tilde{w}|^2\md x\md s+C_*\int_{t_2}^t\int_{\Om}\al |\tilde{w}|^N\md x\md s\\
&\qquad \leq  \frac{2}{N(t_2-t_1) }\int_{t_1}^{t_2}\int_{\Om}|\tilde{w}|^N\md x\md s,\quad \forall t>t_2>t_1\geq 0.
\end{split}
\end{equation*}
\end{lem}
\begin{proof}
Fix $N\geq 2$ and $\tilde{f}\in L^N(\Om)$. Then, $\tilde{w}:=T^{(*,\Om,N)}_{\bullet}\tilde{f}$ satisfies
$$
\pa_t \tilde{w}=\LL^{*,N}_{\be_0}|_{\Om}\tilde{w}\quad\text{in}\quad \Om\times(0,\infty).
$$
Multiplying the above equation by $|\tilde{w}|^{N-2}\tilde{w}$ and integrating by parts, we find for $t>0$
\begin{equation}\label{eqn-4-28-2}
\begin{split}
&\int_{\Om}|\tilde{w}|^{N-2}\tilde{w}\pa_t \tilde{w}\md x\\
&\qquad=-\frac{N-1}{2}\int_{\Om}|\tilde{w}|^{N-2} |\nabla \tilde{w}|^2\md x-(N-1)\int_{\Om}(p+\be_0\nabla U)\cdot|\tilde{w}|^{N-2}\tilde{w}\nabla\tilde{w}\md x-\int_{\Om}e_{\be_0}|\tilde{w}|^N\md x\\
&\qquad=-\frac{N-1}{2}\int_{\Om}|\tilde{w}|^{N-2} |\nabla \tilde{w}|^2\md x-\frac{N-1}{N}\int_{\Om}(p+\be_0\nabla U)\cdot\nabla|\tilde{w}|^N\md x-\int_{\Om}e_{\be_0}|\tilde{w}|^N\md x\\
&\qquad=-\frac{N-1}{2}\int_{\Om} |\tilde{w}|^{N-2}|\nabla \tilde{w}|^2\md x+\frac{N-1}{N}\int_{\Om}(\nabla \cdot p+\be_0 \De U)|\tilde{w}|^N\md x-\int_{\Om}e_{\be_0}|\tilde{w}|^N\md x\\
&\qquad=-\frac{N-1}{2}\int_{\Om} |\tilde{w}|^{N-2}|\nabla \tilde{w}|^2\md x-\int_{\Om}e_{\be_0,N}|\tilde{w}|^N\md x.
\end{split}
\end{equation}

Since $|\tilde{w}|^{N-2}\tilde{w}\pa_t \tilde{w}=\frac{1}{N}\pa_t |\tilde{w}|^N$, we integrate the above equality on $[t_1,t]\subset[0,\infty)$ to derive
\begin{equation*}
\begin{split}
\frac{1}{N}\int_{\Om} |\tilde{w}|^N(\cdot, t)\md x&+\frac{N-1}{2}\int_{t_1}^t\int_{\Om}|\tilde{w}|^{N-2}|\nabla \tilde{w}|^2\md x\md s\\
&+\int_{t_1}^t\int_{\Om}e_{\be_0,N} |\tilde{w}|^N\md x\md s=\frac{1}{N}\int_{\Om} |\tilde{w}|^N(\cdot, t_1)\md x.
\end{split}
\end{equation*}

As Lemma \ref{lem-3-24-2} (3) gives $e_{\be_0,N}+ M\geq C_* \al $ for all  $N\geq 2$, we find
\begin{equation}\label{eqn-4-23-2-1}
\begin{split}
&\frac{1}{N}\int_{\Om} |\tilde{w}|^N(\cdot, t)\md x+\frac{N-1}{2}\int_{t_1}^t\int_{\Om} |\tilde{w}|^{N-2}|\nabla \tilde{w}|^2\md x\md s+C_*\int_{t_1}^t\int_{\Om}\al |\tilde{w}|^N\md x\md s\\
&\qquad \leq M\int_{t_1}^t\int_{\Om}|\tilde{w}|^N\md x\md s+ \frac{1}{N}\int_{\Om}  |\tilde{w}|^N(\cdot, t_1)\md x,\quad\forall t>t_1\geq 0.
\end{split}
\end{equation}

Setting $g(t):=\int_{t_1}^t\int_{\Om}|\tilde{w}|^N\md x\md s$ for $t\geq t_1$, we arrive at $\frac{1}{N}g'\leq Mg+\frac{1}{N}\|\tilde{w}(\cdot, t_1)\|^N_{L^N}$ for all $t>t_1$. Gronwall's inequality gives $g(t)\leq \frac{e^{NM(t-t_1)}}{NM}\|\tilde{w}(\cdot, t_1)\|^N_{L^N}$ for all $t>t_1$. Inserting this into \eqref{eqn-4-23-2-1} yields the first inequality.

Now, we prove the second inequality. Fix $t_1,t_2\in [0,\infty)$ with $t_1<t_2$. Let $\eta\in C^{\infty}((0,\infty))$ be non-negative and non-decreasing such that
\begin{equation*}
\eta(t)=
\begin{cases}
0,&\quad t\in[0,t_1],\\
1,& \quad t\in [t_2,\infty),
\end{cases}
\quad\andd\quad \max_{t\in[t_1,t_2]}\eta'(t)\leq \frac{2}{t_2-t_1}.
\end{equation*}

Multiplying \eqref{eqn-4-28-2} by $\eta$ and integrating by parts, we find for $t>t_2$,
\begin{equation*}
\begin{split}
&\frac{1}{N}\int_{\Om}\eta(t) |\tilde{w}|^N(\cdot, t)\md x -\frac{1}{N}\int_0^t\int_{\Om}\eta' |\tilde{w}|^N\md x\md s\\
&\qquad=-\frac{N-1}{2}\int_0^t\int_{\Om} \eta |\tilde{w}|^{N-2}|\nabla \tilde{w}|^2\md x\md s-\int_0^t\int_{\Om}\eta e_{\be_0,N}|\tilde{w}|^N \md x\md s.
\end{split}
\end{equation*}

The definition of $\eta$ then gives
\begin{equation*}
\begin{split}
&\frac{1}{N}\int_{\Om}|\tilde{w}|^N(\cdot, t)\md x+\frac{N-1}{2}\int_{t_2}^{t}\int_{\Om}|\tilde{w}|^{N-2} |\nabla \tilde{w}|^2\md x\md s+\int_{t_2}^{t}\int_{\Om}e_{\be_0,N} |\tilde{w}|^N\md x\md s\\
&\qquad\leq \frac{1}{N}\int_{t_1}^{t_2}\int_{\Om}\eta' |\tilde{w}|^N\md x\md s\leq \frac{2}{N(t_2-t_1) }\int_{t_1}^{t_2}\int_{\Om}|\tilde{w}|^N\md x\md s,\quad \forall t>t_2.
\end{split}
\end{equation*}
This completes the proof.
\end{proof}

\begin{lem}\label{lem-6-12-1}
For each $t>0$, there exists $C=C(t)$, independent of the domain $\Om$, such that
$$
\|T^{(*,\Om,2_*)}_t\tilde{f}\|_{ L^{2}(\Om)}\leq C \|\tilde{f}\|_{L^{2_*}(\Om)},\quad\forall \tilde{f}\in L^{2_*}(\Om),
$$
where $2_*:=\frac{2(d+2)}{d+4}\in (1,2)$ is the dual exponent of $2+\frac{4}{d}$.
\end{lem}
\begin{proof}
Take $N\in (1,2]$. Then, $N':=\frac{N}{N-1}\geq2$. Denote by $(T^{(\Om,N')}_t)_{t\geq 0}$ the semigroup on $L^{N'}(\Om)$ that is dual to $(T^{(*,\Om,N)}_t)_{t\geq 0}$. Let $\LL^{N'}_{\be_0}|_{\Om}$ be $\LL_{\be_0}$ considered as an operator in $L^{N'}(\Om)$ with domain $W^{2,N'}(\Om)\cap W_0^{1,N'}(\Om)$. It is not hard to check that $\LL^{N'}_{\be_0}|_{\Om}$, being $\LL_{\be_0}$ considered as an operator in $L^{N'}(\Om)$ with domain $W^{2,N'}(\Om;\C)\cap W^{1,N'}_0(\Om;\C)$, is the generator of $(T^{(\Om,N')}_t)_{t\geq 0}$.

Take $\tilde{g}\in L^{N'}(\Om)$ and denote $\tilde{v}:=T^{(\Om,N')}_{\bullet}\tilde{g}$. Then, $\tilde{v}$ is the solution of
$$
\pa_t \tilde{v}=\LL^{N'}_{\be_0}|_{\Om}\tilde{v}\quad\text{in}\quad \Om\times (0,\infty).
$$
Multiplying the above equation by $|\tilde{v}|^{N'-2}\tilde{v}$ and integrating by parts, we find for $t>0$,
\begin{equation*}
\begin{split}
&\int_{\Om}|\tilde{v}|^{N'-2}\tilde{v}\pa_t \tilde{v}\md x\\
&\qquad=-\frac{N'-1}{2}\int_{\Om}|\tilde{v}|^{N'-2} |\nabla \tilde{v}|^2\md x-\int_{\Om}(p+\be_0\nabla U)\cdot\nabla\tilde{v}|\tilde{v}|^{N'-2}\tilde{v}\md x-\int_{\Om}e_{\be_0}|\tilde{v}|^{N'}\md x\\
&\qquad=-\frac{N'-1}{2}\int_{\Om}|\tilde{v}|^{N'-2} |\nabla \tilde{v}|^2\md x+\frac{1}{N'}\int_{\Om}(\nabla \cdot p+\be_0\De U)|\tilde{v}|^{N'}\md x-\int_{\Om}e_{\be_0}|\tilde{v}|^{N'}\md x\\
&\qquad=-\frac{N'-1}{2}\int_{\Om}|\tilde{v}|^{N'-2} |\nabla \tilde{v}|^2\md x-\int_{\Om}e^*_{\be_0,N'}|\tilde{v}|^{N'}\md x,
\end{split}
\end{equation*}
where $e^*_{\be_0,N'}:=e_{\be_0}-\frac{1}{N'}(\nabla \cdot p+\be_0\De U)$. We can follow the proof of  Lemma \ref{lem-3-24-2} (3) to show
$e^*_{\be_0,N}+M\geq C_* \al$ in $\UU$ for all $N
\geq 1$. Then, arguing as in the proof of  Lemma \ref{lem-4-21-2} yields
\begin{equation}\label{eqn-6-12-1}
    \begin{split}
    &\frac{1}{N'}\int_{\Om}|\tilde{v}|^{N'}(\cdot, t)\md x+\frac{N'-1}{2}\int_0^t\int_{\Om}|\tilde{v}|^{N'-2}|\nabla\tilde{v}|^2\md x\md s+C_*\int_0^t\int_{\Om}\al|\tilde{v}|^{N'}\md x\md s\\
    &\qquad\leq \frac{1+e^{N'Mt}}{N'}\int_{\Om}|\tilde{g}|^{N'}\md x,\quad\forall t>0,
    \end{split}
\end{equation}
and
\begin{equation}\label{eqn-6-12-2}
\begin{split}
&\frac{1}{N'}\int_{\Om}|\tilde{v}|^{N'}(\cdot, t)\md x+\frac{N'-1}{2}\int_{t_2}^{t}\int_{\Om}|\tilde{v}|^{N'-2} |\nabla \tilde{v}|^2\md x\md s+C_*\int_{t_2}^t\int_{\Om}\al |\tilde{v}|^{N'}\md x\md s\\
&\qquad \leq  \frac{2}{N'(t_2-t_1) }\int_{t_1}^{t_2}\int_{\Om}|\tilde{v}|^{N'}\md x\md s,\quad \forall t>t_2>t_1\geq 0.
\end{split}
\end{equation}

The Sobolev embedding theorem gives
$$
\|\tilde{v}^{\frac{N'}{2}}\|_{L^{2\ka}(\Om\times[0,t])}\leq C_1 \left( \sup_{s\in[0,t]} \|\tilde{v}^{\frac{N'}{2}}(\cdot, s)\|_{L^2(\Om)}+\|\nabla\tilde{v}^{\frac{N'}{2}}\|_{L^2(\Om\times[0,t])}\right),
$$
where $\ka=\frac{d+2}{d}$ and $C_1>0$ depends only on $d$. This together with \eqref{eqn-6-12-1} gives rise to
\begin{equation*}
\begin{split}
\left(\int_0^t\int_{\Om}|\tilde{v}|^{\ka N'}\md x\md s\right)^{\frac{1}{\ka}}&=\|\tilde{v}^{\frac{N'}{2}}\|_{L^{2\ka}(\Om\times[0,t])}^{2}\\
&\leq 2C_1^2\left( \sup_{s\in[0,t]}\int_{\Om}|\tilde{v}(x,s)|^{N'}\md x+\frac{|N'|^2}{4}\int_0^t\int_{\Om}|\tilde{v}|^{N'-2}|\nabla \tilde{v}|^2\md x\md s\right)\\
&\leq 2C_1^2\left( 1+\frac{N'}{2(N'-1)}\right)(1+e^{N'Mt})\int_{\Om} |\tilde{g}|^{N'}\md x\\
&=C_{2}(1+e^{N'Mt})\int_{\Om} |\tilde{g}|^{N'}\md x,\quad\forall t>0,
\end{split}
\end{equation*}
where $C_{2}:=2C_1^2\left( 1+\frac{N'}{2(N'-1)}\right)$. We then deduce from \eqref{eqn-6-12-2} (with $\ka N'$ instead of $N'$) that
\begin{equation*}
\begin{split}
&\frac{1}{\ka N'}\int_{\Om} |\tilde{v}|^{\ka N'}(\cdot, t)\md x+\frac{\ka N'-1}{2}\int_{t_2}^t\int_{\Om} |\tilde{v}|^{\ka N'-2}|\nabla \tilde{v}|^2\md x\md s+C_*\int_{t_2}^t\int_{\Om}\al |\tilde{v}|^{\ka N'}\md x\md s\\
&\qquad\leq \frac{2}{\ka N'(t_2-t_1)}\int_{t_1}^{t_2}\int_{\Om}|\tilde{v}|^{\ka N'}\md x\md s \leq \frac{2}{\ka N'(t_2-t_1)} C_{2}^{\ka}(1+e^{N'Mt_2})^{\ka}\|\tilde{g}\|^{\ka N'}_{L^{N'}(\Om)}
\end{split}
\end{equation*}
for all $t>t_2>t_{1}\geq 0$, where we used \eqref{eqn-6-12-1} in the second inequality.

As a consequence, for each $t>0$, there exists $C_3=C_3(d,N',t)>0$ such that
$$
\|T^{(\Om,N')}_t\tilde{g}\|_{L^{\ka N'}(\Om)}=\|\tilde{v}(\cdot, t)\|_{L^{\ka N'}(\Om)}\leq C_3 \|\tilde{g}\|_{L^{N'}(\Om)}.
$$

Since $T^{(\Om,N')}$ and $T^{(*,\Om,N)}_t$ are adjoint to each other, it follows that
$$
\|T^{(*,\Om,N)}_t\tilde{f}\|_{ L^{N}(\Om)}\leq C_3 \|\tilde{f}\|_{L^{N_*}(\Om)},\quad\forall \tilde{f}\in L^{N_*}(\Om)\cap L^N(\Om).
$$
where $N_*:=\frac{\ka N'}{\ka N'-1}$. Thanks to Proposition \ref{lem-5-3-3} (5), we deduce from standard approximations that
$$
\|T^{(*,\Om,N_*)}_t\tilde{f}\|_{ L^{N}(\Om)}\leq C_3 \|\tilde{f}\|_{L^{N_*}(\Om)},\quad\forall \tilde{f}\in L^{N_*}(\Om).
$$
Setting $N=2$ yields $2_*=\frac{2(d+2)}{d+4}\in (1,2)$. This completes the proof.
\end{proof}


\subsection{Stochastic representation}\label{subsec-stochastic-representation}

We prove the following theorem concerning the stochastic representation of $(T^{*}_t)_{t\geq0}$.

\begin{thm}\label{thm-4-22-1}
Assume {\bf (H1)}-{\bf (H3)}. For each $f\in C_b(\UU;\C)$ satisfying $\tilde{f}:=fe^{-\frac{Q}{2}-\be_0 U}\in L^2(\UU;\C)$ we have
\begin{equation*}\label{eqn-4-20-1}
T^{*}_t\tilde{f}(x)=e^{-\frac{Q(x)}{2}-\be_0 U(x)}\E^x\left[f(X_t)\mathbbm{1}_{\{t<S_{\Ga}\}}\right],\quad\forall (x,t)\in\UU\times[0,\infty).
\end{equation*}
\end{thm}

Consider the following initial value problem associated to the operator $\LL^{*}_{\beta_{0}}$:
\begin{equation}\label{eqn-4-16-2-1}
\begin{cases}
\pa_t \tilde w=\frac{1}{2}\De \tilde w+\nabla\cdot\left((p+\be_0 \nabla U)\tilde w\right)-e_{\be_{0}}\tilde w \quad&\text{in}\quad \UU\times[0,\infty), \\
\tilde w(\cdot,0)=\tilde f\qquad\qquad\qquad\qquad\qquad\qquad\quad &\text{in}\quad \UU.
\end{cases}
\end{equation}

\begin{defn}\label{weak-soln-cauchy}
A function $\tilde{w}\in C(\UU\times[0,\infty))\bigcap L^2_{loc}([0,\infty),\HH^{1}(\UU))$ is called a {\em weak solution} of \eqref{eqn-4-16-2-1} if for each $\phi\in C^{1,1}_0(\UU\times[0,\infty))$ one has
\begin{equation*}
\begin{split}
    &\int_{\UU}\tilde{w}(\cdot,t)\phi(\cdot,t)\md x-\int_{\UU}\tilde{f}\phi(\cdot,0)\md x-\int_{0}^t\int_{\UU}\tilde{w}\pa_t \phi\md x\md s\\
    &\qquad=-\frac{1}{2}\int_0^t\int_{\UU}\nabla\tilde{w}\cdot\nabla \phi\md x\md s-\int_0^t\int_{\UU}(p+\be_0\nabla U)\cdot\tilde{w}\nabla\phi\md x\md s-\int_0^t\int_{\UU}e_{\be_0}\tilde{w}\phi\md x\md s
\end{split}
\end{equation*}
for all $t\in[0,\infty)$.
\end{defn}
\begin{lem}\label{lem-4-21-1}
Assume {\bf (H1)}-{\bf (H3)}. For each $\tilde f\in C(\UU)\cap L^2(\UU)$, \eqref{eqn-4-16-2-1} admits at most one weak solution.
\end{lem}

The proof of the above lemma follows from energy methods and approximation arguments. Since it is somewhat standard we present its proof in Appendix \ref{subsec-appendix-2}.

Now, we prove Theorem \ref{thm-4-22-1}.

\begin{proof}[Proof of Theorem \ref{thm-4-22-1}]

Treating the real and imaginary parts separately, we only need to prove the theorem for $f\in C_b(\UU)$ such that $\tilde{f}:=fe^{-\frac{Q}{2}-\be_0 U}\in L^2(\UU)$. Fix such an $f$.

We show $T^*_{\bullet}\tilde{f}$ is a weak solution of \eqref{eqn-4-16-2-1}. Due to the analyticity of $(T^{*}_t)_{t\geq 0}$ (see Theorem \ref{thm-5-22-1}) and Lemma \ref{lem-characterization-domain}, we find
\begin{enumerate}
    \item[\rm(1)] $T^*_{\bullet}\tilde{f}\in C([0,\infty),L^2(\UU))\bigcap C^1((0,\infty), L^2(\UU))$;

    \item[\rm(2)] $T^*_t\tilde{f}\in \DD^*\subset \HH^1(\UU)\bigcap H^2_{loc}(\UU)$ for all $t>0$;

    \item[\rm(3)] $\frac{\md}{\md t}T^*_t\tilde{f}=\LL_{\be_0}^*T^*_t\tilde{f}$ for all $t>0$.
\end{enumerate}
Since $\tilde{f}\in C(\UU)$, the classical regularity theory of parabolic equations yields that $T^*_{\bullet}\tilde{f}\in C(\UU\times[0,\infty))$. Applying Lemma \ref{lem-3-25-1} (2) and Lemma \ref{lem-characterization-domain}, we find for each $t>0$,
\begin{equation*}
\begin{split}
    \min\left\{\frac{1}{2}, C_*\right\}\|T^*_t\tilde{f}\|^2_{\HH^1}&\leq\EE_{\be_0}(T^*_t\tilde{f},T^*_t\tilde{f})+M\|T^*_t\tilde{f}\|^2_{L^{2}}\\
    &=-\langle T^*_t\tilde{f}, \LL^*_{\be_0}T^*_t\tilde{f} \rangle_{L^2}+M\|T^*_t\tilde{f}\|^2_{L^{2}}\\
    &=-\langle T^*_t\tilde{f}, \frac{\md}{\md t}T^*_t\tilde{f}\rangle_{L^2}+M\|T^*_t\tilde{f}\|^2_{L^{2}}\\
    &=-\frac{1}{2}\frac{\md}{\md t}\|T^*_t\tilde{f}\|_{L^2}^2+M\|T^*_t\tilde{f}\|^2_{L^{2}}.
\end{split}
\end{equation*}
It follows that
\begin{equation*}
    \min\left\{\frac{1}{2}, C_*\right\}\int_0^t\|T^*_s\tilde{f}\|^2_{\HH^1}\md s\leq\frac{1}{2} \|\tilde{f}\|_{L^2}^2+M\int_0^t\|T^*_s\tilde{f}\|_{L^2}^2\md s,\quad\forall t>0.
\end{equation*}
This yields $T^*_{\bullet}\tilde{f}\in L^2_{loc}([0,\infty),\HH^1(\UU))$.  By (3), it is easy to check that the integral identity in Definition \ref{weak-soln-cauchy} holds with $\tilde{w}$ replaced by $T_{\bullet}^*\tilde{f}$. As a consequence, $T^*_{\bullet}\tilde{f}$ is a weak solution of \eqref{eqn-4-16-2-1}.

Define
$$
\tilde{w}(x,t):=e^{-\frac{Q(x)}{2}-\be_0 U(x)}\E^x\left[f(X_t)\mathbbm{1}_{\{t<S_{\Ga}\}}\right],\quad (x,t)\in\UU\times[0,\infty).
$$
We claim that $\tilde{w}$ is also a weak solution of \eqref{eqn-4-16-2-1}. Then, Lemma \ref{lem-4-21-1} yields $T^{*}_{\bullet}\tilde{f}=\tilde{w}$, leading to the conclusion of the theorem.

The continuity of $\tilde{w}$ in $\UU\times[0,\infty)$ follows from the definition and continuity properties of $X_{t}$. We show $\tilde{w}\in L^2_{loc}([0,\infty),\HH^1(\UU))$. Let $\{\UU_{n}\}_{n\in\N}$ be as in Subsection \ref{subsec-first-exit-times}. It follows from Lemma \ref{lem-exit-times-approximation} and Proposition \ref{lem-5-3-3} (4) that
$\tilde{w}=\lim_{n\to\infty} T^{(*,\UU_n,2)}_{\bullet}\tilde{f}|_{\UU_n}$ in $\UU\times[0,\infty)$, where we recall from Subsection \ref{subsec-stochastic-represent-bd-domain} that $(T_t^{(*,\UU_n,2)})_{t\geq 0}$ is the positive analytic semigroup of contractions on $L^2(\UU_n;\C)$ generated by $\LL^{*,2}_{\be_0}|_{\UU_n}$ with domain $W^{2,2}(\UU_n;\C)\cap W_0^{1,2}(\UU_n;\C)$.

For convenience, we define $\tilde{w}_n:=T^{(*,\UU_n,2)}_{\bullet}\tilde{f}|_{\UU_n}$ for $n\in\N$. Then, $\lim_{n\to\infty}\tilde{w}_n=\tilde{w}$. Lemma \ref{lem-4-21-2} (with $t_1=0$) gives for each $t\in[0,\infty)$ and $n\in\N$,
\begin{equation*}
\begin{split}
&\frac{1}{2}\int_{\UU_n}\tilde{w}_n^2(\cdot, t)\md x+\frac{1}{2}\int_0^t\int_{\UU_n} |\nabla \tilde{w}_n|^2\md x\md s+
C_*\int_0^t\int_{\UU_n}\al \tilde{w}_n^2\md x\md s\leq \frac{1+e^{2Mt}}{2}\int_{\UU_n}\tilde{f}^2\md x.
\end{split}
\end{equation*}
Letting $n\to\infty$ yields $\tilde{w}\in L^2_{loc}([0,\infty),\HH^1(\UU))$. Since  $\pa_t\tilde{w}_n=\LL^{*,2}_{\be_0}|_{\UU_n}\tilde{w}_n$ for all $t>0$ and $n\in\N$, standard approximation arguments ensure that $\tilde{w}$ is a weak solution of \eqref{eqn-4-16-2-1}. This finishes the proof.
\end{proof}


\section{\bf QSD: existence, uniqueness and convergence}\label{sec-exi-uniq-conv}

In this section, we study the existence and uniqueness of QSDs of $X_{t}$, as well as the exponential convergence of the process $X_{t}$ conditioned on the event $[t<S_{\Ga}]$ to QSDs. In Subsection \ref{subsec-qsd-existence}, we show the existence of QSDs of $X_{t}$. In Subsection \ref{subsec-qsd-convergence}, we study the sharp exponential convergence of $X_{t}$ with compactly supported initial distributions. In Subsection \ref{subsec-qsd-uniqueness}, we investigate the uniqueness of QSDs of $X_{t}$ and the exponential convergence of $X_{t}$ with arbitrary initial distribution. The proofs of Theorems \ref{thm-qsd-existence-dynamics} and \ref{thm-uniqueness} are outlined in Subsection \ref{subsec-proof-thm-A-B}.

\subsection{Existence}\label{subsec-qsd-existence}

We construct QSDs for $X_{t}$. Recall that $\la_1$ and $\tilde v_1$ are given in Theorem \ref{thm-3-28-2}.

\begin{thm}\label{thm-4-25-1}
Assume {\bf (H1)-(H3)}.  Then, the following statements hold.
\begin{enumerate}
\item[\rm(1)] $\la_1>0$ and $\int_{\UU}\tilde{v}_1e^{-\frac{Q}{2}-\be_0 U}\md x<\infty$.

\item[\rm(2)] For each $f\in C_b(\UU)$,
\begin{equation*}
\E^{\nu_1}\left[f(X_t)\mathbbm{1}_{\{t<S_{\Ga}\}}\right]=e^{-\la_1 t}\int_{\UU}f\md \nu_1,\quad\forall t\geq0,
\end{equation*}
where $\md \nu_1:=\frac{\tilde{v}_1e^{-\frac{Q}{2}-\be_0 U}}{\int_{\UU}\tilde{v}_1e^{-\frac{Q}{2}-\be_0 U}\md x}\md x$.

\item[\rm(3)] $\nu_1$ is a QSD of $X_{t}$ with extinction rate $\la_{1}$.
\end{enumerate}
\end{thm}

We need the following lemma. Recall that the weight function $\al$ is defined in \eqref{eqn-3-28-2}.

\begin{lem}\label{lem-6-5-1}
Assume {\bf(H1)}-{\bf (H3)}. For each $\tilde{v}\in L^2(\UU,\al\md x;\C)$ one has $\int_{\UU}|\tilde v| e^{-\frac{Q}{2}-\be_0 U}\md x<\infty$.
\end{lem}

\begin{proof}
Fix $\tilde{v}\in L^2(\UU,\al\md x;\C)$. By H\"older's inequality,
$$
\int_{\UU}|\tilde{v}| e^{-\frac{Q}{2}-\be_0U}\md x\leq \left(\int_{\UU}\al|\tilde{v}|^2\md x\right)^{\frac{1}{2}} \left(\int_{\UU}\frac{1}{\al} e^{-Q-2\be_0U}\md x\right)^{\frac{1}{2}}.
$$
It suffices to verify
\begin{equation}\label{eqn-2020-06-30}
\int_{\UU}\frac{1}{\al} e^{-Q-2\be_0U}\md x<\infty.
\end{equation}

Let $\tilde{\al}(t):=\max\{\frac{1}{t^2}, 1\}$ for $t>0$. According to the definition of $\al$ given in \eqref{eqn-3-28-2} and the fact that $\inf_{\UU}\al >0$, there exists $C_1>0$ such that $\al(x)\geq C_1\sum_{i=1}^d\tilde{\al}(x_i)$ for $x\in \UU$. Since   $U(x)=V(\xi^{-1}(x))\geq \sum_{i=1}^d\tilde{V}(\xi^{-1}_i(x_i))$ for $x\in\UU$ due to {\bf (H3)}(2) and $e^{-Q}=\frac{\left[\prod_{i=1}^d a_i(\xi^{-1}_i(1))\right]^{\frac{1}{2}}}{\left[\prod_{i=1}^d a_i(\xi^{-1}_i(x_i))\right]^{\frac{1}{2}}}$, we derive
$$
\int_{\UU}\frac{1}{\al} e^{-Q-2\be_0U}\md x\leq \frac{\left[\prod_{i=1}^d a_i(\xi^{-1}_i(1))\right]^{\frac{1}{2}}}{C_1}\int_{\UU}\frac{\prod_{i=1}^d \exp\left\{-2\be_0 \tilde{V}(\xi^{-1}_i(x_i))\right\}}{\left(\sum_{i=1}^d \tilde{\al}(x_i)\right) \times \left[\prod_{i=1}^d a_i(\xi^{-1}_i(x_i)) \right]^{\frac{1}{2}}}\md x.
$$

For each $k\in \{1,\dots, d\}$, we denote by $\Si_k$ the collection of all subsets of $\{1,\dots, d\}$ with $k$ elements, and set
$$
A_k:=\sup_{\si\in \Si_k} \int_{\left\{x_{\si}=(x_i)_{i\in \si}:x_i>0,\,\,\forall i\in\si\right\}}\frac{\prod_{i\in \si} \exp\left\{-2\be_0 \tilde{V}(\xi_i^{-1}(x_i))\right\}}{\left(\sum_{i\in \si} \tilde{\al}(x_i)\right) \times \left[\prod_{i\in\si} a_i(\xi_i^{-1}(x_i)) \right]^{\frac{1}{2}}}\md x_{\si}.
$$
Clearly, \eqref{eqn-2020-06-30} holds if $A_d<\infty$. We show this by induction.

First, we show  $A_1<\infty$. Following the arguments leading to \eqref{eqn-5-9-1}, we can find $C_2>0$ such that $a_i(\xi^{-1}_i(x_i))\geq C_2^2 x_i^2$ for $x_i\in [0,1]$ and $i\in \{1,\dots, d\}$. It follows that for each $i\in \{1,\dots, d\}$,
\begin{equation*}
    \begin{split}
        \int_0^{\infty}\frac{e^{-2\be_0\tilde{V}(\xi^{-1}_i(x_i))}}{\tilde{\al}(x_i)\left[a_i(\xi^{-1}_i(x_i))\right]^{\frac{1}{2}}}\md x_i
        &\leq\frac{1}{C_2}\int_0^{1}\frac{e^{-2\be_0\tilde{V}(\xi^{-1}_i(x_i))}}{\tilde{\al}(x_i) x_i}\md x_i+\int_1^{\infty}\frac{e^{-2\be_0\tilde{V}(\xi^{-1}_i(x_i))}}{\tilde{\al}(x_i)\left[a_i(\xi^{-1}_i(x_i))\right]^{\frac{1}{2}}}\md x_i\\
        &\leq\frac{1}{C_2}\int_0^{1}x_ie^{-2\be_0\tilde{V}(\xi^{-1}_i(x_i))}\md x_i+
        \int_1^{\infty}\frac{e^{-2\be_0\tilde{V}(\xi^{-1}_i(x_i))}}{\left[a_i(\xi^{-1}_i(x_i))\right]^{\frac{1}{2}}}\md x_i\\
        &\leq\frac{1}{2C_2}+\int_1^{\infty}\frac{e^{-2\be_0\tilde{V}(\xi^{-1}_i(x_i))}}{\left[a_i(\xi^{-1}_i(x_i))\right]^{\frac{1}{2}}}\md x_i,
    \end{split}
\end{equation*}
where we used the definition of $\tilde{\al}$ in the second inequality and the non-negativity of $\tilde{V}$ in the last inequality. Changing variable, we find from {\bf (H3)}(2) that
\begin{equation*}
    \begin{split}
        \int_1^{\infty}\frac{e^{-2\be_0\tilde{V}(\xi^{-1}_i(x_i))}}{\left[a_i(\xi^{-1}_i(x_i))\right]^{\frac{1}{2}}}\md x_i=\int_{\xi^{-1}_i(1)}^{\infty}\frac{e^{-2\be_0\tilde{V}(z_i)}}{a_i(z_i)}\md z_i<\infty,
    \end{split}
\end{equation*}
leading to $A_1<\infty$.

Suppose $A_k<\infty$ for some $k\in\{ 1,\dots, d-1\}$, we show $A_{k+1}<\infty$. Without loss of generality, we only prove
$$
A_{k+1}^{1}:=\int_0^{\infty}\cdots \int_0^{\infty}\frac{\prod_{i=1}^{k+1} \exp\left\{-2\be_0 \tilde{V}(\xi_i^{-1}(x_i))\right\}}{\left(\sum_{i=1}^{k+1} \tilde{\al}(x_i)\right) \times \left[\prod_{i=1}^{k+1} a_i(\xi_i^{-1}(x_i)) \right]^{\frac{1}{2}}}\md x_1\cdots\md x_{k+1}<\infty;
$$
integrals corresponding to other $\si\in\Si_{k+1}$ can be treated in exactly the same way. Note that
\begin{equation*}
    \begin{split}
        &\int_0^1\cdots\int_0^1\frac{\prod_{i=1}^{k+1} \exp\left\{-2\be_0 \tilde{V}(\xi^{-1}_i(x_i))\right\}}{\left(\sum_{i=1}^{k+1} \tilde{\al}(x_i)\right) \times \prod_{i=1}^{k+1}\left[ a_i(\xi_i^{-1}(x_i)) \right]^{\frac{1}{2}}}\md x_1\cdots\md x_{k+1}\\
        &\qquad\leq \frac{1}{C_2^{k+1}}\int_0^1\cdots\int_0^1\frac{1}{\left(\sum_{i=1}^{k+1}\frac{1}{x_i^2}\right)\times\prod_{i=1}^{k+1}x_i}\md x_1\cdots\md x_{k+1}\\
        &\qquad\leq \frac{1}{(k+1)C_2^{k+1}}\int_0^1\cdots\int_0^1\frac{1}{\left(\prod_{i=1}^{k+1}x_i\right)^{1-\frac{2}{k+1}}}\md x_1\cdots\md x_{k+1}<\infty.
    \end{split}
\end{equation*}

Straightforward calculations give
\begin{equation*}
    \begin{split}
        &\int_0^{\infty}\cdots\int_1^{\infty}\cdots\int_0^{\infty}\frac{\prod_{i=1}^{k+1} \exp\left\{-2\be_0 \tilde{V}(\xi^{-1}_i(x_i))\right\}}{\left(\sum_{i=1}^{k+1} \tilde{\al}(x_i)\right) \times \prod_{i=1}^{k+1}\left[ a_i(\xi_i^{-1}(x_i)) \right]^{\frac{1}{2}}}\md x_1\cdots\md x_j\cdots\md x_{k+1}\\
        &\qquad\leq \int_1^{\infty}\frac{e^{-2\be_0\tilde{V}(\xi^{-1}_j(x_j))}}{\left[a_j(\xi^{-1}_j(x_j))\right]^{\frac{1}{2}}}\md x_j\\
        &\qquad\quad\times\int_0^{\infty}\cdots\int_0^{\infty}\frac{\prod_{\substack{i=1\\i\neq j}}^{k+1} \exp\left\{-2\be_0 \tilde{V}(\xi^{-1}_i(x_i))\right\}}{\left(\sum_{\substack{i=1\\i\neq j}}^{k+1} \tilde{\al}(x_i)\right) \times \prod_{\substack{i=1\\i\neq j}}^{k+1}\left[ a_i(\xi_i^{-1}(x_i)) \right]^{\frac{1}{2}}}\md x_1\cdots\widehat{\md x_j}\cdots\md x_{k+1}\\
        &\qquad\leq A_k \int_1^{\infty}\frac{e^{-2\be_0\tilde{V}(\xi^{-1}_j(x_j))}}{\left[a_j(\xi^{-1}_j(x_j))\right]^{\frac{1}{2}}}\md x_j=A_k\int_{\xi^{-1}_j(1)}^{\infty}\frac{e^{-2\be_0\tilde{V}(z_j)}}{a_j(z_j)}\md z_j<\infty,
    \end{split}
\end{equation*}
where we used {\bf (H3)}(2) in the last inequality. It follows that
\begin{equation*}
    \begin{split}
        A_{k+1}^{1}&\leq  \int_0^1\cdots\int_0^1\frac{\prod_{i=1}^{k+1} \exp\left\{-2\be_0 \tilde{V}(\xi^{-1}_i(x_i))\right\}}{\left(\sum_{i=1}^{k=1} \tilde{\al}(x_i)\right) \times \prod_{i=1}^{k+1}\left[ a_i(\xi_i^{-1}(x_i)) \right]^{\frac{1}{2}}}\md x_1\cdots\md x_{k+1}\\
        &\quad+\sum_{j=1}^{k+1}\int_0^{\infty}\cdots\int_1^{\infty}\cdots\int_0^{\infty}\frac{\prod_{i=1}^{k+1} \exp\left\{-2\be_0 \tilde{V}(\xi^{-1}_i(x_i))\right\}}{\left(\sum_{i=1}^d \tilde{\al}(x_i)\right) \times \prod_{i=1}^d\left[ a_i(\xi_i^{-1}(x_i)) \right]^{\frac{1}{2}}}\md x_1\cdots\md x_j\cdots\md x_{k+1}\\
        &<\infty.
    \end{split}
\end{equation*}
This completes the proof.
\end{proof}

\begin{proof}[Proof of Theorem \ref{thm-4-25-1}]
(1) Since $\tilde{v}_{1}\in L^2(\UU,\al\md x)$, Lemma \ref{lem-6-5-1} yields $\int_{\UU}\tilde{v}_1e^{-\frac{Q}{2}-\be_0 U}\md x<\infty$.

\medskip

To see $\la_1>0$, we fix $f\in C_0^{\infty}(\UU)$ and set $\tilde{f}:=fe^{-\frac{Q}{2}-\be_0 U}$. Obviously, $\tilde{f}\in L^2(\UU)$. Theorem \ref{thm-4-22-1} gives
$$
e^{\frac{-Q(x)}{2}-\be_0 U(x)}\E^x\left[f(X_t)\mathbbm{1}_{\{t<S_{\Ga}\}}\right]=T^{*}_t \tilde{f}(x),\quad\forall (x,t)\in\UU\times[0,\infty).
$$
Setting $v_1:=C\tilde{v}_1e^{-\frac{Q}{2}-\be_0 U}$, where $C:=\left(\int_{\UU}\tilde{v}_1e^{-\frac{Q}{2}-\be_0 U}\md x\right)^{-1}$, we deduce
\begin{equation*}
\begin{split}
\int_{\UU}v_1\E^{\bullet}\left[f(X_t)\mathbbm{1}_{\{t<S_{\Ga}\}}\right]\md x&=C\lan\tilde{v}_1, T^{*}_t \tilde{f} \ran_{L^{2}}=C\lan T_t \tilde{v}_1, \tilde f\ran_{L^{2}},\quad\forall t\geq 0,
\end{split}
\end{equation*}
which together with $T_t \tilde{v}_1=e^{-\la_1 t}\tilde{v}_1$ yields
\begin{equation}\label{eqn-4-24-6}
\int_{\UU}v_1\E^{\bullet}\left[f(X_t)\mathbbm{1}_{\{t<S_{\Ga}\}}\right]\md x=Ce^{-\la_1t}\int_{\UU}\tilde v_1 \tilde f\md x=e^{-\la_1t}\int_{\UU}v_1 f\md x,\quad\forall t\geq0.
\end{equation}

For each $x\in\UU$, the fact $\P^x\left[S_{\Ga}<\infty\right]=1$ implies $\lim_{t\to \infty} \E^x\left[f(X_t)\mathbbm{1}_{\{t<S_{\Ga}\}}\right]=0$. This together with the fact $\sup_{x\in\UU}\left| \E^x\left[f(X_t)\mathbbm{1}_{\{t<S_{\Ga}\}}\right]\right|\leq \|f\|_{\infty}$ for all $t\geq0$ and the dominated convergence theorem implies $\lim_{t\to \infty}\int_{\UU}v_1\E^{\bullet}\left[f(X_t)\mathbbm{1}_{\{t<S_{\Ga}\}}\right]\md x=0$. From which, we conclude $\la_1>0$, otherwise a contradiction can be easily derived from \eqref{eqn-4-24-6}.

\medskip

(2) Fix $f\in C_b(\UU)$ and take a sequence of functions $\{f_n\}_{n\in\N} \subset C_0^{\infty}(\UU)$ that locally uniformly converges to $f$ as $n\to\infty$ and satisfies $\|f_n\|_{\infty}\leq \|f\|_{\infty}$ for all $n\in\N$. It follows from \eqref{eqn-4-24-6} that for each $t\geq0$,
\begin{equation*}
\int_{\UU}\E^{\bullet}\left[f_n(X_t)\mathbbm{1}_{\{t<S_{\Ga}\}}\right]\md \nu_1=e^{-\la_1t}\int_{\UU} f_n\md \nu_1,\quad\forall n\in\N,
\end{equation*}
where $\md \nu_1:=v_1\md x$. Letting $n\to\infty$, we conclude the result from the dominated convergence theorem.

\medskip

(3) Applying (2) with $f=\mathbbm{1}_{\UU}$, we find $\P^{\nu_1}\left[t<S_{\Ga}\right]=\E^{\nu_1}\left[\mathbbm{1}_{\{t<S_{\Ga}\}}\right]=e^{-\la_1 t}$ for all $t\geq 0$.
Applying (2) again, we conclude
\begin{equation*}
\frac{\E^{\nu_1} \left[f(X_t)\mathbbm{1}_{\{t<S_{\Ga}\}}\right]}{\P^{\nu_1}\left[t<S_{\Ga}\right]}=\int_{\UU}f \md\nu_1,\quad \forall f\in C_b(\UU).
\end{equation*}
That is, $\nu_1$ is a QSD of $X_{t}$ and $\la_{1}$ is the associated extinction rate.
\end{proof}

\subsection{Sharp exponential convergence}\label{subsec-qsd-convergence}

We study the long-time dynamics of $X_{t}$ before reaching the boundary $\Ga$. Ahead of  stating the result, we recall and introduce some notation.

Recall that the spectra of $-\LL^*_{\be_0}$ and $-\LL_{\be_0}$ coincide, are discrete and contained in $\{\la\in\C: \Re\la>0\,\,\text{and}\,\,\arg \la \leq \tha\}$ for some $\tha\in (0,\frac{\pi}{2})$. The number $\la_{1}$ is the principal eigenvalue of both $-\LL^*_{\be_0}$ and $-\LL_{\be_0}$. Let $\tilde{v}^*_1$ be as in Lemma \ref{lem-characterization-domain} (4) and suppose it satisfies the normalization
\begin{equation}\label{normalization-v_1}
\lan\tilde{v}_{1},\tilde{v}_{1}^{*}\ran_{L^{2}}=\int_{\UU}\tilde{v}_1e^{-\frac{Q}{2}-\be_0 U}\md x,
\end{equation}
where $\tilde v_1$ is given in Theorem \ref{thm-3-28-2}. The last integral converges thanks to Lemma \ref{lem-6-5-1}. Note that $\range(\PP_{1}|_{L^{2}(\UU)})$ and $\range(\PP^*_{1}|_{L^{2}(\UU)})$ are respectively spanned over $\R$ by $\tilde{v}_{1}$ and $\tilde{v}_{1}^{*}$.

Set $\la_{2}:=\min\left\{\Re \la:\la\in \si(-\LL^*_{\be_0}) \andd \Re\la >\la_1\right\}$. Then, $\la_2>\la_1$ and $\{\la\in\si(-\LL^*_{\be_0}):\Re\la=\la_2\}$ consists of finitely many elements. For $k=1,2$, let $\PP^*_k$ and $\PP_k$ be respectively the spectral projections of $-\LL^*_{\be_0}$ and $-\LL_{\be_0}$ corresponding to $\{\la\in\si(-\LL^*_{\be_0}):\Re\la=\la_k\}$. Clearly, $\PP^*_1$ and $\PP_1$ are adjoint to each other. Since the coefficients of $-\LL^*_{\be_0}$ and $-\LL_{\be_0}$ are real-valued resulting in the symmetry of the set $\{\la\in\si(-\LL^*_{\be_0}):\Re\la=\la_2\}$ with respect to the real axis, $\PP^*_2$ and $\PP_2$ are also adjoint to each other.

Suppose the set $\{\la\in \si(-\LL^*_{\be_0}):\Re\la=\la_2\}$ consists of $N_*$ elements and is enumerated as
$$
\la_{2,i},\quad i\in\{0,\dots, N_*-1\}.
$$
Denote by $\PP^*_{2,i}$ and $\PP_{2,i}$ the spectral projections of $-\LL^*_{\be_0}$ and $-\LL_{\be_0}$  corresponding to $\la_{2,i}$ and $\ol{\la_{2,i}}$, respectively. Note that $\PP^*_{2,i}$ and $\PP_{2,i}$ are adjoint to each other. Obviously, $\PP^*_2=\sum_{i=0}^{N_*-1}\PP^*_{2,i}$ and $\PP_2=\sum_{i=0}^{N_*-1}\PP_{2,i}$.

For $i\in\{0,\dots, N_*-1\}$, we let
\begin{itemize}
    \item $N_i$ be the order of the pole $\la_{2,i}$ of the resolvent of $-\LL^*_{\be_0}$,

    \item $d_i=\dim(\range\PP^*_{2,i})$,

    \item $\{\tilde{v}^{(*,2)}_{i,j}:j\in\{1,\dots,d_i\}\}$ and $\{\tilde{v}^{(2)}_{i,j}:j\in\{1,\dots,d_i\}\}$ be generalized eigenfunctions of $-\LL^*_{\be_0}$ and $-\LL_{\be_0}$ that form bases of $\range\PP^*_{2,i}$ and $\range\PP_{2,i}$, respectively, and satisfy the normalization
    \begin{equation}\label{normalization-2020-06-30}
    \langle\Tilde{v}^{(2)}_{i,j},\Tilde{v}^{(*,2)}_{i,k}\rangle_{L^2}=\de_{jk},\quad\forall j,k\in\{1,\dots, d_i\}.
    \end{equation}

\end{itemize}

Recall that $\nu_1$ is the QSD of $X_{t}$ obtained in Theorem \ref{thm-4-25-1}, and $\{T_{t}\}_{t\geq0}$ and $\{T_{t}^{*}\}_{t\geq0}$ are positive and analytic semigroups of contractions on $L^{2}(\UU;\C)$ generated by $\LL_{\beta_{0}}$ and $\LL_{\beta_{0}}^{*}$, respectively.

The main result in this subsection is stated in the next theorem.

\begin{thm}\label{thm-convergence-according-to-gap}
Assume {\bf (H1)}-{\bf(H3)}. For each $\nu\in\PP(\UU)$ with compact support in $\UU$, there holds for each $f\in C_b(\UU)$,
\begin{equation*}
\begin{split}
&\E^{\nu} [f(X_t)\big|t<S_{\Ga}]\\
&\qquad=\int_{\UU}f\md \nu_1+\frac{e^{\la_1 t}}{\int_{\UU} e^{\frac{Q}{2}+\be_0 U}\tilde{v}^*_1\md\nu}\int_{\UU}e^{\frac{Q}{2}+\be_0 U} T^*_t\PP^*_2\left(\tilde{f}-\tilde{\mathbbm{1}}_{\UU}\int_{\UU}f\md \nu_1\right)\md\nu+o(e^{-(\la_2-\la_1)t})\\
&\qquad=\int_{\UU}f\md \nu_1+\frac{e^{-(\la_2-\la_1) t}}{\int_{\UU}e^{\frac{Q}{2}+\be_0 U}\tilde{v}^*_1\md\nu}\\
&\qquad\qquad\qquad\qquad\times\int_{\UU}e^{\frac{Q}{2}+\be_0 U}\sum_{j=0}^{N_*-1} e^{-i\Im \la_{2,j}t} \sum_{k=0}^{N_j-1}\frac{t^k}{k!}(\LL^*_{\be_0}+\la_{2,j})^k\PP^*_{2,j}\left(\tilde{f}-\tilde{\mathbbm{1}}_{\UU}\int_{\UU}f\md \nu_1\right)\md\nu\\
&\qquad\quad+o(e^{-(\la_2-\la_{1})t})\quad \text{as}\quad t\to \infty,
\end{split}
\end{equation*}
where $\tilde{f}:=e^{-\frac{Q}{2}-\be_0 U} f$ and $\tilde{\mathbbm{1}}_{\UU}:=e^{-\frac{Q}{2}-\be_0 U}\mathbbm{1}_{\UU}$.

In particular, 
the following hold:
\begin{itemize}
\item For each $0<\ep\ll1$,
\begin{equation*}
\lim_{t\to\infty} e^{(\la_2-\la_1-\ep)t}\left|\E^{x} [f(X_t)\big |t<S_{\Ga}]-\int_{\UU}f\md\nu_1\right|=0,\quad \forall f\in C_b(\UU)\andd x\in \UU.
\end{equation*}

\item If $f\in C_b(\UU)$ is such that $\PP^*_2\left(\tilde{f}-\tilde{\mathbbm{1}}_{\UU}\int_{\UU}f\md \nu_1\right)\neq0$, then for a.e. $x\in \UU$, there is a discrete set $\II_x\subset(0,\infty)$ with distances between adjacent points admitting an $x$-independent positive lower bound, such that for each $0<\de\ll 1$ there holds
$$
\lim_{\substack{t\to\infty\\t\in (0,\infty)\sm\II_{x,\de}}} e^{(\la_2-\la_1+\ep)t}\left|\E^{x} [f(X_t)\big |t<S_{\Ga}]-\int_{\UU}f\md\nu_1\right|=\infty,\quad\forall 0<\ep\ll 1,
$$
where $\II_{x,\de}$ is the $\de$-neighbourhood of $\II_x$ in $(0,\infty)$.
\end{itemize}
\end{thm}

\begin{rem}
We make some remarks about Theorem \ref{thm-convergence-according-to-gap}.
\begin{enumerate}
    \item[\rm(1)] Theorem \ref{thm-convergence-according-to-gap} appears to be a direct consequence of the decomposition of $(T^{*}_{t})_{t\geq0}$ according to spectral projections ensured by Theorem \ref{thm-5-22-1} and the stochastic representation given in Theorem \ref{thm-4-22-1}. This is however deceptive due to the following two reasons: (i) the stochastic representation given in Theorem \ref{thm-4-22-1} is only true for $f\in C_{b}(\UU)$ such that $e^{-\frac{Q}{2}-\beta_{0}U}f\in L^{2}(\UU)$; this is indeed a restriction as $e^{-\frac{Q}{2}-\beta_{0}U}$ and $e^{\frac{Q}{2}+\beta_{0}U}$ are respectively unbounded near $\Ga$ and $\infty$; (ii) the semigroup $(T^{*}_{t})_{t\geq0}$ is naturally defined on $L^{2}(\UU)$, but we need its $L^{\infty}$ properties.

    \item[\rm(2)] For $f\in C_{b}(\UU)$, the function $\tilde{f}:=e^{-\frac{Q}{2}-\be_0 U} f$ does not necessarily belong to $L^{2}(\UU)$. Neither does $\tilde{f}-\tilde{\mathbbm{1}}_{\UU}\int_{\UU}f\md \nu_1$. Its projections under $\PP_{2}^{*}$ and $\PP_{2,j}^{*}$ are justified in Lemma \ref{lem-5-28-3} (2).

    \item[\rm(3)] Theorem \ref{thm-convergence-according-to-gap} actually holds for all initial distributions $\nu\in\PP(\UU)$ satisfying the condition $\int_{\UU}e^{\frac{Q}{2}+\be_0U}\md\nu<\infty$. See Remark \ref{rem-on-sharp-convergence} for more details.
\end{enumerate}
\end{rem}

We prove several lemmas before proving Theorem \ref{thm-convergence-according-to-gap}.

\begin{lem}\label{lem-5-28-2}
Assume {\bf (H1)}-{\bf(H3)}.
\begin{enumerate}
\item[\rm(1)] One has
$$
T^*_t\PP^*_2=\sum_{j=0}^{N_*-1}T^*_t\PP^*_{2,j}=e^{-\la_2t}\sum_{j=0}^{N_*-1}e^{-i\Im\la_{2,j} t}\sum_{k=0}^{N_j-1}\frac{t^k}{k!}(\LL^*_{\be_0}+\la_{2,j})^k\PP^*_{2,j},\quad\forall t\geq 0.
$$

\item[\rm(2)] For each $0<\ep\ll 1$, there exists  $C=C(\ep)>0$ such that
\begin{equation*}
\left\|T^*_t-e^{-\la_1 t}\PP^*_1-T^*_t\PP^*_2\right\|_{L^2\to L^2}\leq C e^{-(\la_2+\ep) t},\quad \forall t\geq 0.
\end{equation*}

\item[\rm(3)] For each $0<\ep\ll1$ and $f\in\range\PP^*_2\sm\{0\}$ we have
$$
\lim_{t\to\infty} e^{(\la_2-\ep)t}|T^*_tf|=0 \quad\text{in}\quad\UU.
$$

\item[\rm(4)] Let  $f\in\range\PP^*_2\sm\{0\}$. Then,  for a.e. $x\in \UU$, there is a discrete set $\II_x\subset(0,\infty)$ with distances between adjacent points admitting an $x$-independent positive lower bound, such that for each $0<\de\ll 1$ one has
$$
\lim_{\substack{t\to\infty\\t\in (0,\infty)\sm\II_{x,\de}}} e^{(\la_2+\ep)t}|T^*_tf|(x)=\infty,\quad\forall 0<\ep\ll 1,
$$
where $\II_{x,\de}$ is the $\de$-neighbourhood of $\II_x$ in $(0,\infty)$.
\end{enumerate}
\end{lem}

\begin{proof}
(1) and (2) are special cases of \cite[Corollary V. 3.2]{EN00} due to Theorem \ref{thm-5-22-1}, the fact $\Re\la_{2,i}=\la_{2}$ for all $i\in\{1,\dots,N_*\}$, and the simplicity of the principle eigenvalue $\la_1$ of $-\LL^*_{\be_0}$. (3) is a simple consequence of (1).


We show (4). Fix $f\in \range{\PP^*_2}\sm \{0\}$. We consider three cases.

\smallskip

\paragraph{\bf Case 1} $N_*=1$. In this case, $\{\la\in\si(-\LL^*_{\be_0}): \Re\la=\la_2\}=\{\la_2\}$. Then, $f$ is a generalized eigenfunction of $-\LL^*_{\be_0}$ associated to $\la_2$, and thus, there exists $\tilde{N}\in\N$ such that $(\LL^*_{\be_0}+\la_2)^{\tilde{N}+1}f=0$ and $(\LL^*_{\be_0}+\la_2)^{\tilde{N}}f\neq 0$ in $\UU$. By the strong unique continuation principle for elliptic equations (see e.g. \cite{Lerner19}), we find $(\LL^*_{\be_0}+\la_2)^{\tilde{N}}f\neq 0$ a.e. in $\UU$. Since
\begin{equation*}
    \begin{split}
    T^*_t f &=e^{-\la_{2}t}\sum_{k=0}^{\tilde{N}} \frac{t^k}{k!}(\LL^*_{\be_0}+\la_{2})^k f
    =e^{-\la_{2}t}\left(\frac{t^{\tilde{N}}}{\tilde{N}!}(\LL^*_{\be_0}+\la_{2})^{\tilde{N}} f+ o( t^{\tilde{N}})\right)\quad\text{as}\quad t\to\infty,
    \end{split}
\end{equation*}
we derive $\lim_{t\to\infty} e^{(\la_2+\ep)t}|T^*_tf|(x)=\infty$ for a.e. $x\in\UU$ and each $0<\ep\ll1$. The conclusion follows immediately.

\smallskip

\paragraph{\bf Case 2} $N_*=2K+1$ for some $K\in\N$. Considering the symmetry of the set $\{\la\in\si(-\LL^*_{\be_0}): \Re\la=\la_2\}$ with respect to the real axis, we can re-enumerate it as $\{\la_{2,j}\}_{j=-K}^K$ such that $\la_{2,0}=\la_2$ and $\la_{2,j}=\ol{\la}_{2,-j}$ for $j\in \{1,\dots, K\}$.

Note that $f=\sum_{j=-K}^{K}f_j$, where $f_j$ is the projection of $f$ onto the generalized eigenspace of $\la_{2,j}$. Since $f$ is real-valued we must have $f_j=\ol{f}_{-j}$ for all $j\in \{1,\dots, K\}$. We may assume, without loss of generality, that $f_j\neq 0$ for all $j\in \{-K,\dots, K\}$.

Since $\la_{2,j}$ is a pole of the resolvent of $-\LL^*_{\be_0}$ with finite order, there exists $\tilde{N}_j\in\N$ such that
$(\LL^*_{\be_0}+\la_{2,j})^{\tilde{N}_j+1}f_j=0$ and $(\LL^*_{\be_0}+\la_{2,j})^{\tilde{N}_j}f_j\neq 0$. Applying the strong unique continuation principle for elliptic equations (see e.g. \cite{Lerner19}), we find
\begin{equation}\label{eqn-2020-08-29-1}
    (\LL^*_{\be_0}+\la_{2,j})^{\tilde{N}_j}f_j\neq 0\quad\text{a.e. in}\quad\UU.
\end{equation}
Clearly, $\tilde{N}_j=\tilde{N}_{-j}$ for all $j\in \{1,\dots,K\}$. Straightforward calculations then give for $t\gg 1$,
\begin{equation}\label{eqn-7-8-1}
    \begin{split}
    e^{\la_2 t}T^*_t f &= \sum_{j=-K}^{K}e^{-\Im\la_{2,j}t}\sum_{k=0}^{\tilde{N}_j} \frac{t^k}{k!}(\LL^*_{\be_0}+\la_{2,j})^k f_j\\
    &=\sum_{j=-K}^{K}e^{-i\Im\la_{2,j} t}\left[ \frac{t^{\tilde{N}_j}}{\tilde{N}_j!}(\LL^*_{\be_0}+\la_{2,j})^{\tilde{N}_j}f_j+o(t^{\tilde{N}_j})\right]\\
    &=\left[\frac{t^{\tilde{N}_0}}{\tilde{N}_0!}(\LL^*_{\be_0}+\la_{2})^{\tilde{N}_0}f_0+o(t^{\tilde{N}_0})\right]\\
    &\quad+\sum_{j=1}^{K}\left[ \frac{2t^{\tilde{N}_j}}{\tilde{N}_j!}\Re\left(e^{-i\Im\la_{2,j} t}(\LL^*_{\be_0}+\la_{2,j})^{\tilde{N}_j}f_j\right)+o(t^{\tilde{N}_j})\right].
    \end{split}
\end{equation}

Since the asymptotics of $e^{\la_2t}T^*_tf$ as $t\to \infty$ is determined by the terms with the highest degree, we may assume, without loss of generality, that $\tilde{N}_0=\tilde{N}_{1}=\dots =\tilde{N}_K$.

Set $F_0:=\frac{1}{\tilde{N}_0!}(\LL^*_{\be_0}+\la_{2})^{\tilde{N}_0}f_0$ and $F_j:=\frac{2}{\tilde{N}_j!}(\LL^*_{\be_0}+\la_{2,j})^{\tilde{N}_j}f_j$ for $j\in \{1,\dots, K\}$. We rewrite \eqref{eqn-7-8-1} as
\begin{equation}\label{eqn-2020-08-28-1}
    \begin{split}
       \frac{e^{\la_2t} T^*_tf}{t^{\tilde{N}_0}}&=F_0 +\sum_{j=1}^K \Re(e^{-i\Im \la_{2,j}t}F_j)+o(1)\\ &=F_0+\sum_{j=1}^K\left(\cos(\Im\la_{2,j}t)\Re F_j-\sin(\Im\la_{2,j}t)\Im F_j\right)+o(1)\\
       &=F_0+\sum_{j=1}^K |F_j|\sin(\Im\la_{2,j}t+\vp_j)+o(1),\quad\forall t\gg1,
    \end{split}
\end{equation}
where $\vp_j\in [0,2\pi)$ satisfies $\tan\vp_j=-\frac{\Re F_j}{\Im F_j}$ for $j\in \{1,\dots, K\}$.

Note that \eqref{eqn-2020-08-29-1} ensures the set $\NN:=\{x\in \UU:\exists j\in\{0,1,\dots,K\}\text{ s.t. } |F_j|(x)=0\}$ has zero Lebesgue measure. Fix  $x\in \UU\sm \NN$. Denote by $\II_x$ the zeros of the function $t\mapsto \sum_{j=1}^K |F_j|(x)\sin(\Im\la_{2,j}t+\vp_j):(0,\infty)\to\R$. It is not hard to see that $\II_{x}$ is a discrete set with distances between adjacent points admitting a positive lower bound, which is independent of $x\in \UU\sm \NN$.

For each $0<\de\ll 1$, we set $\II_{x,\de}:=\{t\in(0,\infty):\dist(t, \II_x)<\de\}$  and conclude from \eqref{eqn-2020-08-28-1} that
$$
\lim_{\substack{t\to\infty\\t\in (0,\infty)\sm\II_{x,\de}}}e^{(\la_2+\ep)t}|T^*_tf|(x)=\infty,\quad\forall 0<\ep\ll 1.
$$

\smallskip

\paragraph{\bf Case 3} $N_*=2K$ for some $K\in \N$. The proof is exactly the same as that in {\bf Case 2} except that $f_0$ does not appear due to the fact $\la_{2}\not\in\{\la\in\si(-\LL^*_{\be_0}): \Re\la=\la_2\}$.

This completes the proof.
\end{proof}

\begin{lem}\label{lem-5-28-3}
Assume {\bf (H1)}-{\bf(H3)}. The following hold:
\begin{enumerate}
\item[\rm(1)] $\PP^*_1 \tilde{f}=\tilde{v}^*_1\int_{\UU}\tilde{f} e^{\frac{Q}{2}+\be_0 U}\md \nu_1$ for all $\tilde{f}\in L^2(\UU)$.


\item[\rm(2)] For each $i\in\{1,\dots,N_{*}\}$ one has
    \begin{equation}\label{eqn-6-6-1}
    \PP^*_{2,i}\tilde{f}=\sum_{j=1}^{d_i}\Tilde{v}^{(*,2)}_{i,j}\langle \tilde{f}, \tilde{v}^{(2)}_{i,j}\rangle_{L^2},\quad\forall \tilde{f}\in L^2(\UU).
    \end{equation}
In particular, $\PP^*_{2,i}$, $i\in\{1,\dots, N_*\}$, and hence, $\PP^*_2$ are well-defined on $\{\Tilde{f}:\tilde{f}e^{\frac{Q}{2}+\be_0 U}\in C_b(\UU)\}$.
\end{enumerate}
\end{lem}
\begin{proof}
(1) Note that $\range(\PP_{1}^{*}|_{L^{2}(\UU)})$ is spanned over $\R$ by $\tilde{v}_{1}^{*}$. By the Riesz representation theorem, there exists $h\in L^2(\UU)$ such that
\begin{equation}\label{eqn-5-31-1}
\PP^*_1\tilde{f}=\langle \tilde{f},h \rangle_{L^2}\tilde{v}^*_1,\quad\forall \tilde{f}\in L^2(\UU).
\end{equation}
As $\PP_1$ and $\PP^*_1$ are adjoint to each other it must be true that
$\PP_{1}\tilde{v}=\langle \tilde{v},\tilde{v}_{1}^{*} \rangle_{L^2}h$ for all $\tilde{v}\in L^2(\UU)$. Since $\range(\PP_{1}|_{L^{2}(\UU)})$ is spanned over $\R$ by $\tilde{v}_{1}$, there exists $C_{1}\in\R$ such that $h=C_{1}\tilde{v}_{1}$. Thus, the normalization \eqref{normalization-v_1} gives
$$
\tilde{v}_{1}=\PP_{1}\tilde{v}_{1}=C_{1}\langle \tilde{v}_{1},\tilde{v}_{1}^{*} \rangle_{L^2}\tilde{v}_{1}=C_{1}\tilde{v}_{1}\int_{\UU}\tilde{v}_1e^{-\frac{Q}{2}-\be_0 U}\md x,
$$
leading to $C_{1}=\frac{1}{\int_{\UU}\tilde{v}_1e^{-\frac{Q}{2}-\be_0 U}\md x}$, and hence, $h=\frac{\tilde{v}_{1}}{\int_{\UU}\tilde{v}_1e^{-\frac{Q}{2}-\be_0 U}\md x}$. Inserting this into \eqref{eqn-5-31-1} and noting
the definition of $\nu_1$ gives rise to the formula for $\PP^{*}_{1}\tilde{f}$. 

(2) The formula for $\PP^*_{2,i}\tilde{f}$ can be derived from arguments similar to those in the proof of (1), and the normalization \eqref{normalization-2020-06-30}.


As Lemma \ref{lem-6-5-1} ensures $C_{ij}:=\int_{\UU} e^{-\frac{Q}{2}-\be_0 U}|\tilde{v}^{(2)}_{i,j}|\md x<\infty$, we find for each $\tilde{f}$ satisfying $\tilde{f} e^{\frac{Q}{2}+\be_0 U}\in C_b(\UU)$ that
$\left|\langle \tilde{f}, \tilde{v}^{(2)}_{i,j}\rangle_{L^2}\right|\leq C_{ij}\sup_{\UU}\left|\tilde{f} e^{\frac{Q}{2}+\be_0 U} \right|$.
Thus, using the formula \eqref{eqn-6-6-1}, we can define $\PP^*_{2,i}$,  $i\in \{1,\dots, d_i\}$, and hence, $\PP^*_2$ on the set $\left\{\tilde{f}:\tilde{f}e^{\frac{Q}{2}+\be_0 U}\in C_b(\UU)\right\}$. This completes the proof.
\end{proof}

\begin{lem}\label{lem-4-25-1}
Assume {\bf (H1)}-{\bf(H3)}. For each $t>0$, there exists $C=C(t)>0$ such that
\begin{equation*}
\left\| e^{-\frac{Q}{2}-\be_0 U}\E^{\bullet}\left[f(X_t)\mathbbm{1}_{\{t<S_{\Ga}\}}\right]\right\|_{L^2}\leq C\|f\|_{\infty},\quad \forall f\in C_b(\UU).
\end{equation*}
\end{lem}

\begin{proof}
Fix $f\in C_b(\UU)$ and set $\tilde{f}:=fe^{-\frac{Q}{2}-\be_0U}$. Recall that $2_*:=\frac{2(d+2)}{d+4}\in (1,2)$ (see Lemma \ref{lem-6-12-1}).
Since $e^{-\frac{Q(x)}{2}}=\frac{\left[\prod_{i=1}^d a_i(\xi^{-1}_i(1))\right]^{\frac{1}{4}}}{\left[\prod_{i=1}^d a_i(\xi^{-1}_i(x_i))\right]^{\frac{1}{4}}}$, we find
\begin{equation*}
\begin{split}
\int_{\UU}|\tilde{f}|^{2_*}\md x=\int_{\UU}|f|^{2_*}e^{-\frac{2_*Q}{2}-2_*\be_0 U}\md x\leq \|f\|^{2_*}_{\infty}\left[\prod_{i=1}^d a_i(\xi^{-1}_i(1))\right]^{\frac{2_*}{4}} \int_{\UU} \frac{ e^{-2_*\be_0 U}}{\left[\prod_{i=1}^d a_i(\xi^{-1}_i(x_i))\right]^{\frac{2_*}{4}}} \md x.
\end{split}
\end{equation*}
Arguments as in the proof of Lemma \ref{lem-6-5-1} yield
$\int_{\UU} \frac{ e^{-2_*\be_0 U(x)}}{\left[\prod_{i=1}^d a_i(\xi^{-1}_i(x_i))\right]^{\frac{2_*}{4}}} \md x<\infty$. This implies the existence of $C_1>0$ (independent of $f$) such that
\begin{equation}\label{eqn-4-28-3}
\|\tilde{f}\|_{L^{2_*}(\UU)}\leq C_{1} \|f\|_{\infty}.
\end{equation}

Let $\{\UU_{n}\}_{n\in\N}$ and $\{\tau_{n}\}_{n\in\N}$ be as in Subsection \ref{subsec-first-exit-times}. For each $n\in\N$, we recall that $(T^{(*, \UU_n,2_*)}_t)_{t\geq 0}$ is the positive and analytic semigroup of contractions on $L^{2_*}(\UU_n;\C)$ generated by $(\LL_{\be_0}^{*,2_*}|_{\UU_n}, W^{2,2_*}(\UU_n;\C)\cap W^{1,2_*}_0(\UU_n;\C))$. Since $\tilde{f}\in C(\ol{\UU}_n)$,  Proposition \ref{lem-5-3-3} ensures
\begin{equation}\label{eqn-4-23-5}
T^{(*,\UU_n,2_*)}_{t}\tilde{f}|_{\UU_n} =e^{-\frac{Q}{2}-\be_0 U}\E^{\bullet}\left[f(X_t)\mathbbm{1}_{\{t<\tau_n\}}\right],\quad \forall t\in[0,\infty).
\end{equation}
It follows from Lemma \ref{lem-6-12-1} that for each $t>0$, there is a constant $C_2=C_2(t)>0$ such that $\|T^{(*,\UU_n,2_*)}_{t}\tilde{f}\|_{L^2(\UU_n)}\leq C_2\|\tilde{f}\|_{L^{2_*}(\UU_n)}$ for all $n\in\N$. Letting $n\to \infty$, we derive from Lemma \ref{lem-exit-times-approximation}, \eqref{eqn-4-23-5} and Fatou's lemma that
$$
\|e^{-\frac{Q}{2}-\be_0 U}\E^{\bullet}\left[f(X_t)\mathbbm{1}_{\{t<S_{\Ga}\}}\right]\|_{L^2(\UU)}\leq C_2\|\tilde{f}\|_{L^{2_*}(\UU)}\leq C_1 C_2\|f\|_{\infty},
$$
where we used \eqref{eqn-4-28-3} in the second inequality. This completes the proof.
\end{proof}

\begin{lem}\label{lem-5-28-1}
Assume {\bf (H1)}-{\bf(H3)}. There exists $C>0$ such that for each $f\in C_b(\UU)$ with $\tilde{f}:=e^{-\frac{Q}{2}-\be_0 U}f \in L^2(\UU)$ one has $\|T^*_{t}\tilde{f}\|_{\infty}\leq C \|T^*_{t-1}\tilde{f}\|_{L^2}$ for all $t\geq 1$.
\end{lem}
\begin{proof}
Fix $f\in C_b(\UU)$ satisfying $\tilde{f}:=e^{-\frac{Q}{2}-\be_0 U}f\in L^2(\UU)$.  Theorem \ref{thm-4-22-1} gives
\begin{equation}\label{eqn-5-28-3}
T^*_{t}\tilde{f}(x)=e^{-\frac{Q(x)}{2}-\be_0 U(x)}\E^x\left[f(X_t)\mathbbm{1}_{\{t<S_{\Ga}\}}\right],\quad\forall (x,t)\in\UU\times[0,\infty).
\end{equation}

Let $\{\UU_{n}\}_{n\in\N}$ and $\{\tau_{n}\}_{n\in\N}$ be as in Subsection \ref{subsec-first-exit-times}. We show the existence of $C>0$ such that
\begin{equation}\label{eqn-5-3-8-1}
\begin{split}
&\sup_{\UU_n} e^{-\frac{Q}{2}-\be_0 U}\left| \E^{\bullet}[f(X_t)\mathbbm{1}_{\{t<\tau_n\}}\right|\\
&\qquad\leq C \left\|e^{-\frac{Q}{2}-\be_0 U}\E^{\bullet}\left[f(X_{t-1})\mathbbm{1}_{\{t-1<\tau_n\}}\right]\right\|_{L^2(\UU_n)},\quad\forall t\geq 1\andd  n\in\N.
\end{split}
\end{equation}
The lemma then follows immediately from \eqref{eqn-5-28-3} and Lemma \ref{lem-exit-times-approximation}.

We show \eqref{eqn-5-3-8-1} by Moser iteration. Recall that for each $n\in\N$ and $N>1$, $(T^{(*,\UU_n,N)}_t)_{t\geq 0}$ is the positive and  analytic semigroup on $L^N(\UU_n;\C)$ generated by  $(\LL_{\be_0}^{*,N}|_{\UU_n},W^{2,N}(\UU_{n};\C)\cap W^{1,N}_0(\UU_n;\C))$. Since here for each $n$ we only consider the action of  $(T^{(*,\UU_n,N)}_t)_{t\geq 0}$ on functions in $C(\ol{\UU}_{n};\C)$, we simply write $(T^{(n)}_t)_{t\geq 0}$ for all $\left\{(T^{(*,\UU_n,N)}_t)_{t\geq 0},N>1\right\}$ in consideration of Proposition \ref{lem-5-3-3} (5). Obviously, $\tilde{f}_n:=\tilde{f}|_{\UU_n}\in C(\ol{\UU}_n)$ for all $n\in\N$. It follows from Proposition \ref{lem-5-3-3} (4) that
\begin{equation}\label{eqn-6-1-2}
T^{(n)}_t \tilde{f}_n(x)=e^{-\frac{Q(x)}{2}-\be_0 U(x)}\E^x\left[f(X_t)\mathbbm{1}_{\{t<\tau_n\}}\right],\quad\forall (x,t)\in \UU_n\times[0,\infty)\andd n\in\N.
\end{equation}

Set $\tilde{w}_n:=T^{(n)}_{\bullet} \tilde{f}_n$. It follows from Lemma \ref{lem-4-21-2} that for all $n\in\N$ and $N\geq 2$,
\begin{equation}\label{eqn-4-23-3}
\begin{split}
&\frac{1}{N}\int_{\UU_n} |\tilde{w}_n|^N(\cdot, t_2)\md x+\frac{N-1}{2}\int_{t_1}^{t_2}\int_{\UU_n} |\tilde{w}_n|^{N-2}|\nabla \tilde{w}_n|^2\md x\md s\\
&\qquad \leq \frac{1}{N}(1+e^{NM(t_2-t_1)})\int_{\UU_n} |\tilde{w}(\cdot, t_1)|^N\md x,\quad\forall t_2>t_1\geq 0,
\end{split}
\end{equation}
where we recall that $M>0$ is fixed and independent of $n\in\N$ and $N\geq 2$, such that the conclusion in Lemma \ref{lem-3-24-2} (3) holds. The Sobolev embedding theorem gives
$$
\|\tilde{w}_n^{\frac{N}{2}}\|_{L^{2\ka}(\UU_n\times[t_1,t_2])}\leq C_1 \left( \sup_{s\in[t_1,t_2]} \|\tilde{w}_n^{\frac{N}{2}}(\cdot, s)\|_{L^2(\UU_n)}+\|\nabla\tilde{w}_n^{\frac{N}{2}}\|_{L^2(\UU_n\times[t_1,t_2])}\right),
$$
where $\ka:=\frac{d+2}{d}$ and $C_1>0$ only depends on $d$. Therefore, \eqref{eqn-4-23-3} gives rise to
\begin{equation}\label{eqn-2020-06-03}
\begin{split}
&\left(\int_{t_1}^{t_2}\int_{\UU_n}|\tilde{w}_n|^{\ka N}\md x\md s\right)^{\frac{1}{\ka}}\\
&\qquad\leq 2C_1^2\left( \sup_{s\in[t_1,t_2]}\int_{\UU_n}|\tilde{w}_n(x,s)|^{N}\md x+\frac{N^2}{4}\int_{t_1}^{t_2}\int_{\UU_n}|\tilde{w}_n|^{N-2}|\nabla \tilde{w}_n|^2\md x\md s\right)\\
&\qquad\leq 2C_1^2\left( 1+\frac{N}{2(N-1)}\right)(1+e^{NM(t_2-t_1)})\int_{\UU_n} |\tilde{w}(\cdot, t_1)|^N\md x\\
&\qquad\leq 4C_1^2(1+e^{NM(t_2-t_1)})\int_{\UU_n} |\tilde{w}(\cdot, t_1)|^N\md x,\quad\forall t_2>t_1\geq 0,
\end{split}
\end{equation}
for all $n\in \N$ and $N\geq 2$. We then deduce from Lemma \ref{lem-4-21-2} (with $\ka N$ instead of $N$) and \eqref{eqn-2020-06-03} that
\begin{equation}\label{eqn-4-23-6}
\begin{split}
\frac{1}{\ka N}\int_{\UU_n} |\tilde{w}_n|^{\ka N}(\cdot, t_3)\md x&\leq \frac{2}{\ka N(t_2-t_1)}\int_{t_1}^{t_2}\int_{\UU_n}|\tilde{w}_n|^{\ka N}\md x\md s\\
& \leq \frac{2(4C_1^2)^{\ka}}{\ka N (t_2-t_1)} \left(1+e^{NM(t_2-t_1)}\right)^{\ka}\|\tilde{w}(\cdot, t_1)\|^{\ka N}_{L^N(\UU_n)}
\end{split}
\end{equation}
for all $t_3>t_2>t_{1}\geq 0$, $n\in \N$ and $N\geq 2$.

Fix $t\geq 1$. For each $\ell\in\N\cup \{0\}$, we set
$N=N_{\ell}:=2\ka^\ell$, $t_1:=t- 2^{-\ell}$, $t_2:=t-\frac{3}{2}2^{-(\ell+1)}$ and $t_3:=t- 2^{-(\ell+1)}$ in \eqref{eqn-4-23-6} to find
\begin{equation}\label{eqn-4-29-2-111}
\|\tilde{w}_n(\cdot, t- 2^{-(\ell+1)})\|_{L^{N_{\ell+1}}(\UU_n)}\leq C_2^{\frac{1}{N_{\ell+1}}} 2^{\frac{\ell+2}{N_{\ell+1}}}e^{M2^{-(\ell+1)}} \|\tilde{w}_n(\cdot, t- 2^{-\ell})\|_{L^{N_{\ell}}(\UU_n)}
\end{equation}
for all $\ell\in \N\cup\{0\}$ and $n\in\N$, where $C_2>0$ is independent of $\ell$ and $n$. Set
$$
A_{\ell}:=C_2^{\frac{1}{N_{\ell+1}}} 2^{\frac{\ell+2}{N_{\ell+1}}}e^{M 2^{-(\ell+1)}},\quad\ell\in\N\cup\{0\}.
$$
It follows from \eqref{eqn-4-29-2-111} that for each $n\in\N$,
\begin{equation*}
\begin{split}
\sup_{x\in\UU_n}|\tilde{w}_n(x,t)|&=\lim_{k\to \infty}\|\tilde{w}_n(\cdot, 1-2^{-(k+1)})\|_{L^{N_{k+1}}}\\
&\leq\lim_{k\to\infty}\left(\prod_{\ell=0}^{k}A_{\ell}\right)\times \|\tilde{w}_n(\cdot, t-1)\|_{L^2(\UU_n)}
= C_3 \|\tilde{w}_n(\cdot, t-1)\|_{L^2(\UU_n)},
\end{split}
\end{equation*}
where $C_3:=\prod_{\ell=0}^{\infty}A_{\ell}<\infty$. This, together with \eqref{eqn-6-1-2}, gives \eqref{eqn-5-3-8-1}.  This completes the proof.
\end{proof}

\begin{cor}\label{cor-L2-infty}
Assume {\bf (H1)}-{\bf(H3)}. There exists $C>0$ such that for each $\tilde{f}\in L^2(\UU)$ one has
$\|T^*_t\tilde{f}\|_{\infty}\leq C\|T^*_{t-1}\tilde{f}\|_{L^2}$ for all $t\geq 1$.
\end{cor}
\begin{proof}
By Lemma \ref{lem-5-28-1}, the conclusion holds for all $\tilde{f}\in C_0^{\infty}(\UU)$. Thus, the density of $C_0^{\infty}(\UU)$ is in $L^2(\UU)$ and the standard approximation arguments give the desired result.
\end{proof}

\begin{lem}\label{lem-5-28-4}
Assume {\bf (H1)}-{\bf(H3)}. For each $0<\ep\ll1$, there exists $C=C(\ep)>0$ such that for each $f\in C_b(\UU)$ it is true that
\begin{equation*}\label{eqn-5-28-1}
\begin{split}
&\left| \E^{\bullet} [f(X_{t})\mathbbm{1}_{\{t<S_{\Ga}\}}]- e^{\frac{Q}{2}+\be_0 U}\left(e^{-\la_1t}\tilde{v}^*_1 \int_{\UU}f\md \nu_1+T^*_t\PP^*_{2}\Tilde{f}\right)\right|\leq Ce^{\frac{Q}{2}+\be_0 U} e^{-(\la_{2}+\ep)t}\|f\|_{\infty}
\end{split}
\end{equation*}
for all $t\geq 2$, where $\tilde{f}:=e^{-\frac{Q}{2}-\be_0 U}f$.
\end{lem}
\begin{proof}
Fix $f\in C_b(\UU)$. By the Markov property and homogeneity of $X_t$,
\begin{equation}\label{eqn-4-25-5-2}
\begin{split}
 \E^{x}\left[f(X_t)\mathbbm{1}_{\{t<S_{\Ga}\}}\right]= \E^{x}\left[ g(X_{t-1})\mathbbm{1}_{\{t-1<S_{\Ga}\}}\right],\quad\forall (x,t)\in\UU\times[1,\infty),
\end{split}
\end{equation}
where $g:= \E^{\bullet} [f(X_1)\mathbbm{1}_{\{1<S_{\Ga}\}}]\in C_b(\UU)$. Lemma \ref{lem-4-25-1} ensures that $\tilde{g}:=e^{-\frac{Q}{2}-\be_0 U} g\in L^2(\UU)$ and the existence of $C_1>0$ (independent of $f$) such that
\begin{equation}\label{eqn-5-28-5}
\|\tilde{g}\|_{L^2}\leq C_1\|f\|_{\infty}.
\end{equation}

Fix $0<\ep\ll1$. By Lemma \ref{lem-5-28-2} (2), there exists $C_2=C_{2}(\ep)>0$ such that
\begin{equation*}
\begin{split}
\left\| T^*_{t-2}\tilde{g}-T^*_{t-2}\PP^*_1 \tilde{g}-T^*_{t-2}\PP^*_{2}\tilde{g}\right\|_{L^2}\leq C_2 e^{-(\la_{2}+\ep)(t-2)}\|\tilde{g}\|_{L^2},\quad\forall t\geq2.
\end{split}
\end{equation*}
Thanks to Lemma \ref{lem-5-28-1}, there exists $C_3>0$ (independent of $f$) such that
\begin{equation}\label{eqn-5-28-2}
\begin{split}
\left\| T^*_{t-1}\tilde{g}-T^*_{t-1}\PP^*_1 \tilde{g}-T^*_{t-1}\PP^*_{2} \tilde{g}\right\|_{\infty}&\leq C_3 \left\| T^*_{t-2}\tilde{g}-T^*_{t-2}\PP^*_1 \tilde{g}-T^*_{t-2}\PP^*_{2}\tilde{g}\right\|_{L^2}\\
&\leq  C_2 C_3e^{-(\la_{2}+\ep)(t-2)}\|\tilde{g}\|_{L^2}\\
&\leq C_1C_2 C_3 e^{2(\la_{2}+\ep)} e^{-(\la_{2}+\ep)t}\|f\|_{\infty}\\
&=:C_4e^{-(\la_{2}+\ep)t}\|f\|_{\infty},\quad\forall t\geq2,
\end{split}
\end{equation}
where we used \eqref{eqn-5-28-5} in the third inequality, and $C_{4}:=C_1C_2 C_3 e^{2(\la_{2}+\ep)}$.

We treat the terms $T^*_{t-1}\tilde{g}$, $T^*_{t-1}\PP^*_1 \tilde{g}$ and $T^*_{t-1}\PP^*_{2} \tilde{g}$ on the left-hand side of \eqref{eqn-5-28-2} to finish the proof.
It follows from Theorem \ref{thm-4-22-1} and \eqref{eqn-4-25-5-2} that
\begin{equation}\label{eqn-6-5-4}
T^*_{t-1}\tilde{g}=e^{-\frac{Q}{2}-\be_0 U} \E^{\bullet} [g(X_{t-1})\mathbbm{1}_{\{t-1<S_{\Ga}\}}]=e^{-\frac{Q}{2}-\be_0 U} \E^{\bullet} [f(X_t)\mathbbm{1}_{\{t<S_{\Ga}\}}],\quad\forall t\geq1.
\end{equation}

Noting that Lemma \ref{lem-5-28-3} (1) and Theorem \ref{thm-4-25-1} (2) give
\begin{equation*}
\begin{split}
\PP^*_1 \tilde{g}=\tilde{v}^*_1\int_{\UU}  \E^{\bullet}[f(X_1)\mathbbm{1}_{\{1<S_{\Ga}\}}]\md \nu_1=\tilde{v}^*_1e^{-\la_1}\int_{\UU}f\md \nu_1,
\end{split}
\end{equation*}
we deduce
\begin{equation}\label{eqn-6-5-3}
\begin{split}
T^*_{t-1}\PP^*_1 \tilde{g}&=T^*_{t-1}\tilde{v}^*_1e^{-\la_1}\int_{\UU}f\md \nu_1=e^{-\la_1t}\tilde{v}^*_1 \int_{\UU}f\md \nu_1,\quad \forall t\geq 1.
\end{split}
\end{equation}

For the term $T^*_{t-1}\PP^*_{2} \tilde{g}$, we show that
\begin{equation}\label{eqn-6-5-2}
    \begin{split}
    T^*_{t-1}\PP^*_2\tilde{g}
    =T^*_t\PP^*_2\tilde{f},
    \quad\forall t\geq 1,
    \end{split}
\end{equation}
where $\tilde{f}:=e^{-\frac{Q}{2}-\be_0U}f$. Due to Lemma \ref{lem-5-28-2} (1), \eqref{eqn-6-5-2} holds if we show $\PP^*_{2,i}\tilde{g}=T^*_1\PP^*_{2,i}\tilde{f}$ for all $i\in \{0,\dots, N_*-1\}$.

Recall $\tilde{g}=e^{-\frac{Q}{2}-\be_0 U} \E^{\bullet}[f(X_1)\mathbbm{1}_{\{1<S_{\Ga}\}}]$. Since $\tilde{f}$ does not necessarily belong to $L^{2}(\UU)$, we can not directly apply Theorem \ref{thm-4-22-1} to derive $\tilde{g}=T^*_1\tilde{f}$; otherwise, the conclusion follows immediately from  $\PP^*_{2,i}T^*_1=T^*_1\PP^*_{2,i}$. We proceed by approximation. Let $\{f_n\}_{n\in\N}\subset C^{\infty}_0(\UU)$ be a sequence of functions that satisfy $\|f_n\|_{\infty}\leq \|f\|_{\infty}$ for all $n\in\N$ and converge locally uniformly in $\UU$ to $f$. Since $\PP^*_{2,i}T^*_1=T^*_1\PP^*_{2,i}$, we derive from Lemma \ref{lem-5-28-3} (2) that
\begin{equation}\label{eqn-9-11-1}
    \begin{split}
    \sum_{j=1}^{d_i}\tilde{v}^{(*,2)}_{i,j}\langle T^*_1\tilde{f}_n,\tilde{v}^{(2)}_{i,j}\rangle_{L^2}=\sum_{j=1}^{d_i}T^*_1\tilde{v}^{(*,2)}_{i,j}\langle \tilde{f}_n,\tilde{v}^{(2)}_{i,j}\rangle_{L^2},\quad\forall  n\in\N,
    \end{split}
\end{equation}
where $\tilde{f}_n:=e^{-\frac{Q}{2}-\be_0 U} f_n$.
Thanks to Theorem \ref{thm-4-22-1}, we find
$$
\langle T^*_1\tilde{f}_n,\tilde{v}^{(2)}_{i,j}\rangle_{L^2}=\int_{\UU}e^{-\frac{Q}{2}-\be_0 U}\E^{\bullet}[f_n(X_1)\mathbbm{1}_{\{1<S_{\Ga}\}}]\ol{\tilde{v}^{(2)}_{i,j}}\md x,\quad\forall n\in\N.
$$
Since $\tilde{v}^{(2)}_{ij}\in \DD^*\subset L^2(\UU,\al \md x;\C)$, Lemma \ref{lem-6-5-1} ensures $e^{-\frac{Q}{2}-\be_0 U} \tilde{v}^{(2)}_{i,j}\in L^1(\UU;\C)$. Letting $n\to\infty$ in the above equality, we conclude from the dominated convergence theorem that $\lim_{n\to\infty}\langle T^*_1\tilde{f}_n,\tilde{v}^{(2)}_{i,j}\rangle_{L^2}
=\langle \tilde{g},\tilde{v}^{(2)}_{i,j} \rangle_{L^2}$. Similarly,
$\lim_{n\to\infty}\langle \tilde{f}_n,\tilde{v}^{(2)}_{i,j}\rangle_{L^2}=\langle \tilde{f},\tilde{v}^{(2)}_{i,j}\rangle_{L^2}$. Letting $n\to\infty$ in \eqref{eqn-9-11-1} then yields
\begin{equation*}\label{eqn-6-5-1}
    \sum_{j=1}^{d_i}\tilde{v}^{(*,2)}_{i,j}\langle \tilde{g},\tilde{v}^{(2)}_{i,j} \rangle_{L^2}=\sum_{j=1}^{d_i}T^*_1\tilde{v}^{(*,2)}_{i,j}\langle \tilde{f},\tilde{v}^{(2)}_{i,j} \rangle_{L^2},
\end{equation*}
which is the same as  $\PP^*_{2,i}\tilde{g}=T^*_1\PP^*_{2,i}\tilde{f}$ thanks to Lemma \ref{lem-5-28-3} (2). Hence,
 \eqref{eqn-6-5-2} follows.

Inserting \eqref{eqn-6-5-4}, \eqref{eqn-6-5-3} and  \eqref{eqn-6-5-2} into \eqref{eqn-5-28-2} yields
\begin{equation*}
\begin{split}
\left\|e^{-\frac{Q}{2}-\be_0 U} \E^{\bullet} [f(X_{t})\mathbbm{1}_{\{t<S_{\Ga}\}}]- e^{-\la_1t}\tilde{v}^*_1 \int_{\UU}f\md \nu_1-T^*_t\PP^*_{2}\Tilde{f}\right\|_{\infty}\leq C_4 e^{-(\la_{2}+\ep)t}\|f\|_{\infty},\quad\forall t\geq2.
\end{split}
\end{equation*}
Multiplying both sides by $e^{\frac{Q}{2}+\be_0U}$ concludes the proof.
\end{proof}

\begin{lem}\label{lem-6-1-1}
Assume {\bf (H1)}-{\bf (H3)}. For each $0<\ep\ll 1$, there is $C=C(\ep)>0$ such that for each $f\in C_b(\UU)$ one has
$$
\|T^*_t\PP^*_2\tilde{f}\|_{\infty}\leq C e^{(\ep-\la_2)t}\|f\|_{\infty},\quad\forall t\geq 2,
$$
where $\tilde{f}:=e^{-\frac{Q}{2}-\be_0U}f$.
\end{lem}
\begin{proof}
Fix $f\in C_b(\UU)$ and set $\Tilde{g}:=e^{-\frac{Q}{2}-\be_0 U} \E^{\bullet}[f(X_1)\mathbbm{1}_{\{1<S_{\Ga}\}}]$. It is shown in \eqref{eqn-6-5-2} that $T^*_t\PP^*_2\tilde f=T^*_{t-1}\PP^*_2\tilde{g}$ for all $t\geq 1$.
By Lemma \ref{lem-5-28-2} (1) and Lemma \ref{lem-4-25-1}, for each $0<\ep\ll 1$, there exists $C_1=C_1(\ep)>0$ such that
\begin{equation}\label{eqn-6-6-2}
\|T^*_{t-2}\PP^*_2\tilde{g}\|_{L^2}\leq C_1 e^{(\ep-\la_2)(t-2)}\|f\|_{\infty},\quad\forall t\geq 2.
\end{equation}

Thanks to Lemma \ref{lem-5-28-1}, there exists $C_2>0$ such that
$$
\|T^*_t\PP^*_2\tilde{f}\|_{\infty}=\|T^*_{t-1}\PP^*_2\tilde{g}\|_{\infty}\leq C_2 \|T^*_{t-2}\PP^*_2\tilde{g}\|_{L^2},\quad\forall t\geq 2.
$$
This together with \eqref{eqn-6-6-2} leads to this lemma.
\end{proof}

We are ready to prove Theorem \ref{thm-convergence-according-to-gap}.

\begin{proof}[Proof of Theorem \ref{thm-convergence-according-to-gap}]
Let $\nu$ and $f$ be as in the statement. For fixed $0<\ep\ll 1$, we apply Lemma \ref{lem-5-28-4} to find some $C>0$ (independent of $f$) such that
$$
\left|\E^{\bullet} [f(X_t)\mathbbm{1}_{\{t<S_{\Ga}\}}]-e^{\frac{Q}{2}+\be_0 U}\left(e^{-\la_1t}\tilde{v}^*_1 \int_{\UU}f\md \nu_1+T^*_{t}\PP^*_2\tilde{f}\right)\right|\leq C e^{\frac{Q}{2}+\be_0 U}e^{-(\la_{2}+\ep)t}\|f\|_{\infty}
$$
for all $t\geq2$, where $\Tilde{f}:=e^{-\frac{Q}{2}-\be_0 U}f$.

Since $\nu$ is compactly supported in $\UU$, integrating the above inequality on $\UU$ with respect to $\nu$ yields
\begin{equation*}
    \begin{split}
   &\left|\int_{\UU}\E^{\bullet} [f(X_t)\mathbbm{1}_{\{t<S_{\Ga}\}}]\md \nu-e^{-\la_1t}\int_{\UU} e^{\frac{Q}{2}+\be_0 U}\tilde{v}^*_1\md\nu \int_{\UU}f\md \nu_1-\int_{\UU} e^{\frac{Q}{2}+\be_0 U}T^*_{t}\PP^*_2\tilde{f}\md \nu\right|\\
   &\qquad \leq C e^{-(\la_{2}+\ep)t}\|f\|_{\infty} \int_{\UU}e^{\frac{Q}{2}+\be_0 U}\md\nu,\quad\forall t\geq2.
   \end{split}
\end{equation*}
In particular, setting $f=\mathbbm{1}_{\UU}$ yields
\begin{equation*}
    \begin{split}
   &\left|\int_{\UU}\P^x [t<S_{\Ga}]\md \nu-e^{-\la_1t}\int_{\UU} e^{\frac{Q}{2}+\be_0 U}\tilde{v}^*_1\md\nu -\int_{\UU} e^{\frac{Q}{2}+\be_0 U}T^*_{t}\PP^*_2\tilde{\mathbbm{1}}_{\UU}\md \nu\right|\\
   &\qquad \leq C e^{-(\la_{2}+\ep)t} \int_{\UU}e^{\frac{Q}{2}+\be_0 U}\md\nu,\quad\forall t\geq2,
   \end{split}
\end{equation*}
where $\tilde{\mathbbm{1}}_{\UU}:=e^{-\frac{Q}{2}-\be_0 U}\mathbbm{1}_{\UU}$.

Since $\nu$ is compactly supported in $\UU$, we apply Lemma \ref{lem-6-1-1} to derive
$$
\lim_{t\to\infty}e^{(\la_1+\ep)t}\int_{\UU}e^{\frac{Q}{2}+\be_0 U}T^*_{t}\PP^*_2\tilde{f}\md\nu=0\quad\andd\quad \lim_{t\to\infty}e^{(\la_1+\ep)t}\int_{\UU}e^{\frac{Q}{2}+\be_0 U}T^*_{t}\PP^*_2\tilde{\mathbbm{1}}_{\UU}\md\nu=0.
$$
It follows that as $t\to\infty$,
\begin{equation*}
\begin{split}
&\frac{\int_{\UU}\E^{\bullet} [f(X_t)\mathbbm{1}_{\{t<S_{\Ga}\}}]\md\nu}{\int_{\UU}\P^{\bullet}[t<S_{\Ga}]\md\nu}\\
&\qquad=\frac{e^{-\la_1t}\int_{\UU}e^{\frac{Q}{2}+\be_0 U}\tilde{v}^*_1\md\nu \int_{\UU}f\md \nu_1+\int_{\UU}e^{\frac{Q}{2}+\be_0 U}T^*_{t}\PP^*_2\tilde{f}\md\nu+o(e^{-\la_2t})}{e^{-\la_1t}\int_{\UU}e^{\frac{Q}{2}+\be_0 U}\tilde{v}^*_1\md \nu+\int_{\UU}e^{\frac{Q}{2}+\be_0 U}T^*_{t}\PP^*_2\tilde{\mathbbm{1}}_{\UU}\md \nu+o(e^{-\la_2t})}\\
&\qquad=\int_{\UU}f\md \nu_1+\frac{e^{\la_1 t}}{\int_{\UU}e^{\frac{Q}{2}+\be_0 U}\tilde{v}^*_1\md\nu}\int_{\UU}e^{\frac{Q}{2}+\be_0 U}T^*_{t}\PP^*_2\left(\tilde{f}-\tilde{\mathbbm{1}}_{\UU}\int_{\UU}f\md \nu_1\right)\md\nu+o(e^{-(\la_2-\la_{1})t}),
\end{split}
\end{equation*}
which together with Lemma \ref{lem-5-28-2} (1) leads to the result.

The last part of the lemma follows from Lemma \ref{lem-6-1-1} and Lemma \ref{lem-5-28-2} (4).
\end{proof}

\begin{rem}\label{rem-on-sharp-convergence}
Corollary \ref{cor-L2-infty} implies $\tilde{v}^*_1\in L^{\infty}(\UU)$. Lemma \ref{lem-6-1-1} gives $ \sup_{t>1}\|T^*_{t}\PP^*_2\tilde{f}\|_{\infty}<\infty$. It is then easy to see from the proof of Theorem \ref{thm-convergence-according-to-gap} that the conclusions hold for all initial distributions $\nu\in\PP(\UU)$ satisfying $\int_{\UU}e^{\frac{Q}{2}+\be_0 U}\md\nu<\infty$.
\end{rem}

\subsection{Uniqueness and exponential convergence}\label{subsec-qsd-uniqueness}

In this subsection, we study the uniqueness of QSDs of $X_{t}$ as well as the conditioned dynamics of $X_{t}$ for any initial distribution. The result is stated as follows. Recall that $\nu_1$ is the QSD of $X_{t}$ obtained in Theorem \ref{thm-4-25-1}.


\begin{thm}\label{thm-uniqueness-and-dynamics}
Assume {\bf (H1)-(H4)}. Then,  $X_{t}$ admits a unique QSD, and for each $\nu\in \PP(\UU)$ and $0<\ep\ll1$, there holds
\begin{equation*}
    \begin{split}
    &\lim_{t\to\infty} e^{(\la_2-\la_1-\ep)t}\left|\E^{\nu}\left[f(X_t)\big|t<S_{\Ga}\right] -\int_{\UU}f\md\nu_1\right|=0,\quad\forall f\in C_b(\UU).
     \end{split}
\end{equation*}
\end{thm}

We need the following result asserting that $X_t$ comes down from infinity under {\bf (H1)}-{\bf (H4)}.

\begin{lem}\label{lem-5-29-1}
Assume {\bf (H1)}-{\bf(H4)}. For each $\la>0$, there are $R=R(\la)>0$ and $C_1=C_1(\la)>0$ such that $\E^{x}\left[e^{\la S_R}\right]\leq C_1$ for all $x\in \UU\sm B^{+}_{R}$, where $S_R:=\inf\left\{t\geq0: X_t\not\in\UU\sm B_R^+\right\}$.
\end{lem}
\begin{proof}
Recall from \eqref{definition-U} that $U=V\circ \xi^{-1}$. Set $w:=\exp\left\{-\frac{\ep}{U^{\ga}}\right\}$, where $\ep>0$ is a parameter to be chosen. According to the assumptions on $V$, we can modify $V$ on a bounded domain to make sure $\inf_{\UU} V>0$, while preserving the other properties. We thus assume without loss of generality that $\inf_{\UU} V>0$. This together with $\lim_{|z|\to \infty}V(z)=\infty$ implies
\begin{equation}\label{bounds-on-w-2020-05-13-1}
    0<\inf_{\UU} w\leq \sup_{\UU} w\leq 1.
\end{equation}

Let $C$, $R_*$ and $\ga$ be as in {\bf(H4)}. Recall $\LL^{X}=\frac{1}{2}\De+(p_i-q_i)\pa_i$. Straightforward calculations give
$$
\LL^{X}U=\left(\LL^{Z}V\right)\circ\xi^{-1}\leq -CU^{1+\ga}\quad\text{in}\quad \UU\sm \xi(B^+_{R_*}).
$$
It follows that
\begin{equation*}
\begin{split}
\LL^{X}w+\la w&=\frac{\ep\ga w\LL^{X}U}{U^{\ga+1}}+\frac{1}{2}\left(a_i|\pa_{z_i} V|^2\right) \circ \xi^{-1}\left[-\frac{\ep\ga(\ga+1)}{U^{\ga+2}}+\frac{\ep^2\ga^2}{U^{2\ga+2}}\right]w+\la w\\
&\leq (-C\ep\ga+\la)w+\frac{1}{2}\left(a_i|\pa_{z_i} V|^2\right) \circ \xi^{-1}\left[-\frac{\ep\ga(\ga+1)}{U^{\ga+2}}+\frac{\ep^2\ga^2}{U^{2\ga+2}}\right]w\,\,\text{in}\,\, \UU\sm \xi(B_{R_*}^+),
\end{split}
\end{equation*}
where we used {\bf (H4)} in the inequality.

Set $\ep:=\frac{3\la}{2C\ga}$. As {\bf (H4)} ensures
$\lim_{|z|\to \infty} a_i|\pa_{z_i} V|^2\left[ -\frac{\ep\ga(\ga+1)}{V^{\ga+2}}+\frac{\ep^2\ga^2}{V^{2\ga+2}}\right] =0$,
there must exist $R>0$ such that
\begin{equation}\label{eqn-5-17-3-1}
    \LL^{X}w+\la w\leq -\frac{\la}{3}w\quad\text{in}\quad \UU\sm B_R^+.
\end{equation}

We recall from Remark \ref{rem-X-t-sde} that $X_t$ satisfies the SDE \eqref{sde-trans} before hitting $\Ga$. An application of It\^o's formula gives
\begin{equation*}
    \md e^{\la t}w(X_t)=(\LL^{X}w+\la w)(X_t)e^{\la t}\md t+\pa_{i}w(X_t) e^{\la t}\md W^i_t\quad\text{in}\quad \UU.
\end{equation*}
It follows from \eqref{eqn-5-17-3-1} that for each $(x,t)\in (\UU\sm B_R^+) \times[0,\infty)$,
\begin{equation*}
\begin{split}
    \E^{x}\left[e^{\la(t\wedge S_R)}w(X_{t\wedge S_R})\right]=w(x)+\E^{x}\left[\int_0^{t\wedge S_R} (\LL^{X}w+\la w)(X_s)e^{\la s}\md s\right]\leq w(x) ,
\end{split}
\end{equation*}
where $S_R$ is as in the statement of the lemma. Thanks to \eqref{bounds-on-w-2020-05-13-1}, we pass to the limit $t\to\infty$ in the above inequality to conclude $\E^{x}\left[e^{\la S_R}\right]\leq \frac{1}{\inf w}$ for all $x\in \UU\sm B^{+}_{R}$. This completes the proof.
\end{proof}

\begin{rem}\label{lem-5-29-1-1}
Since $Z_t=\xi^{-1}(X_t)$ and $\xi^{-1}:\ol{\UU}\to \ol{\UU}$ is a homeomorphism, we find from the above lemma that for each $\la>0$, there exists $R=R(\la)>0$ such that $\sup_{z\in \UU\sm B^+_{R}}\E^z[e^{\la T_R}]<\infty$, where $T_R:=\inf\{t\geq 0: Z_t\notin \UU\sm B^+_R\}$.
\end{rem}

We next prove Theorem \ref{thm-uniqueness-and-dynamics}.

\begin{proof}[Proof of Theorem \ref{thm-uniqueness-and-dynamics}]
 Fix $\nu\in \PP(\UU)$, $f\in C_b(\UU)$ and $0<\ep\ll 1$. Set $\la:=\la_1+\la_2$. By Lemma \ref{lem-5-29-1}, there exist $R_0>0$ and $C_1>0$ such that
\begin{equation}\label{eqn-5-29-8}
\sup_{(x,t)\in (\UU\sm B_{R_0}^+) \times[0,\infty) }e^{\la t}\P^{x}\left[t<S_{R_0}\right]\leq \sup_{x\in \UU\sm B^+_{R_0}}\E^{x}\left[e^{\la S_{R_0}}\right]\leq C_{1}.
\end{equation}
Clearly, the above inequality holds with $R>R_0$ replacing $R_0$. Choosing $R_0$ large enough, we may assume without loss of generality that $\nu(B^+_{R_0})>0$. We split
$$
\E^{\nu}\left[f(X_t)\mathbbm{1}_{\{t<S_{\Ga}\}}\right]=\int_{B^+_{R_0}} \E^{\bullet}\left[f(X_t)\mathbbm{1}_{\{t<S_{\Ga}\}}\right]\md\nu+\int_{\UU\sm B^+_{R_0}} \E^{\bullet}\left[f(X_t)\mathbbm{1}_{\{t<S_{\Ga}\}}\right]\md\nu,\quad\forall t\geq 0.
$$

Applying Lemma \ref{lem-5-28-4} and Lemma \ref{lem-6-1-1}, we find the existence of $C_2>0$ such that
\begin{equation}\label{eqn-5-29-7}
\begin{split}
\left|\E^{\bullet}\left[f(X_t)\mathbbm{1}_{\{t<S_{\Ga}\}}\right]-e^{\frac{Q}{2}+\be_0 U}e^{-\la_1 t}\tilde{v}^*_1\int_{\UU}f\md\nu_1\right|\leq C_2 e^{\frac{Q}{2}+\be_0 U} e^{(\ep-\la_{2})t}\|f\|_{\infty},\quad\forall t\geq2.
\end{split}
\end{equation}
It follows that
\begin{equation}\label{eqn-6-3-1}
    \begin{split}
        \left| \int_{B^+_{R_0}}\E^{\bullet}\left[f(X_t)\mathbbm{1}_{\{t<S_{\Ga}\}}\right]\md \nu-A_1e^{-\la_1 t}\int_{\UU}f\md\nu_1\right|\leq C_2 e^{(\ep-\la_{2})t}\|f\|_{\infty} \int_{B^+_{R_0}} e^{\frac{Q}{2}+\be_0 U}\md\nu, \quad\forall t\geq 2,
    \end{split}
\end{equation}
where $A_1:=\int_{B^+_{R_0}}e^{\frac{Q}{2}+\be_0 U}\tilde{v}^*_1\md\nu$. Since $\tilde{v}^*_1>0$ in $\UU$ and $\nu(B^+_{R_0})>0$, there holds $A_1>0$.

We claim the existence of a bounded function $A_2:[0,\infty)\to [0,\infty)$ and a $C_3>0$ such that
\begin{equation}\label{eqn-6-3-2}
    \begin{split}
      &\left|\int_{\UU\sm B^+_{R_0}}\E^{\bullet}\left[f(X_t)\mathbbm{1}_{\{ t<S_{\Ga}\}}\right]\md \nu-A_2(t)e^{-\la_1 t}\int_{\UU}f\md\nu_1\right|\leq C_3 e^{(\ep-\la_{2})t}\|f\|_{\infty},\quad\forall t\gg 1.
    \end{split}
\end{equation}
This together with \eqref{eqn-6-3-1} leads to the existence of $C_4>0$ such that
\begin{equation*}
    \begin{split}
      &\left|\E^{\nu}\left[f(X_t)\mathbbm{1}_{\{ t<S_{\Ga}\}}\right]-(A_1+A_2(t))e^{-\la_1 t}\int_{\UU}f\md\nu_1\right| \leq C_4 e^{(\ep-\la_{2})t}\|f\|_{\infty},\quad\forall t\gg 1.
    \end{split}
\end{equation*}
In particular, setting $f=\mathbbm{1}_{\UU}$ yields
$\left|\P^{\nu}[t<S_{\Ga}]-(A_1+A_2(t))e^{-\la_1 t}\right|\leq C_4 e^{(\ep-\la_{2})t}$ for all $t\gg 1$. Since $A_1>0$, this implies $\P^{\nu}[t<S_{\Ga}]>0$ for $t\gg 1$. Consequently, we deduce
\begin{equation*}\label{eqn-5-29-5}
    \begin{split}
    &\left|\E^{\nu}\left[f(X_t)\big |t<S_{\Ga}\right]-\int_{\UU}f\md\nu_1\right|\\
    &\qquad\leq \left| \frac{\E^{\nu}\left[f(X_t)\mathbbm{1}_{\{t<S_{\Ga}\}}\right] }{\P^{\nu}\left[t<S_{\Ga}\right] }-\frac{(A_1+A_2(t))e^{-\la_1 t} \int_{\UU}f\md\nu_1}{\P^{\nu}[t<S_{\Ga}]}\right|\\
    &\qquad\quad+\left|\frac{(A_1+A_2(t))e^{-\la_1 t}\int_{\UU}f\md\nu_1}{\P^{\nu}[t<S_{\Ga}]}-\int_{\UU}f\md\nu_1\right|\\
    &\qquad= \frac{1}{\P^{\nu}\left[t<S_{\Ga}\right]}\left|\E^{\nu}\left[f(X_t)\mathbbm{1}_{\{t<S_{\Ga}\}}\right]-(A_1+A_2(t))e^{-\la_1 t}\int_{\UU}f\md\nu_1\right|\\
    &\qquad\quad+\frac{ \int_{\UU}|f|\md\nu_1}{\P^{\nu}[t<S_{\Ga}]} \left|(A_1+A_2(t))e^{-\la_1 t}-\P^{\nu}[t<S_{\Ga}]\right|\\
    &\qquad\leq \frac{2C_4e^{(\ep-\la_2)t}\|f\|_{\infty}}{\P^{\nu}[t<S_{\Ga}]},\quad\forall t\gg 1.
     \end{split}
\end{equation*}
The theorem follows immediately.

\medskip
It remains to prove \eqref{eqn-6-3-2}. To do so, we write for $(x,t)\in (\UU\sm B^+_{R_0})\times [0,\infty)$,
\begin{equation*}\label{eqn-5-3-4-1}
\begin{split}
\E^x[f(X_t)\mathbbm{1}_{\{t<S_{\Ga}\}}]=\E^x[f(X_t)\mathbbm{1}_{\{t<S_{R_0}\}}]+\E^x[f(X_t)\mathbbm{1}_{\{S_{R_0}\leq t<S_{\Ga}\}}]=:E_{1}(x,t)+E_{2}(x,t).
\end{split}
\end{equation*}
It follows from \eqref{eqn-5-29-8} that
\begin{equation}\label{eqn-6-1-1111}
\begin{split}
\int_{\UU\sm B^+_{R_{0}}}|E_{1}(\cdot,t)|\md\nu&\leq  \|f\|_{\infty} \sup_{x\in \UU\sm B^+_{R_{0}}}\P^{x}\left[t<S_{R_{0}}\right]\md \nu\\
&\leq \|f\|_{\infty}e^{-\la t}\sup_{x\in \UU\sm B^+_{R_{0}}}\E^x[e^{\la S_{R_0}}]\leq C_{1}\|f\|_{\infty}e^{-\la t},\quad\forall t\geq0.
\end{split}
\end{equation}

To treat $E_{2}$, we set $h(x,t):=\E^{x}\left[f(X_t)\mathbbm{1}_{\{t<S_{\Ga}\}}\right]$ for $(x,t)\in\ol{\UU}\times[0,\infty)$. Obviously, $\|h\|_{\infty}\leq \|f\|_{\infty}$ and $h(x,t)=0$ for $(x,t)\in \Ga\times[0,\infty)$. The strong Markov property and homogeneity of $X_{t}$ yield that for each $(x,t)\in (\UU\sm B^+_{R_{0}})\times[0,\infty)$,
\begin{equation*}\label{eqn-5-31-2}
\begin{split}
E_{2}(x,t)&=\E^{x}\left[f(X_t)\mathbbm{1}_{\{S_{R_{0}}\leq t<S_{\Ga}\}}\right]\\
&=\E^{x}\left[h(X_{S_{R_{0}}},t-S_{R_{0}})\mathbbm{1}_{\{S_{R_0}\leq t\}}\right]\\
&=\E^{x}\left[h(X_{S_{R_{0}}},t-S_{R_{0}})\mathbbm{1}_{\{S_{R_0}\leq t\leq  S_{R_0}+2\}}\right]+\E^{x}\left[h(X_{S_{R_{0}}},t-S_{R_{0}})\mathbbm{1}_{\{t>  S_{R_0}+2\}}\right]\\
&=:E_{21}(x,t)+E_{22}(x,t).
\end{split}
\end{equation*}

Note that \eqref{eqn-5-29-8} ensures
\begin{equation}\label{eqn-5-31-3}
    \begin{split}
    \int_{\UU\sm B^+_{R_0}}|E_{21}(\cdot,t)|\md \nu&\leq \|h\|_{\infty}\int_{\UU\sm B^+_{R_0}}\P^x[t< S_{R_0}+2]\md\nu\\
    &\leq \|f\|_{\infty} e^{-\la t}  \int_{\UU\sm B^+_{R_0}}\E^x[e^{\la (S_{R_0}+2)}]\md \nu\leq C_1\|f\|_{\infty} e^{2\la -\la t},\quad\forall t\geq 0.
    \end{split}
\end{equation}

Fix $0<\ep\ll 1$. Setting $\Phi:=\exp\left\{\frac{Q(X_{S_{R_{0}}})}{2}+\be_0 U(X_{S_{R_{0}}})\right\}$, we see from \eqref{eqn-5-29-7} that on the event $\{t\geq  S_{R_0}+2\}$ there holds
\begin{equation*}
\begin{split}
\left| h(X_{S_{R_{0}}},t-S_{R_{0}})-\Phi e^{-\la_1 (t-S_{R_0})}\tilde{v}^*_1(X_{S_{R_0}})\int_{\UU}f\md\nu_1\right|\leq C_2\Phi e^{(\ep-\la_{2})(t-S_{R_0})}\|f\|_{\infty}.
\end{split}
\end{equation*}

Since $S_{R_0}\leq S_{\Ga}$ and $h(X_{S_{R_0}},t-S_{R_0})=0$ if $S_{R_0}=S_{\Ga}$, we deduce
$$
E_{22}(x,t)=\E^{x}\left[ h(X_{S_{R_0}},t-S_{R_0})\mathbbm{1}_{\{S_{R_0}<S_{\Ga}\wedge(t-2)\}}\right],\quad\forall(x,t)\in (\UU\sm B^+_{R_0})\times[0,\infty).
$$
As a result, there holds
\begin{equation}\label{eqn-6-1-3}
    \begin{split}
    &\left|\int_{\UU\sm B^+_{R_0}}E_{22}(\cdot,t)\md \nu-e^{-\la_1 t}\int_{\UU\sm B^+_{R_0}}\E^{\bullet}\left[\Phi e^{\la_1 S_{R_0}} \tilde{v}^*_1(X_{S_{R_0}})\mathbbm{1}_{\{S_{R_0}<S_{\Ga}\wedge(t-2)\}}\right]\md \nu\int_{\UU}f\md\nu_1\right|\\
    &\qquad\leq C_2 e^{(\ep-\la_{2})t}\int_{\UU\sm B^+_{R_0}} \E^{\bullet} \left[\Phi e^{-(\ep-\la_2)S_{R_0}}\mathbbm{1}_{\{S_{R_0}<S_{\Ga}\wedge(t-2)\}}\right]  \md\nu \|f\|_{\infty}\\
    &\qquad\leq C_2 e^{(\ep-\la_{2})t}\|f\|_{\infty}\left(\max_{\UU\cap\pa B^+_{R_0}} e^{\frac{Q}{2}+\be_0 U}\right)\left(\sup_{\UU\sm B^+_{R_0}}\E^{\bullet}\left[e^{-(\ep-\la_2)S_{R_0}}\right]\right)\\
    &\qquad\leq C_5e^{(\ep-\la_{2})t}\|f\|_{\infty},\quad\forall t\geq 0.
     \end{split}
\end{equation}
where we used \eqref{eqn-5-29-8} and the fact $\max_{\UU\cap \pa B^+_{R_0}}e^{\frac{Q}{2}+\be_0 U}<\infty$ to conclude the existence of $C_5>0$ in the last inequality.

Set
$$
A_2(t):=\int_{\UU\sm B^+_{R_0}}\E^{\bullet}\left[\Phi e^{\la_1 S_{R_0}} \tilde{v}^*_1(X_{S_{R_0}})\mathbbm{1}_{\{S_{R_0}<S_{\Ga}\wedge(t-2)\}}\right]\md \nu,\quad \forall t\geq 0.
$$
Obviously, $A_2$ is non-negative and bounded. Since
$$
\int_{\UU\sm B^+_{R_0}}\E^{\bullet}\left[f(X_t)\mathbbm{1}_{t<S_{\Ga}\}}\right] \md\nu=\int_{\UU\sm B^+_{R_0}}\left[E_{1}(\cdot,t)+E_{21}(\cdot,t)+E_{22}(\cdot,t)\right]\md\nu,\quad\forall t\geq0,
$$
we deduce from \eqref{eqn-6-1-1111}, \eqref{eqn-5-31-3} and \eqref{eqn-6-1-3} that
\begin{equation*}
    \begin{split}
      \left|\int_{\UU\sm B^+_{R_0}}\E^{\bullet}\left[f(X_t)\mathbbm{1}_{t<S_{\Ga}\}}\right] \md\nu-A_2(t)e^{-\la_1 t}\int_{\UU}f\md\nu_1\right|\leq \left[C_5 e^{(\ep-\la_{2})t}+C_1(1+e^{2\la}) e^{-\la t}\right]\|f\|_{\infty}
    \end{split}
\end{equation*}
for all $t\geq 0$. Since $\la=\la_1+\la_2$ and $0<\ep\ll 1$, \eqref{eqn-6-3-2} follows. This completes the proof.
\end{proof}

\subsection{Proof of Theorem \ref{thm-qsd-existence-dynamics} and Theorem \ref{thm-uniqueness}}\label{subsec-proof-thm-A-B}

Because of the fact $X_{t}=\xi(Z_{t})$ and Proposition \ref{prop-equi-qsd}, conclusions in Theorem \ref{thm-qsd-existence-dynamics} and Theorem \ref{thm-uniqueness} follow directly from Theorem \ref{thm-4-25-1}, Theorem \ref{thm-convergence-according-to-gap} and Theorem \ref{thm-uniqueness-and-dynamics}.


\section{\bf Applications}\label{sec-app}

In this section, we discuss a series of important applications of Theorem \ref{thm-qsd-existence-dynamics} and Theorem \ref{thm-uniqueness}. We first provide a general result that holds for most ecological models and then show how to apply this result to specific situations, including:  stochastic Lotka-Volterra systems of competitive, predator-prey or cooperative type, systems modelled by Holling type functional responses and predator-prey systems modelled by Beddington-DeAngelis functional responses.

Consider the following stochastic system:
\begin{equation}\label{eqn-6-16-1}
    \md Z^i_{t}=Z^i_{t}f_i(Z_{t})\md t+\sqrt{\ga_{i}Z^i_{t}}\md W^i_t,\quad i\in\{1,\dots, d\},
\end{equation}
where $Z_{t}=(Z^i_{t})\in\ol{\UU}$, $\{f_i\}_i$ belong to $ C^{1}(\ol{\UU})$, $\{\ga_{i}\}_i$ are positive constants, and $\{W^i\}_i$  are independent standard one-dimensional Wiener processes on some probability space. We make the following assumption.
\begin{enumerate}
\item [{\bf (A)}] There exist $m\geq 0$, $0\leq n\leq m$, $C_1, C_2, C_3, C_4>0$ and $R>0$ such that
\begin{equation}\label{eqn-6-26-1}
-C_1\left(1+\sum_{j=1}^d z_j^m\right)\leq f_i(z)\leq C_2\mathbbm{1}_{[0,R]}(z_i)-C_3z_i^m\mathbbm{1}_{(R,\infty)}(z_i)+\de\sum_{j\neq i} z_j^n,\quad\forall z\in\ol{\UU},
\end{equation}
and
\begin{equation}\label{eqn-7-9-1}
    |\pa_{z_i} f_i(z)|\leq C_4 |z|^{m-1},\quad\forall z\in \UU\sm B^+_{R},
\end{equation}
for all $i\in\{1,\dots, d\}$ and $\de\geq 0$ if $n<m$ or $\de\in \left[0,\frac{C_3}{d-1}\right)$ if $n=m$.
\end{enumerate}

\begin{rem}
Conditions \eqref{eqn-6-26-1} and \eqref{eqn-7-9-1} say that $f_i$ and $\partial_{z_i}f_i$ are bounded above and below by simple polynomials.  Conditions in the case $n<m$ tells us that the intraspecific competition dominates the interactions among species. In the case $n=m$, we can only treat weakly cooperative interactions among species -- this is reflected by the smallness of $\de$. These are natural assumptions that can be applied to many population dynamics models: competitive Lotka-Volterra, weakly cooperative Lotka-Volterra, predator-prey Lotka-Volterra as well as more complex systems modelled by Holling type-II/III functional responses. These assumptions also allow us to use a very simple Lyapunov function $V(z)=|z|^{m+1}$ which satisfies {\bf (H1)}-{\bf (H3)} and sometimes {\bf(H4)}.
\end{rem}

Under the assumption {\bf(A)}, the stochastic system \eqref{eqn-6-16-1} generates a diffusion process $Z_{t}$ that has $\Ga$ as an absorbing set. Furthermore, $Z_t$ hits $\Ga$ in finite time almost surely.

\begin{thm}\label{thm-app}
Assume {\bf (A)}.
\begin{enumerate}
    \item[\rm(1)] $Z_{t}$ admits a QSD $\mu_1$, and there exists  $r_1>0$ such that
\begin{itemize}
\item for any $0<\ep\ll1$ and $\mu\in\PP(\UU)$ with compact support in $\UU$ one has
$$
\lim_{t\to\infty}e^{(r_{1}-\ep)t}\left|\E^{\mu}\left[f(Z_{t})\big|t<T_{\Ga}\right]-\int_{\UU}f\md\mu_{1}\right|=0,\quad\forall f\in C_{b}(\UU);
$$
\item there exists $f\in C_{b}(\UU)$ such that for a.e. $x\in \UU$, there is a discrete set $\II_x\subset(0,\infty)$ with distances between adjacent points admitting an $x$-independent positive lower bound, such that for each $0<\de\ll 1$ one has
$$
\lim_{\substack{t\to\infty\\t\in (0,\infty)\sm\II_{x,\de}}} e^{(r_1+\ep)t}\left|\E^{x} [f(X_t)\big |t<T_{\Ga}]-\int_{\UU}f\md\mu_1\right|=\infty,\quad\forall 0<\ep\ll 1,
$$
where $\II_{x,\de}$ is the $\de$-neighbourhood of $\II_x$ in $(0,\infty)$.
\end{itemize}

\item[\rm(2)] If, in addition, {\bf (A)} holds with $m>0$, then $Z_{t}$ admits a unique QSD, and for any $0<\ep\ll1$ and $\mu\in\PP(\UU)$, there holds
$$
\lim_{t\to\infty}e^{(r_{1}-\ep)t}\left|\E^{\mu}\left[f(Z_{t})\big|t<T_{\Ga}\right]-\int_{\UU}f\md\mu_{1}\right|=0,\quad\forall f\in C_{b}(\UU).
$$
\end{enumerate}
\end{thm}
\begin{proof}
Let $m,\,\, C_1,\,\, C_2, \,\,C_3,\,\,C_4,\,\, R$ and $\de$ be as in {\bf (A)}.
Set $V(z):=|z|^{m+1}$ for $z\in\UU$. Since $\pa_i V=(m+1)|z|^{m-1}z_i$, we deduce from {\bf (A)} that
\begin{equation}\label{eqn-2020-10-16}
\begin{split}
\sum_{i=1}^d z_if_i\pa_i V&\leq (m+1)|z|^{m-1}\sum_{i=1}^d z_i^2\left( C_2\mathbbm{1}_{[0,R]}(z_i)-C_3 z_i^m\mathbbm{1}_{(R,\infty)}(z_i)+\de\sum_{j\neq i} z_j^n\right)\\
&\leq (m+1)|z|^{m-1}\left[\left(C_2R+C_3R^{m+1}\right)\sum_{i=1}^d z_i-C_3 \sum_{i=1}^d  z_i^{m+2}+\de\sum_{i=1}^d\sum_{j\neq i} z_i^2 z_j^n \right]
\end{split}
\end{equation}
for all $z\in\UU$.

We apply Young's inequality to find
$$
z_i^2 z_j^n\leq
\begin{cases}\frac{2\al}{m+2}z_i^{m+2}+ \frac{m\al^{-\frac{2}{m}}}{m+2}z_j^{\frac{n(m+2)}{m}},\quad& \text{if}\quad  m>0\andd n\in[0,m],\\
\frac{2\al}{m+2}z_i^{m+2}+ \frac{m\al^{-\frac{2}{m}}}{m+2}z_j^{m+2},\quad&\text{if}\quad m=n=0,
\end{cases}
$$
where $\al>0$ is a parameter to be determined. For convenience, we set $\be(m,n)=\frac{n(m+2)}{m}\in [0,m+2]$ if $m>0$ and $n\in[0,m]$ and $\be(n,m)=m+2$ if $m=n=0$. Thus, it follows from \eqref{eqn-2020-10-16} that
\begin{equation}\label{eqn-2020-10-13}
    \begin{split}
       \sum_{i=1}^d z_if_i\pa_i V&\leq (m+1)(C_2R+C_3R^{m+1})|z|^{m-1}\sum_{i=1}^d z_i-(m+1)C_3|z|^{m-1} \sum_{i=1}^d z_i^{m+2}\\
       &\quad+\de(m+1)|z|^{m-1}\sum_{i=1}^d\sum_{j\neq i}\left( \frac{2\al}{m+2}z_i^{m+2}+ \frac{m\al^{-\frac{2}{m}}}{m+2}z_j^{\be(m,n)} \right)\\
       &= (m+1)(C_2R+C_3R^{m+1})|z|^{m-1}\sum_{i=1}^d z_i\\
       &\quad+\frac{\de m(m+1)\al^{-\frac{2}{m}}(d-1)}{m+2}|z|^{m-1}\sum_{i=1}^d z_i^{\be(m,n)}\\
       &\quad-(m+1)\left(C_3-\frac{2\de\al(d-1)}{m+2}\right) |z|^{m-1}\sum_{i=1}^d z_i^{m+2}
    \end{split}
\end{equation}
for all $z\in\UU$.

Note that $0<\be(m,n)<m+2$ if $n<m$ and $\be(n,m)=1$ if $n=m$. We consider the following two cases.
\begin{itemize}
    \item If $n<m$, we set $\al=\frac{(m+2)C_3}{4\de (d-1)}$ in \eqref{eqn-2020-10-13} (so that $C_3-\frac{2\de\al (d-1)}{m+2}=\frac{1}{2} C_3>0$) to find the existence of $C_5, R_1>0$ such that
    \begin{equation}\label{eqn-7-2-1}
        \sum_{i=1}^d z_if_i\pa_i V\leq -C_5|z|^{2m+1}\quad \text{in} \quad\UU\sm B^+_{R_1}.
    \end{equation}
    \item If $n=m$, setting $\al =1$ in \eqref{eqn-2020-10-13} leads to
    \begin{equation*}
    \begin{split}
        \sum_{i=1}^d z_i f_i \pa_i V&=(m+1)(C_2R+C_3R^{m+1})|z|^{m-1}\sum_{i=1}^d z_i\\
        &\quad-(m-1)[C_3-\de(d-1)]|z|^{m-1}\sum_{i=1}^d z_i^{m+2}.
    \end{split}
    \end{equation*}
    Then, it is easy to conclude from $\de\in [0,\frac{C_3}{d-1})$ the existence of positive constants $C'_5$ and $R'_1$ such that \eqref{eqn-7-2-1} holds with $C_5$ and $R_1$ replaced by $C'_5$ and $R'_1$, respectively.
\end{itemize}

As a result, we no longer distinguish the above two cases and assume \eqref{eqn-7-2-1} always holds for some $C_5>0$ and $R_1>0$.

Now, we verify {\bf (H1)}-{\bf (H3)}. It is easy to check that {\bf (H1)} and {\bf (H2)} hold. As $V\geq C(d)\sum_{i=1}^d z_i^{m+1}$ in $\UU$ for some $C(d)>0$ and $\int_{1}^{\infty}\frac{1}{s}\exp\left\{-\be s^{m+1}\right\}\md s<\infty$ for any $\be>0$, {\bf (H3)} (1)(2) follow from \eqref{eqn-7-2-1}.
Since
\begin{equation*}
\begin{split}
    \pa_i(z_if_i)&=f_i(z)+z_i\pa_{z_i}f_i(z),\\
    \ga_iz_i\pa^2_{z_iz_i} V&=\ga_i (m+1)(m-1)|z|^{m-3}z_i^3+\ga_i (m+1)|z|^{m-1}z_i,\\
    \ga_iz_i |\pa_{z_i}V|^2&=\ga_i(m+1)^2|z|^{2m-2}z_i^3,
\end{split}
\end{equation*}
it is straightforward to verify {\bf (H3)} (3)(4) by applying \eqref{eqn-6-26-1},  \eqref{eqn-7-9-1} and \eqref{eqn-7-2-1}. Hence, an application of Theorem \ref{thm-qsd-existence-dynamics} gives the conclusions in (1).

If $m>0$, {\bf (H4)} holds with $\ga:=\frac{m}{m+1}$. The conclusion in (2) follows from Theorem \ref{thm-uniqueness}.
\end{proof}

In the following, we apply Theorem \ref{thm-app} to various important ecological models.

\begin{ex}[Lotka-Volterra systems]\label{ex-logistic-feller}
For each $i\in\{1,\dots, d\}$ let
$$
f_i(z)=r_i-\sum_{j=1}^d c_{ij} z_j,\quad z\in\ol{\UU},
$$
where $r_{i}\in\R$, $c_{ii}>0$ and $c_{ij}\in\R$ for $j\neq i$.
\end{ex}

\begin{cor}
Consider the stochastic system \eqref{eqn-6-16-1} with  $f_{i}$, $i\in\{1,\dots,d\}$ being as in Example \ref{ex-logistic-feller}. Assume
\begin{equation}\label{e:LV}
-\min_{i\neq j}c_{ij}< \frac{1}{d-1}\min_{i}c_{ii}.
\end{equation}
Then, there exists a unique QSD of \eqref{eqn-6-16-1} such that the conclusions of Theorem \ref{thm-app} hold.
\end{cor}
\begin{proof}
It is straightforward to check that the assumption {\bf (A)} with $m=n=1$, $C_3=\min_{i}c_{ii}$ and $\de=-\min_{i\neq j}c_{ij}$ is satisfied. The corollary then follows from Theorem \ref{thm-app}.
\end{proof}
\begin{rem}
If the system is competitive, namely, $c_{ij}\geq 0$ for all $i\neq j$, then \eqref{e:LV} is trivially satisfied. If the Lotka-Volterra system has either cooperation or predation, the condition \eqref{e:LV} says that the intraspecific competition terms have to dominate in some sense the cooperative and the predation terms. Note that cooperative systems are known to behave poorly: see \cite[Example 2.3]{HN18} for details as to how a two-species stochastic cooperative system can exhibit either blow-up in finite time or have no stationary distributions.
\end{rem}

\begin{ex}[Holling type-\RN{2}/\RN{3} functional response]\label{ex-2}
For each $i\in\{1,\dots, d\}$,
$$
f_i(z)=r_i-\sum_{j=1}^d \frac{c_{ij} z^{k}_j}{1+z^{k}_j},\quad z\in\ol{\UU},
$$
where $k\in\{1,2\}$, $r_{i}\in\R$, $c_{ii}>0$ and $c_{ij}\in\R$ for $j\neq i$. In literature, $k=1$ and $k=2$ correspond to Holling type-\RN{2} and -\RN{3} functional responses, respectively.
\end{ex}

\begin{cor}\label{cor-holling-2}
Consider the stochastic system \eqref{eqn-6-16-1} with  $f_{i}$, $i\in\{1,\dots,d\}$ being as in Example \ref{ex-2}. Assume
$$c_{ii}>r_i,\quad\forall i\in \{1,\dots, d\}\quad\andd\quad -\min_{i\neq j}c_{ij}< \frac{1}{d-1}\min_{i}(c_{ii}-r_i).
$$
Then, the conclusions of Theorem \ref{thm-app} (1) hold.
\end{cor}
\begin{proof}
 By Theorem \ref{thm-app}, it suffices to verify the assumption {\bf (A)} with $m=n=0$. Clearly,  the first inequality in \eqref{eqn-6-26-1} holds. To check the second inequality in \eqref{eqn-6-26-1}, we note from the assumptions that there exists $\al\in(0,1)$ such that
$-\min_{i\neq j}c_{ij}< \frac{\al}{d-1}\min_{i}\{c_{ii}-r_i\}$. For this $\al>0$, there exists $R>0$ such that
$$
r_i-\frac{c_{ii}z^k_i}{1+z^k_i}\leq \al (r_i-c_{ii})\leq -\al\min_{i}(c_{ii}-r_i),\quad\forall z_i\in (R,\infty) \andd i\in\{1,\dots,d\},
$$
leading to
\begin{equation*}
    r_i-\sum_{j=1}^d\frac{c_{ij}z^k_j}{1+z^k_j}\leq
    \begin{cases}
    r_i-(d-1)\times\min_{i\neq j}c_{ij},&\forall z\in\{z\in\ol{\UU}: z_i\in [0,R]\},\\
    -\al\min_{i}(c_{ii}-r_i)-(d-1)\times\min_{i\neq j}c_{ij},& \forall z\in\{z\in\ol{\UU}: z_i\in (R,\infty)\}.
    \end{cases}
\end{equation*}
This shows the second inequality in \eqref{eqn-6-26-1} with $C_2=\max_i\{r_i\}$, $C_3=\al\min_{i}(c_{ii}-r_i)$ and $\de=-\min_{i\neq j} c_{ij}$.
Straightforward calculations give \eqref{eqn-7-9-1}. Hence, the assumption {\bf (A)} with $m=0$ holds.
\end{proof}




\begin{rem}
For the stochastic Lotka-Volterra system with Holling type-\RN{2}/\RN{3} functional response considered in Example \ref{ex-2} or Corollary \ref{cor-holling-2}, the existence of a unique QSD that attracts all initial distributions supported in $\UU$ is not expected. This is essentially due to the weak dissipativity of the system. Indeed, in the case $d=1$, these properties are equivalent to showing that the process comes down from infinity, and therefore, according to \cite[Theorem 7.3 and Proposition 7.5]{CCLMMS09}, equivalent to Assumption (H5) in \cite{CCLMMS09}. However, it is easy to check that (H5) in \cite{CCLMMS09} is not satisfied for the Holling type-II/III functional responses.

The situation in higher dimensions is worse. Even in the competitive case, the dissipativity of the system is weaker than that of the system with $f_{i}(z)=r_{*}-c_{*}\sum_{j=1}^{d}\frac{z_{j}^{k}}{1+z_{j}^{k}}$ for all $i\in\{1,\dots,d\}$, where $r_{*}=\min_{i\in\{1,\dots,d\}}r_{i}$ and $c_{*}=\max_{i,j\in\{1,\dots,d\}}c_{ij}$. This latter system does not come down from infinity as it is bounded from below by a decoupled system whose individual components do not come down from infinity. In fact, we have
$$
r_{*}-c_{*}\sum_{j=1}^{d}\frac{z_{j}^{k}}{1+z_{j}^{k}}\geq r_{*}-c_{*}(d-1)-c_{*}\frac{z_{i}^{k}}{1+z_{i}^{k}},\quad\forall i\in\{1,\dots,d\}\andd z\in\ol{\UU}.
$$
Hence, the stochastic system in Example \ref{ex-2} or Corollary \ref{cor-holling-2} does not come down from infinity.
\end{rem}
We exhibit below a few more types of functional responses that can be treated by our framework.
\begin{ex}\label{ex-regular-holling}
Consider the functional response
$$
f_i(z)=r_i-c_{ii}z_i-\sum_{j\neq i} \frac{c_{ij} z^{k}_j}{1+z^{k}_j},\quad z\in\ol{\UU},
$$
where $k\in\{1,2\}$, $r_{i}\in\R$, $c_{ii}>0$ and $c_{ij}\in\R$ for $j\neq i$.
This is a combination of the regular intraspecific competition of the form $-c_{ii}z_i$ and Holling type functional responses for the interspecific competition/predation.
\end{ex}

\begin{cor}
Consider the stochastic system \eqref{eqn-6-16-1} with  $f_{i}$, $i\in\{1,\dots,d\}$ being as in Example \ref{ex-regular-holling}. Then, there exists a unique QSD of \eqref{eqn-6-16-1} such that the conclusions of Theorem \ref{thm-app} hold.
\end{cor}
\begin{proof}
It is straightforward to check that Assumption {\bf (A)} holds with $m=1$ and $n=0$. Then, the application of Theorem \ref{thm-app} yields the conclusion.
\end{proof}

\begin{ex}\label{ex-BD-predator-prey}
{Consider the extensively used Beddington-DeAngelis predator-prey dynamics.
For each $i\in\{1,\dots, d\}$, let
$$
f_i(z)=r_i-c_{ii}z_i-\sum_{j\neq i} \frac{c_{ij} z_j}{1+\sum_{l=1}^d z_l},\quad z\in\ol{\UU},
$$
where $r_{i}\in\R$, $c_{ii}>0$, and $c_{ij}\in\R$ for $j\neq i$.
This system was first proposed in \cite{B75,D75} in order to better explain certain predator-prey interactions.}
\end{ex}

\begin{cor}
Consider the stochastic system \eqref{eqn-6-16-1} with  $f_{i}$, $i\in\{1,\dots,d\}$ being as in Example \ref{ex-BD-predator-prey}. Then, there exists a unique QSD of \eqref{eqn-6-16-1} such that the conclusions of Theorem \ref{thm-app} hold.
\end{cor}
\begin{proof}
It is straightforward to check that Assumption {\bf (A)} holds with $m=1$ and $n=0$. Then, the application of Theorem \ref{thm-app} yields the conclusion.
\end{proof}

\begin{ex}\label{ex-CM}
Let $d=2$. Consider the Crowley-Martin dynamics.
Let
$$
f_1(z)=r_1-c_{11}z_1-z_2\frac{z_1}{\beta+\alpha z_1 + \alpha_2 z_2+\alpha_3 z_1 z_2},\quad z\in\ol{\UU},
$$
$$
f_2(z)=-r_2-c_{22}z_2+z_1\frac{z_1}{\beta+\alpha z_1 + \alpha_2 z_2+\alpha_3 z_1 z_2},\quad z\in\ol{\UU},
$$
where $c_{11}, c_{22}, \beta >0$ and all the other quantities are nonnegative. This system was first proposed in \cite{CM89} to study dragonflies.
\end{ex}

\begin{cor}
Consider the stochastic system \eqref{eqn-6-16-1} in the case $d=2$ with  $f_1$ and $f_{2}$  being as in Example \ref{ex-CM}. Assume $\al> \frac{2}{3\min\{2c_{11},c_{22}\}}$. Then, there exists a unique QSD of \eqref{eqn-6-16-1} such that the conclusions of Theorem \ref{thm-app} hold.
\end{cor}
\begin{proof}
Note that $f_1(z)\leq r_1-c_{11}z_1$ and $f_2(z)\leq -r_2-c_{22}z_2+\frac{z_1}{\al}$. Following the arguments as in the proof of Theorem \ref{thm-app}, it is straightforward to see that $V(z):=|z|^2$ for $z\in \UU$ is a Lyapunov function satisfying {\bf (H1)}-{\bf (H4)}. From which, the conclusions of Theorem \ref{thm-app} hold.
\end{proof}


\appendix

\section{\bf Proof of technical lemmas}\label{sec-app-proof-technical-lem}

We prove technical lemmas in this appendix.

\subsection{Proof of Lemma \ref{lem-3-24-2}}\label{appendix-1}

We need the following result.

\begin{lem}\label{lem-3-24-1}
Assume {\bf (H1)}. For each $i\in \{1,\dots, d\}$, there exists $C_i>0$ such that
\begin{equation*}
\lim_{x_i\to 0} x_i^2\left[q^2_i(x_i)-q'_i(x_i)\right]=C_i.
\end{equation*}
\end{lem}
\begin{proof}
Fix $ i\in\{1,\dots,d\}$. Recall that $q_i(x_i)=\frac{a'_i(\xi^{-1}_i(x_i))}{4\sqrt{a_i(\xi^{-1}_i(x_i))}}$. It is straightforward to calculate
$$
q'_i(x_i)=\frac{1}{4}a''_i(\xi^{-1}_i(x_i))-\frac{|a'_i|^2(\xi^{-1}_i(x_i))}{8a_i(\xi^{-1}_i(x_i))},
$$
which results in
\begin{equation}\label{formula-q2-qderivative}
(q^2_i-q'_i)(x_i)=\frac{3|a'_i|^2(\xi^{-1}_i(x_i))}{16 a_i(\xi^{-1}_i(x_i))}-\frac{1}{4}a''_i(\xi^{-1}_i(x_i)).
\end{equation}

Since $\xi_i^{-1}\in C([0,\infty))$ and $\xi_{i}^{-1}(0)=0$, we see from {\bf (H1)} that
$\lim_{x_i\to 0} a'_i(\xi^{-1}_i(x_i))=a_{i}'(0)>0$ and $\lim_{x_i\to 0} a''_i(\xi^{-1}_i(x_i))=a''_{i}(0)$. Hence,
\begin{equation*}\label{eqn-3-24-1}
\left(q^2_i-q'_i\right)(x_i)\sim \frac{3|a'_i|^2(0)}{16 a_i(\xi^{-1}_i(x_i))}-\frac{1}{4}a''_{i}(0)\quad \text{as}\quad x_i\to 0.
\end{equation*}
The conclusion follows if there is $C>0$ such that
\begin{equation}\label{eqn-5-9-1}
a_i(\xi^{-1}_i(x_i))\sim C x^2_i\quad \text{as}\quad x_i\to 0.
\end{equation}

We show that \eqref{eqn-5-9-1} holds with $C=\frac{|a'_i(0)|^2}{4}$. The assumption {\bf (H1)} and Taylor's expansion give
\begin{equation}\label{eqn-5-7-1}
a_i(z_i)\sim a'_{i}(0)z_i+o(z^2_i)\quad \text{as}\quad x_i\to 0,
\end{equation}
leading to
\begin{equation*}
\xi_i(z_i)=\int_0^{z_i}\frac{\md s}{\sqrt{a_i(s)}}=\int_0^{z_i}\frac{\md s}{\sqrt{a'_i(0)s+o(s^2)}}\sim \frac{2\sqrt{z_i}}{\sqrt{a'_i(0)}}\quad \text{as}\quad z_i\to 0.
\end{equation*}
Thus, $\xi^{-1}_i(x_i)\sim \frac{a'_{i}(0)x_i^2}{4}$ as $x_i\to 0$. Inserting this into \eqref{eqn-5-7-1} yields \eqref{eqn-5-9-1} with $C=\frac{|a'_i(0)|^2}{4}$.
This completes the proof.
\end{proof}

\begin{rem}\label{rem-about-Q-near-Ga}
Thanks to \eqref{eqn-5-9-1}, it is straightforward to check from the definition of $Q$ given in \eqref{eqn-function-Q} that $Q(x)$ behaves like $\sum_{i=1}^{d}\ln x_{i}$ as $x$ approaches to $\Ga$. Hence, $e^{-\frac{Q}{2}}$ is as singular as $\prod_{i=1}^{d}\frac{1}{\sqrt{x_{i}}}$ near $\Ga$.
\end{rem}

\begin{proof}[Proof of Lemma \ref{lem-3-24-2}]
We first prove (1). Recall that $U$ is given in \eqref{definition-U}. Clearly,
$$
\pa_{x_i} U(x)=\pa_{z_i} V(\xi^{-1}(x))\sqrt{a_i(\xi_i^{-1}(x_i))},\quad\forall x\in\UU.
$$
We derive from {\bf (H3)} (4) the existence of $C_1>0$ and $R_1>0$ such that
$$
\left(|\nabla U|^2+|p|^2\right)(x)\leq -C_1(b\cdot\nabla_{z} V)(\xi^{-1}(x))\leq C_1\al(x),\quad \forall x\in \UU\sm B_{R_1}^+.
$$

 Since $\sup_{B^+_{R_1}} \left(|\nabla U|^2+|p|^2\right)<\infty$ due to {\bf (H2)} and {\bf (H3)}(1) and $\inf_{\UU}\al>0$, there must exist some $C_2>0$ such that
$(|\nabla U|^2+|p|^2)<C_2\al$ in $B^+_{R_1}$. Setting $C:=\min\{C_1,C_2\}$ yields the result.

\medskip
The rest of the proof is arranged as follows. In {\bf Step 1}, we analyze the asymptotic behaviors of terms in $e_{\be,N}$ near the boundary $\Ga$ and in the vicinity of infinity. Based on these, the asymptotic behaviors of $e_{\be,N}$ are derived in {\bf Step 2}. The proof of (2) and (3) are respectively given in {\bf Step 3} and {\bf Step 4}. Recall that $R_0$ and $\de_0$ are fixed in Subsection \ref{subsec-alpha} when defining $\al$.

\medskip

\paragraph{\bf Step 1} We analyze the asymptotic behaviors of terms in $e_{\be,N}$.

\begin{itemize}
\item For the term $p\cdot \nabla U$, we see from {\bf (H3)} (1) that
\begin{equation}\label{eqn-6-25-1}
(p\cdot \nabla U)(x)
=(b\cdot\nabla V)(\xi^{-1}(x))\to -\infty\quad \text{as}\quad |x|\to \infty.
\end{equation}

\item For the term $\frac{1}{2}\sum_{i=1}^d (q_i^2-q'_i)$, Lemma \ref{lem-3-24-1} ensures the existence of $\de_{*}\in(0,\de_0)$ and $C_3, C_4>0$ such that
\begin{equation}\label{eqn-8-23-1}
\frac{C_3}{x_i^2}\leq \frac{1}{2}(q_i^2-q'_i)(x_i)\leq \frac{C_4}{x_i^2},\quad \forall x_i\in (0,\de_*]\andd i\in\{1,\dots, d\}.
\end{equation}

Since {\bf (H1)} gives $\limsup_{s\to \infty}\left(\frac{|a'_i(s)|^2}{a_i(s)}+a''_i(s)\right)<\infty$, we find from \eqref{formula-q2-qderivative} and \eqref{eqn-6-25-1}  that for any $0<\ep_1\ll 1$, there exists $R_2=R_{2}(\ep_{1})>0$ such that
\begin{equation}\label{eqn-8-23-2}
\frac{1}{2}|q_i^2-q'_i|(x_i)\leq -\frac{\ep_1}{d} (p\cdot\nabla U)(x),\quad\forall x\in\{x\in \UU:x_i\in (R_2,\infty)\}\andd i\in\{1,\dots, d\}.
\end{equation}


\item  For the terms $\De U$, $p\cdot q$ and $\nabla \cdot p$, we calculate
\begin{equation*}
    \begin{split}
        \pa_{x_ix_i}^{2}U(x)&=\left[\pa^2_{z_iz_i}V(\xi^{-1}(x)) a_i(\xi^{-1}_i(x_i))+\frac{1}{2}\pa_{z_i} V(\xi^{-1}(x)) a'_i(\xi^{-1}_i(x_i))\right],\\
        p_i(x)q_i(x_{i})&=\frac{b_i(\xi^{-1}(x)) a'_i(\xi^{-1}_i(x_i))}{4a_i(\xi^{-1}_i(x_i))},\\
        \pa_{x_i} p_i(x)&=\pa_{z_i}b_i(\xi^{-1}(x)) -\frac{b_i(\xi^{-1}(x)) a'_i(\xi^{-1}_i(x_i))}{2a_i(\xi^{-1}_i(x_i))}.
    \end{split}
\end{equation*}

By {\bf (H1)}-{\bf (H3)}, we have $U\in C^2(\ol{\UU})$, and $p\cdot q,\nabla\cdot p\in C(\ol{\UU})$. Moreover, {\bf (H3)}(3) and \eqref{eqn-6-25-1} guarantee that for any $0<\ep_2\ll 1$, there exists $R_3=R_{3}(\ep_{2})>0$ such that
\begin{equation}\label{eqn-3-24-3}
|\De U|+|p\cdot q|+|\nabla\cdot p |\leq -\ep_2 p\cdot\nabla U\quad \text{in}\quad\UU\sm B^+_{R_3}.
\end{equation}

\item For the term $\frac{1}{2}|\nabla U|^2$, we find from $|\nabla U|^2(x)=\sum_{i=1}^d |\pa_{z_i}V|^2(\xi^{-1}(x)) a_i(\xi^{-1}_i(x_i))$, the assumption {\bf (H3)}(4) and \eqref{eqn-6-25-1} that there are $C_5>0$ and $R_4>0$ such that
\begin{equation}\label{eqn-3-24-4}
\frac{1}{2}|\nabla U|^2\leq -C_5(p\cdot\nabla U)\quad\text{in}\quad \UU\sm B^+_{R_4}.
\end{equation}
\end{itemize}

\medskip

\paragraph{\bf Step 2} We analyze the asymptotic behaviors of $e_{\be,N}$ near  $\Ga$ and in the vicinity of infinity.

Set $R_*:=\max\{R_0,R_2,R_3,R_4\}$ and $C_6:=\frac{1}{2}\max_{i}\max_{x_i\in[\de_*,R_*]}|q_i^2-q'_i|(x_i)$. It is obvious that $R_*$ and $C_6$ depend on $\ep_1$ and $\ep_2$, which are to be determined in the proof of (3). Since $\al$ is piecewise defined, we analyze $e_{\be,N}$ in four subdomains: $\Ga_{\de_*}\bigcap B_{R_*}^+$, $\Ga_{\de_*}\bigcap (\UU\sm B_{R_*}^+)$, $(\UU\sm \Ga_{\de_*})\bigcap B_{R_*}^+$ and $(\UU\sm \Ga_{\de_*})\bigcap (\UU\sm B_{R_*}^+)$ separately,  where we recall that $\Ga_{\de_{*}}:=\{x\in\UU: x_i\leq \de_{*} \text{ for some } i\in\{1,\dots, d\}\}$.

For notational simplicity, we set
$$
\Psi= \frac{\be}{2}|\De U|+\frac{\be^2}{2}|\nabla U|^2+\be |p\cdot\nabla U|+|p\cdot q|+|\nabla \cdot p|.
$$

\begin{enumerate}
\item [(a)] In $\Ga_{\de_*}\bigcap B_{R_*}^+$. We see from $U\in C^2(\ol{\UU})$ and $p\cdot\nabla U$, $p\cdot q, \nabla\cdot p\in C(\ol{\UU})$ that
$\max_{\Ga_{\de_*}\cap B_{R_*}^+}\Psi<\infty$. It follows from \eqref{eqn-8-23-1} that
\begin{equation*}
\begin{split}
|e_{\be,N}|&\leq \sum_{i=1}^d\left(\frac{C_4}{x_i^2}\mathbbm{1}_{(0,\de_*)}(x_i)+C_6\mathbbm{1}_{(\de_*,R_*)}(x_i)\right)+\max_{\Ga_{\de_*}\cap B_{R_*}^+}\Psi\\
&\leq C_4\sum_{i=1}^d\frac{1}{x_i^2}+dC_6+\max_{\Ga_{\de_*}\cap B_{R_*}^+}\Psi
\end{split}
\end{equation*}
and
\begin{equation*}
\begin{split}
e_{\be,N}&\geq C_3\sum_{i=1}^d\frac{1}{x_i^2}\mathbbm{1}_{(0,\de_*)}(x_i)-dC_6-\max_{\Ga_{\de_*}\cap B_{R_*}^+}\Psi\\
&\geq C_3\sum_{i=1}^d\max\left\{\frac{1}{x_i^2},1\right\}-d\left(\frac{C_3}{\de_*^2}+C_6\right)-\max_{\Ga_{\de_*}\cap B_{R_*}^+}\Psi.
\end{split}
\end{equation*}

\item [(b)] In $\Ga_{\de_*}\bigcap (\UU\sm B_{R_*}^+)$. It follows from \eqref{eqn-8-23-1},  \eqref{eqn-3-24-3} and \eqref{eqn-3-24-4} that
\begin{equation*}
    \begin{split}
        |e_{\be,N}|&\leq  C_4\sum_{i=1}^d\frac{1}{x_i^2}+dC_6-\left(\be+\ep_2(1+\frac{\be}{2})+C_5\be^2\right)p\cdot\nabla U,\\
        e_{\be,N}&\geq  C_3\sum_{i=1}^d\frac{1}{x_i^2}\mathbbm{1}_{(0,\de_*)}(x_i)-dC_6-\left(\be-\ep_2(1+\frac{\be}{2})-C_5\be^2\right)p\cdot\nabla U\\
        &\geq  C_3\sum_{i=1}^d\max\left\{\frac{1}{x_i^2},1\right\}-d\left(\frac{C_3}{\de_*^2}+C_6\right)-\left(\be-\ep_2(1+\frac{\be}{2})-C_5\be^2\right)p\cdot\nabla U.
    \end{split}
\end{equation*}

\item [(c)] In $(\UU\sm \Ga_{\de_*})\bigcap B_{R_*}^+$. There hold
\begin{equation*}
\begin{split}
|e_{\be,N}|&\leq \max_{(\UU\sm \Ga_{\de_*})\cap B_{R_*}^+}\left[\Psi+\frac{1}{2}\sum_{i=1}^d |q_i^2-q'_i|\right],\\
e_{\be,N}&\geq - \max_{(\UU\sm \Ga_{\de_*})\cap B_{R_*}^+}\left[\Psi+\frac{1}{2}\sum_{i=1}^d |q_i^2-q'_i|\right].
\end{split}
\end{equation*}

\item [(d)]  In $(\UU\sm \Ga_{\de_*})\bigcap (\UU\sm B_{R_*}^+)$. It follows from \eqref{eqn-8-23-2}, \eqref{eqn-3-24-3} and \eqref{eqn-3-24-4} that
\begin{equation*}
    \begin{split}
        |e_{\be,N}|&\leq dC_6 -\left(\be+\ep_1+\ep_2(1+\frac{\be}{2})+C_5\be^2\right)p\cdot\nabla U,\\
        e_{\be,N}&\geq -d C_6 -\left(\be-\ep_1-\ep_2(1+\frac{\be}{2})-C_5\be^2\right)p\cdot\nabla U.
    \end{split}
\end{equation*}
\end{enumerate}

\medskip

\paragraph{\bf Step 3} We prove (2). As $\al\geq\sum_{i=1}^d\max\left\{\frac{1}{x_i^2},1\right\}$ in $\Ga_{\de_*}$ and $\inf_{\UU}\al>0$, we deduce from {\bf Step 2} (a) the existence of $D_1(\be)>0$ such that $e_{\be,N}\leq D_1(\be)\al$ in $\Ga_{\de_*}\cap B_{R_*}^+$ for all $N\geq 1$.

Since $\inf_{\UU} \al >0$ and
\begin{equation*}
\al=\begin{cases}
\displaystyle\sum_{i=1}^d\max\left\{\frac{1}{x_i^2},1\right\}-p\cdot \nabla U&\quad\text{in}\quad \Ga_{\de_*}\bigcap (\UU\sm B_{R_*}^+),\\
-p\cdot \nabla U&\quad\text{in}\quad  (\UU\sm \Ga_{\de_*})\bigcap (\UU\sm B_{R_*}^+),
\end{cases}
\end{equation*}
we see from {\bf Step 2} (b) and (d) that there exists $D_2(\be)>0$ such that $|e_{\be,N}|\leq D_2(\be)\al$ in $\UU\sm B_{R_*}^+$ for all $N\geq 1$.

Thanks to $\inf_{\UU} \al >0$, it follows from {\bf Step 2} (c) the existence of $D_3(\be)>0$ such that $|e_{\be,N}|\leq D_3(\be)\al$ in $(\UU\sm \Ga_{\de_*})\cap B^+_{R_*}$ for all $N\geq 1$.

Setting $C(\be):=\max\{D_1(\be),D_2(\be),D_3(\be)\}$ yields (2).

\medskip

\paragraph{\bf Step 4} We show (3). Setting $\be_0:=\frac{1}{2C_5}$, $\ep_1:=\min\left\{1,\frac{1}{16C_5}\right\}$ and $\ep_2:=\min\left\{1,\frac{1}{2+8C_5}\right\}$, we deduce from {\bf Step 2} (b) and (d) that
\begin{equation}\label{eqn-3-30-2}
\begin{split}
e_{\be_0,N}&\geq C_3\sum_{i=1}^d\max\left\{\frac{1}{x_i^2},1\right\}-d\left(\frac{C_3}{\de_*^2}+C_6\right)-\left(\be_0-\ep_2(1+\frac{\be_0}{2})-C_5\be^2_0\right)p\cdot\nabla U\\
&\geq C_3\sum_{i=1}^d\max\left\{\frac{1}{x_i^2},1\right\}-\frac{p\cdot\nabla U}{8C_5}-d\left(\frac{C_3}{\de_*^2}+C_6\right)\\
&\geq \min\left\{C_3,\frac{1}{8C_5}\right\} \al-d\left(\frac{C_3}{\de_*^2}+C_6\right)\quad\text{in}\quad \Ga_{\de_*}\bigcap (\UU\sm B_{R_*}^+)
\end{split}
\end{equation}
and
\begin{equation}\label{eqn-3-30-3}
\begin{split}
e_{\be_0,N}&\geq -dC_6-\left(\be_0-\ep_1-\ep_2(1+\frac{\be_0}{2})-C_5\be^2_0\right)p\cdot\nabla U\\
&\geq  -\frac{p\cdot\nabla U}{16C_5}-dC_6\geq \frac{1}{16C_5}\al -dC_6 \quad\text{in}\quad (\UU\sm \Ga_{\de_*})\bigcap (\UU\sm B_{R_*}^+).
\end{split}
\end{equation}

Since $\al\leq \sum_{i=1}^d\max\left\{\frac{1}{x_i^2},1\right\}+\max_{\Ga_{\de_*}\bigcap B_{R_*}^+} |p\cdot \nabla U|$ in $\Ga_{\de_*}\bigcap B_{R_*}^+$
and $\sup_{(\UU\sm \Ga_{\de_*})\bigcap B_{R_*}^+}\al <\infty$,
we conclude from (a) and (c)  the existence of positive constants $C_7$ and $M>d\left(\frac{C_3}{\de_*^2}+C_6\right)$ such that
$$
e_{\be_0,N}+M\geq C_7\al \quad \text{in}\quad B_{R_*}^+,\quad\forall N\geq 1,
$$
which together with \eqref{eqn-3-30-2} and \eqref{eqn-3-30-3} implies that
$$
e_{\be_0,N}+M\geq C_*\al\quad\text{in}\quad\UU,\quad\forall N\geq1,
$$
where $C_*:=\min\{C_3,\frac{1}{16C_5}, C_7\}$. This proves (3), and completes the proof.


\subsection{Proof of Lemma \ref{lem-4-21-1}}\label{subsec-appendix-2}

Suppose $\tilde{w}\in C(\UU\times[0,\infty))\bigcap L^2([0,\infty),\HH^{1}(\UU))$ is a weak solution of \eqref{eqn-4-16-2-1}. The proof is broken into two steps.

\medskip

\paragraph{\bf Step 1} We show
\begin{equation}\label{eqn-5-22-8}
\frac{1}{2}\int_{\UU}\tilde{w}^2(\cdot, t)\md x+\frac{1}{2}\int_0^t\int_{\UU}|\nabla \tilde{w}|^2\md x\md s+\int_0^t\int_{\UU}e_{\be_0,2} \tilde{w}^2\md x\md s=\frac{1}{2}\int_{\UU}\tilde{f}^2\md x, \quad\forall t\in [0,\infty).
\end{equation}

The idea of proving \eqref{eqn-5-22-8} is based on the classical ``energy method". But, we have to deal with the fact that $\tilde{w}$ lacks the differentiability in $t$. For each $0<h\ll1$, we define
$$
\tilde{w}_h(x,t):=\frac{1}{h}\int_t^{t+h}\tilde{w}(x,s)\md s,\quad(x,t)\in\UU\times[0,\infty).
$$
Obviously, $\tilde{w}_h\in C(\UU\times[0,\infty))\bigcap L^2([0,\infty),\HH^{1}(\UU))$ and $\pa_t\tilde{w}_h\in L^2(\UU\times[0,T])$ for each $T>0$. It is easy to verify that $\tilde{w}_h$ is a weak solution of \eqref{eqn-4-16-2-1} with $\tilde{f}$ replaced by $\tilde{f}_h:=\tilde{w}_h(\cdot, 0)=\frac{1}{h}\int_0^h\tilde{w}_h(\cdot,s)\md s$. Namely, for each $\phi\in C_0^{1,1}(\UU\times[0,\infty))$, one has
\begin{equation}\label{eqn-5-17-1}
\begin{split}
    &\int_{\UU}\tilde{w}_h(\cdot,t)\phi(\cdot,t)\md x-\int_{\UU}\tilde{f}_h\phi(\cdot,0)\md x-\int_{0}^t\int_{\UU}\tilde{w}_h\pa_t \phi\md x\md s\\
    &\qquad=-\frac{1}{2}\int_{0}^{t}\int_{\UU}\nabla\tilde{w}_h\cdot\nabla \phi\md x\md s-\int_0^t\int_{\UU}(p_i+\be_0\pa_i U)\tilde{w}_h\pa_i\phi\md x\md s\\
    &\qquad\quad-\int_0^t\int_{\UU}e_{\be_0}\tilde{w}_h\phi\md x\md s,\quad \forall t\in[0,\infty).
\end{split}
\end{equation}

Let $\{\eta_n\}_{n\in\N}\subset C_0^{\infty}(\UU)$ be a sequence of  functions taking values in $[0,1]$ and satisfying
\begin{equation*}
\eta_n(x)=
\begin{cases}
1,& x\in\left(\UU\sm \Ga_{\frac{2}{n}} \right)\bigcap B^+_{\frac{n}{2}}, \\
0,& x\in \Ga_{\frac{1}{n}}\bigcup\left(\UU\sm B^+_{n} \right),
\end{cases}
\quad
\text{and}
\quad
|\nabla \eta_n(x)|\leq
\begin{cases}
2n,& x\in\Ga_{\frac{2}{n}}\sm \Ga_{\frac{1}{n}},\\
4,& x\in\left(\UU\sm \Ga_{\frac{2}{n}} \right)\bigcap \left( B^+_{n}\sm B^+_{\frac{n}{2}}\right).
\end{cases}
\end{equation*}
By standard approximation arguments, we deduce that \eqref{eqn-5-17-1} holds with $\phi$ replaced by $\eta_n^2\tilde{w}_h$ for each $n\in\N$ and $0<h\ll1$, namely,
\begin{equation*}
\begin{split}
&\int_{\UU}\tilde{w}^2_h(\cdot,t)\eta_n^2\md x-\int_{\UU}\tilde{f}_h\eta_n^2\tilde{w}_h(\cdot,0)\md x-\int_{0}^t\int_{\UU}\tilde{w}_h\pa_t(\eta_n^2\tilde{w}_{h})\md x\md s\\
&\qquad=-\frac{1}{2}\int_0^t\int_{\UU}\nabla \tilde{w}_h\cdot \nabla(\eta_n^2 \tilde{w}_h)\md x\md s-\int_0^t\int_{\UU}(p+\be_0\nabla U)\cdot\tilde{w}_h\nabla(\eta_n^2 \tilde{w}_h)\md x\md s\\
&\qquad\quad-\int_0^t\int_{\UU}e_{\be_0} \eta_n^2\tilde{w}_h^2\md x\md s,\quad \forall t\in[0,\infty).
\end{split}
\end{equation*}
Note that the left hand side of the above equality equals $\frac{1}{2}\int_{\UU}(\eta_n^2\tilde{w}_h^2)(\cdot, t)\md x-\frac{1}{2}\int_{\UU}\eta_n^2\tilde{f}_h^2\md x$. Thus, for each $t\in[0,\infty)$, $n\in\N$ and $0<h\ll1$,
\begin{equation}\label{identity-to-take-limit}
\begin{split}
&\frac{1}{2}\int_{\UU}(\eta_n^2\tilde{w}_h^2)(\cdot, t)\md x-\frac{1}{2}\int_{\UU}\eta_n^2\tilde{f}_h^2\md x\\
&\qquad=-\frac{1}{2}\int_0^t\int_{\UU}\nabla \tilde{w}_h\cdot \nabla(\eta_n^2 \tilde{w}_h)\md x\md s-\int_0^t\int_{\UU}(p_i+\be_0\pa_i U)\tilde{w}_h\pa_i(\eta_n^2 \tilde{w}_h)\md x\md s\\
&\qquad\quad-\int_0^t\int_{\UU}e_{\be_0} \eta_n^2\tilde{w}_h^2\md x\md s.
\end{split}
\end{equation}

We claim that passing to the limit $h\to 0$ in \eqref{identity-to-take-limit} yields that for each $t\in[0,\infty)$ and $n\in\N$,
\begin{equation}\label{eqn-5-22-7}
\begin{split}
&\frac{1}{2}\int_{\UU}(\eta_n^2\tilde{w}^2)(\cdot, t)\md x-\frac{1}{2}\int_{\UU}\eta_n^2\tilde{f}^2\md x\\
&\qquad=-\frac{1}{2}\int_0^t\int_{\UU}\nabla \tilde{w}\cdot \nabla(\eta_n^2 \tilde{w})\md x\md s-\int_0^t\int_{\UU}(p_i+\be_0\pa_i U)\tilde{w}\pa_i(\eta_n^2 \tilde{w})\md x\md s\\
&\qquad\quad-\int_0^t\int_{\UU}e_{\be_0} \eta_n^2\tilde{w}^2\md x\md s.
\end{split}
\end{equation}
Assuming \eqref{eqn-5-22-7}, we conclude \eqref{eqn-5-22-8} from letting $n\to \infty$ in \eqref{eqn-5-22-7} and arguments as in the proof of Lemma \ref{lem-3-25-1} (2).

It remains to justify \eqref{eqn-5-22-7}. Fix $t\in[0,\infty)$ and $n\in\N$. Note for each $0<h\ll1$, there hold
\begin{equation*}
\tilde{w}_h(\cdot,t)-\tilde{w}(\cdot,t)=\frac{1}{h}\int_t^{t+h}[\tilde{w}(\cdot, s)-\tilde{w}(\cdot, t)]\md s=\int_0^1[\tilde{w}(\cdot, t+hs)-\tilde{w}(\cdot, t)]\md s,
\end{equation*}
and
$$
\tilde{f}_h-\tilde{f}=\frac{1}{h}\int_0^h[\tilde{w}(\cdot,s)-\tilde{f}]\md s=\int_0^1[\tilde{w}(\cdot,hs)-\tilde{f}]\md s.
$$
Since $\tilde{w}\in C(\UU\times[0,\infty))$, we find for each compact set $K\subset\UU$,
\begin{equation}\label{eqn-5-22-2}
\sup_{K\times[0,t]}|\tilde{w}_h-\Tilde{w}|\to 0\quad \text{as}\quad h\to 0
\end{equation}
and
$$
\sup_{K}|\tilde{f}_h-\tilde{f}|\to 0\quad\text{as}\quad h\to 0.
$$
It follows that
\begin{equation}\label{eqn-5-22-3}
\lim_{h\to 0}\left[\frac{1}{2}\int_{\UU}(\eta_n^2\tilde{w}_h^2)(\cdot, t)\md x-\frac{1}{2}\int_{\UU}\eta_n^2\tilde{f}_h^2\md x\right]=\frac{1}{2}\int_{\UU}(\eta_n^2\tilde{w}^2)(\cdot, t)\md x-\frac{1}{2}\int_{\UU}\eta_n^2\tilde{f}^2\md x,
\end{equation}
and
\begin{equation}\label{eqn-5-22-4}
\lim_{h\to 0} \int_0^t\int_{\UU}e_{\be_0} \eta_n^2\tilde{w}_h^2\md x\md s=\int_0^t\int_{\UU}e_{\be_0} \eta_n^2\tilde{w}^2\md x\md s.
\end{equation}

Since
\begin{equation*}
\nabla\tilde{w}_h(\cdot,t)-\nabla\tilde{w}(\cdot,t)=\frac{1}{h}\int_t^{t+h}[\nabla\tilde{w}(\cdot, s)-\nabla\tilde{w}(\cdot, t)]\md s=\int_0^1[\nabla\tilde{w}(\cdot, t+hs)-\nabla\tilde{w}(\cdot, t)]\md s,
\end{equation*}
we apply H\"older's inequality and Fubini's theorem to find
\begin{equation}\label{eqn-5-22-11}
\begin{split}
\int_0^t\int_{\UU}|\nabla\tilde{w}_h-\nabla\tilde{w}|^2\md x\md t'&\leq\int_0^t\int_{\UU}\int_0^1|\nabla\tilde{w}(x, t'+hs)-\nabla\tilde{w}(x, t')|^2\md s\md x\md t'\\
&=\int_0^1\int_0^t\int_{\UU}|\nabla\tilde{w}(x, t'+hs)-\nabla\tilde{w}(x, t')|^2\md x\md t'\md s\\
&\leq \sup_{s\in [0,h]}\int_0^t\int_{\UU}|\nabla\tilde{w}(x, t'+s)-\nabla\tilde{w}(x, t')|^2\md x\md t'.
\end{split}
\end{equation}

Since $\nabla\tilde{w}\in L^2(\UU\times[0,2t])$ and $C_0(\UU\times[0,2t])$ is dense in $L^2(\UU\times[0,2t])$, for each $\ep>0$, we could find some $\Phi\in C_0(\UU\times[0,2t])$ such that
$\|\Phi-\nabla \tilde{w}\|_{L^2(\UU\times[0,2t])}<\ep$. Obviously, $\Phi$ is uniformly continuous on $\UU\times[0,2t]$, resulting in
$\sup_{s\in [0,h]}\int_0^t\int_{\UU}|\Phi(x, t'+s)-\Phi(x, t')|^2\md x\md t'\to 0$ as $h\to 0$. Therefore,
\begin{equation*}
    \begin{split}
       & \sup_{s\in [0,h]}\int_0^t\int_{\UU}|\nabla\tilde{w}(x, t'+s)-\nabla\tilde{w}(x, t')|^2\md x\md t'\\
        &\qquad=\sup_{s\in [0,h]}\int_0^t\int_{\UU}|\nabla\tilde{w}(x, t'+s)-\Phi(x,t'+s)|^2\md x\md t'\\
        &\qquad\quad+\sup_{s\in [0,h]}\int_0^t\int_{\UU}|\Phi(x, t'+s)-\Phi(x, t')|^2\md x\md t'\\
        &\qquad\quad+\sup_{s\in [0,h]}\int_0^t\int_{\UU}|\Phi(x, t')-\nabla \tilde{w}(x, t')|^2\md x\md t'\\
        &\qquad\leq 2\|\Phi-\nabla \tilde{w}\|^2_{L^2(\UU\times[0,2t])}+\sup_{s\in [0,h]}\int_0^t\int_{\UU}|\Phi(x, t'+s)-\Phi(x, t')|^2\md x\md t'\\
        &\qquad\leq 2\ep +\sup_{s\in [0,h]}\int_0^t\int_{\UU}|\Phi(x, t'+s)-\Phi(x, t')|^2\md x\md t'.
    \end{split}
\end{equation*}
Letting $h\to0$ in the above estimates, we find from the arbitrariness of $\ep>0$ and \eqref{eqn-5-22-11} that
\begin{equation}\label{eqn-5-22-1}
  \int_0^t\int_{\UU}|\nabla\tilde{w}_h-\nabla\tilde{w}|^2\md x\md s\to 0\quad \text{as}\quad h\to 0.
\end{equation}
Hence, for each $n\in\N$, one has $\lim_{h\to 0}\int_0^t\int_{\UU}\eta_n^2|\nabla\tilde{w}_h|^2\md x\md s=\int_0^t\int_{\UU}\eta_n^2|\nabla\tilde{w}|^2\md x\md s$.

Since
\begin{equation*}
    \begin{split}
        &\int_0^t\int_{\UU}\eta_n\tilde{w}_h\nabla \tilde{w}_h\cdot\nabla\eta_n\md x\md s-\int_0^t\int_{\UU}\eta_n\tilde{w}\nabla \tilde{w}\cdot\nabla\eta_n\md x\md s\\
        &\qquad=\int_0^t\int_{\UU}\eta_n(\tilde{w}_h-\tilde{w})\nabla \tilde{w}_h\cdot\nabla\eta_n\md x\md s+\int_0^t\int_{\UU}\eta_n\tilde{w}(\nabla \tilde{w}_h-\nabla\tilde{w})\cdot\nabla\eta_n\md x\md s,
    \end{split}
\end{equation*}
we apply H\"older's inequality to deduce from \eqref{eqn-5-22-2} and \eqref{eqn-5-22-1} that
$$
\lim_{h\to 0}\int_0^t\int_{\UU}\eta_n\tilde{w}_h\nabla \tilde{w}_h\cdot\nabla\eta_n\md x\md s=\int_0^t\int_{\UU}\eta_n\tilde{w}\nabla \tilde{w}\cdot\nabla\eta_n\md x\md s.
$$
Hence,
\begin{equation}\label{eqn-5-22-5}
    \begin{split}
        &\lim_{h\to 0}\int_0^t\int_{\UU}\nabla \tilde{w}_h\cdot \nabla(\eta_n^2 \tilde{w}_h)\md x\md s\\
        &\qquad=\lim_{h\to 0}\int_0^t\int_{\UU}\eta_n^2|\nabla\tilde{w}_h|^2\md x\md s+2\int_0^t\int_{\UU}\eta_n\tilde{w}_h\nabla \tilde{w}_h\cdot\nabla\eta_n\md x\md s\\
        &\qquad=\int_0^t\int_{\UU}\eta_n^2|\nabla\tilde{w}|^2\md x\md s+2\int_0^t\int_{\UU}\eta_n\tilde{w}\nabla \tilde{w}\cdot\nabla\eta_n\md x\md s\\
        &\qquad=\int_0^t\int_{\UU}\nabla \tilde{w}\cdot \nabla(\eta_n^2 \tilde{w})\md x\md s.
    \end{split}
\end{equation}

Similar arguments yield
\begin{equation}\label{eqn-5-22-6}
    \begin{split}
        \lim_{h\to 0}\int_0^t\int_{\UU}(p+\be_0\nabla U)\cdot\tilde{w}_h\nabla(\eta_n^2 \tilde{w}_h)\md x\md s=\int_0^t\int_{\UU}(p+\be_0\nabla U)\cdot\tilde{w}\nabla(\eta_n^2 \tilde{w})\md x\md s.
    \end{split}
\end{equation}

Consequently, letting $h\to0$ in \eqref{identity-to-take-limit}, we conclude \eqref{eqn-5-22-7} from \eqref{eqn-5-22-3}, \eqref{eqn-5-22-4}, \eqref{eqn-5-22-5} and \eqref{eqn-5-22-6}.

\medskip
\paragraph{\bf Step 2} We show that $\int_{\UU}\tilde{w}^2(\cdot,t)\md x\leq \frac{e^{2Mt}}{M}\int_{\UU}\Tilde{f}^2\md x$ for all $t\in[0,\infty)$. Hence, $\tilde{w}=0$ if $\tilde{f}=0$. This proves the lemma.

As $e_{\be_0,2}+M\geq 0$ by Lemma \ref{lem-3-24-2} (3), we derive from \eqref{eqn-5-22-8} that
\begin{equation}\label{eqn-5-21-1}
\frac{1}{2}\int_{\UU}\tilde{w}^2(\cdot, t)\md x\leq M\int_0^t\int_{\UU}\tilde{w}^2\md x\md s+\int_{\UU}\tilde{f}^2\md x,\quad\forall t\in[0,\infty).
\end{equation}
Setting $g(t)=\int_0^t\int_{\UU}\tilde{w}^2\md x\md s$ for $t\in [0,\infty)$, we arrive at $\frac{1}{2} g'(t)\leq Mg(t)+\int_{\UU}\tilde{f}^2\md x$ for all $t\in [0,\infty)$. The conclusion then follows from Gronwall's inequality.
\end{proof}

\bibliographystyle{amsplain}

\end{document}